\documentclass[preprint,3p]{elsarticle}
\usepackage[utf8x]{inputenc}
\usepackage{amsmath}
\usepackage{amsfonts}
\usepackage{latexsym}
\usepackage{amsthm}
\usepackage{xargs}
\usepackage{eucal}
\usepackage{graphicx}
\usepackage{subfigure}

\newtheorem{thm}{Theorem}[section]
\newtheorem{thm*}{Theorem}
\newtheorem{lem}[thm]{Lemma}
\newtheorem{prop}[thm]{Proposition}
\newtheorem{cor}[thm]{Corollary}
\newtheorem*{cor*}{Corollary}
\theoremstyle{definition}
\newtheorem{rmk}[thm]{Remark}
\newtheorem*{rmk*}{Remark}

\newcommand{\vphi}{\varphi}
\newcommand{\vxi}{\xi}
\newcommand{\vxib}{\bar{\xi}}

\newcommand{\vzeta}{\zeta}
\newcommand{\veta}{\eta}
\newcommand{\vetab}{\bar{\eta}}

\newcommand{\Vpsi}{\Psi}
\newcommand{\vpsi}{\psi}
\newcommand{\vpsib}{\overline{\psi}}

\newcommand{\het}{z_0}
\newcommand{\dist}{\kappa}
\newcommand{\coef}{d}
\newcommand{\param}{\sigma}
\newcommand{\kk}{\lambda}

\newcommandx{\f}[1][1=\delta\vs]{l(#1)}
\newcommandx{\g}[1][1=\delta\vs]{m(#1)}
\newcommand{\xz}{X_0}
\newcommand{\xu}{X_1}

\newcommand{\vu}{u}
\newcommand{\vv}{v}
\newcommand{\vw}{w}
\newcommand{\vs}{s}

\newcommand{\uns}{\textup{u}}
\newcommand{\sta}{\textup{s}}

\newcommand{\re}{\textup{Re}\,}
\newcommand{\im}{\textup{Im}\,}

\newcommandx{\dout}[4][1=\dist, 2={,}, 3=\beta]{D_{#1,#3}^{\textup{out}#2#4}}
\newcommandx{\doutb}[4][1=\dist, 2={,}, 3=\beta]{D_{\overline{#1},#3}^{\textup{out}#2#4}}
\newcommandx{\dmatch}[5][1=\dist, 2={,}, 3=\beta_1, 4=\beta_2]{D_{#1,#3,#4}^{\textup{mch}#2#5}}
\newcommandx{\dmatchs}[2][1={,}]{\CMcal{D}_{\dist,\beta_1,\beta_2}^{\textup{mch}#1#2}}

\newcommand{\din}[1]{\CMcal{D}_{\beta_0,\rho}^{\textup{in},#1}}
\newcommand{\dinu}[1]{D_{\beta_0,\rho}^{\textup{in},#1}}
\newcommand{\ein}{E_{\beta_0,\rho}}

\newcommand{\Ain}{\tilde{\CMcal{A}}}
\newcommand{\hin}{\tilde{h}}
\newcommand{\Rin}{\CMcal{R}}
\newcommand{\Bin}{\CMcal{B}}
\newcommand{\Gin}{\CMcal{G}}
\newcommand{\Rinasy}{\tilde{\CMcal{R}}}

\newcommandx{\bsout}[2][1={,}]{\CMcal{X}^{\textup{out}#1#2}}
\newcommandx{\bsin}[2][1=\nu]{\CMcal{X}^{\textup{in},#2}_{#1}}
\newcommandx{\bsmatch}[2][1={,}]{\CMcal{X}^{\textup{mch}#1#2}}
\newcommand{\bssplit}{\CMcal{X}^{\textup{spl}}}

\newcommandx{\doutbinf}[2][1={,}]{D_{\overline{\kappa},\beta,\infty}^{\textup{out}#1#2}}
\newcommandx{\doutbT}[2][1={,}]{D_{\overline{\kappa},\beta,T}^{\textup{out}#1#2}}
\newcommandx{\doutinf}[2][1={,}]{D_{\kappa,\beta,\infty}^{\textup{out}#1#2}}
\newcommandx{\doutT}[2][1={,}]{D_{\kappa,\beta,T}^{\textup{out}#1#2}}

\newcommand{\xb}{\bar{x}}
\newcommand{\yb}{\bar{y}}
\newcommand{\zb}{\bar{z}}
\newcommand{\tb}{\bar{t}}

\begin{document}

\begin{frontmatter}

\title{Exponentially small heteroclinic breakdown in the generic Hopf-zero singularity}
\author{I. Baldomá, O. Castejón and T.M. Seara}

\begin{abstract}
In this paper we prove the breakdown of an heteroclinic connection in the analytic versal unfoldings of the generic Hopf-Zero singularity in an open set of the parameter space. This heteroclinic orbit appears at any order if one performs the normal form around the origin, therefore it is a phenomenon ``beyond all orders''. In this paper we provide a formula for the distance between the corresponding stable and unstable one dimensional manifolds which is given by an exponentially small function in the perturbation parameter. Our result applies both for conservative and dissipative unfoldings.
\end{abstract}
\end{frontmatter}

\section{Introduction and main result}
The so-called Hopf-zero (or central) singularity consists in an analytic vector field $X^*:\mathbb{R}^3\rightarrow\mathbb{R}^3$, having the origin as a critical point, and such that the eigenvalues of the linear part at this point are $0$, $\pm i\alpha^*$, for some $\alpha^*\neq0$. Hence, after a linear change of variables, we can assume that the linear part of this vector field near the origin is:
$$DX^*(0,0,0)=\left(\begin{array}{ccc}0 & \alpha^* & 0\\ -\alpha^* & 0  &0\\ 0 & 0 & 0\end{array}\right).$$
In this paper, assuming generic conditions on $X^*$, we will study some heteroclinic phenomena which appear in versal analytic unfoldings of this singularity in an open region of the parameter space. Note that, in the linear setting, it is clear that this singularity can be met by a generic family of linear vector fields depending on at least two parameters. Thus, it has codimension two. However, since $DX^*(0,0,0)$ has zero trace, it is reasonable to study it in the context of conservative vector fields. In this case, the singularity can be met by a generic linear family depending on one parameter, and so it has codimension one.

Here, we will work in the general setting (that is, with two parameters), since the conservative one is just a particular case of it. Hence, we will study generic analytic families $X_{\mu,\nu}$ of vector fields on $\mathbb{R}^3$ depending on two parameters $(\mu,\nu)\in\mathbb{R}^2$, such that $X_{0,0}=X^*$, the vector field described above. Following \cite{Guc81} and \cite{GH90}, after some changes of variables we can write $X_{\mu,\nu}$ in its normal form of order two, namely:
\begin{eqnarray}\label{normalformcartabc}
 \frac{d\xb}{d\tb}&=&\xb\left(\beta_0\nu-\beta_1\zb\right)+\yb\left(\alpha^*+\alpha_1\nu+\alpha_2\mu+\alpha_3\zb\right)+O_3(\xb,\yb,\zb,\mu,\nu),\medskip\nonumber\\
\frac{d\yb}{d\tb}&=&-\xb\left(\alpha^*+\alpha_1\nu+\alpha_2\mu+\alpha_3\zb\right)+\yb\left(\beta_0\nu-\beta_1\zb\right)+O_3(\xb,\yb,\zb,\mu,\nu),\medskip\\
\frac{d\zb}{dt}&=&-\gamma_0\mu+\gamma_1\zb^2+\gamma_2(\xb^2+\yb^2)+\gamma_3\mu^2+\gamma_4\nu^2+\gamma_5\mu\nu+O_3(\xb,\yb,\zb,\mu,\nu).\nonumber
\end{eqnarray}

Note that the coefficients $\alpha_3$, $\beta_1$, $\gamma_1$ and $\gamma_2$ depend exclusively on the vector field $X^*$. We also observe that the conservative setting corresponds to taking $\nu=0$, $\gamma_1=\beta_1$ and imposing also that the higher order terms are divergence-free.

From now on, we will assume that $X^*$ satisfies the following generic conditions:
\begin{equation}\label{condicionsX0}
 \beta_1\neq0, \qquad \gamma_1\neq 0.
\end{equation}
Moreover, we will consider unfoldings satisfying the generic conditions:
$$\beta_0\neq0, \qquad \gamma_0\neq 0.$$
Depending on the other coefficients $\alpha_i$ and $\gamma_i$, one obtains different qualitative behaviors for the orbits of the vector field $X_{\mu,\nu}$. The different versal unfoldings have been widely studied in the past, see for example \cite{Guc81,GH90,Tak74,BV84,BaSe06}. However, if $(\mu,\nu)$ belongs to a particular open set of the parameter space, these unfoldings are still not completely understood. This set is defined by the following conditions:
\begin{equation}
 \gamma_0\gamma_1\mu>0,\qquad |\beta_0\nu|<|\beta_1|\sqrt{|\mu|}\label{condicionsparams}.
\end{equation}
In this paper we will study the unfoldings $X_{\mu,\nu}$ with the parameters belonging to the open set defined by \eqref{condicionsparams}. In fact, redefining the parameters $\mu$ and $\nu$ and the variable $\zb$, one can achieve:
\begin{equation}\label{condicionscompactes}
 \beta_0=\gamma_0=1,\qquad \beta_1>0,\qquad\gamma_1>0,
\end{equation}
and consequently the open set defined by \eqref{condicionsparams} is now:
\begin{equation}\label{condicionsparamsconcr}
 \mu>0,\qquad |\nu|<\beta_1\sqrt{\mu}.
\end{equation}

Moreover, dividing the variables $\xb,\yb$ and $\zb$ by $\sqrt{\gamma_1}$, multiplying $\tb$ by $\sqrt{\gamma_1}$, redefining the coefficients and denoting $\alpha_0=\alpha^*/\sqrt{\gamma_1}$, we can assume that $\gamma_1=1$, and therefore the system \eqref{normalformcartabc} becomes:
\begin{eqnarray}\label{sistema}
 \frac{d\xb}{d\tb}&=&\xb\left(\nu-\beta_1\zb\right)+\yb\left(\alpha_0+\alpha_1\nu+\alpha_2\mu+\alpha_3\zb\right)+O_3(\xb,\yb,\zb,\mu,\nu),\medskip\nonumber\\
\frac{d\yb}{d\tb}&=&-\xb\left(\alpha_0+\alpha_1\nu+\alpha_2\mu+\alpha_3\zb\right)+\yb\left(\nu-\beta_1\zb\right)+O_3(\xb,\yb,\zb,\mu,\nu),\medskip\\
\frac{d\zb}{dt}&=&-\mu+\zb^2+\gamma_2(\xb^2+\yb^2)+\gamma_3\mu^2+\gamma_4\nu^2+\gamma_5\mu\nu+O_3(\xb,\yb,\zb,\mu,\nu).\nonumber
\end{eqnarray}
We denote by $X_{\mu,\nu}^2$, usually called the normal form of second order, the vector field obtained considering the terms of \eqref{sistema} up to order two. Therefore, one has:
$$
X_{\mu,\nu}=X_{\mu,\nu}^2+F_{\mu,\nu}^2, \quad \mbox{where } \ F_{\mu,\nu}^2(\xb,\yb,\zb)=O_{3}(\xb,\yb,\zb,\mu,\nu).
$$

Similarly, doing the normal form procedure up to any finite order $n$, one has:
$$X_{\mu,\nu}=X_{\mu,\nu}^n+F_{\mu,\nu}^n\qquad\qquad n\geq2,$$
where $X_{\mu,\nu}^n(\xb,\yb,\zb)$ is a polynomial of degree $n$ and:
$$F_{\mu,\nu}^n(\xb,\yb,\zb)=O_{n+1}(\xb,\yb,\zb,\mu,\nu).$$
Moreover, one can show (see \cite{GH90}) that if $\mu$ and $\nu$ are small enough:
\begin{enumerate}
 \item $X_{\mu,\nu}^n$ has to critical points $\bar{S}_\pm^n(\mu,\nu)=(0,0,\zb_\pm^n(\mu,\nu))$, with:
$$\zb_\pm^n(\mu,\nu)=\pm\sqrt{\mu}+O((\mu^2+\nu^2)^{1/2}),$$
with eigenvalues:
\begin{eqnarray*}\lambda_1^\pm&=&\pm2\sqrt{\mu}+O((\mu^2+\nu^2)^{1/2}),\\
\lambda_2^\pm&=&\nu\mp\beta_1\sqrt{\mu}+i(\alpha_0\pm\alpha_3\sqrt{\mu})+O((\mu^2+\nu^2)^{1/2}),\\
\lambda_3^\pm&=&\overline{\lambda_2^\pm}
\end{eqnarray*}
Hence, $\bar S_\pm^n(\mu,\nu)$ are both of saddle-focus type, $\bar S_+^n(\mu,\nu)$ having a one-dimensional unstable manifold and a two-dimensional stable one, and $\bar S_-^n(\mu,\nu)$ having a one-dimensional stable manifold and a two-dimensional unstable one.
 \item The segment of the $\zb$-axis between $\bar{S}_+^n(\mu,\nu)$ and $\bar{S}_-^n(\mu,\nu)$ is a heteroclinic connection.
 \item If $\gamma_2>0$ there exists a curve $\Gamma_n$ in the $(\mu,\nu)$-plane of the form $\nu=m\sqrt{\mu}+O(\mu^{3/2})$, such that for $(\mu,\nu)\in\Gamma_n$ the two-dimensional invariant manifolds of the points $\bar{S}_\pm^n(\mu,\nu)$ are coincident. In the conservative setting (where $\nu=0$), the two-dimensional invariant manifolds of $\bar{S}_\pm^n(\mu)$ coincide for all values of $\mu$.
\end{enumerate}
Then, the whole vector field $X_{\mu,\nu}=X_{\mu,\nu}^n+F_{\mu,\nu}^n$ will have two critical points $\bar{S}_\pm(\mu,\nu)$ close to $\bar{S}_\pm^n(\mu,\nu)$, which will be also of saddle-focus type. However it is reasonable to expect that the heteroclinic connections will no longer persist. Moreover, for $(\mu,\nu)$ close to $\Gamma_n$, what one might expect is that this breakdown of the heteroclinic connections causes the birth of homoclinic orbits to the point $S_+(\mu,\nu)$ (or $S_-(\mu,\nu)$), giving rise to what is known as a \v{S}il'nikov bifurcation (see \cite{Shil65}, and also \cite{Shil67} for the analogous phenomenon for vector fields in $\mathbb{R}^4$).

The existence of such \v{S}il'nikov bifurcations for $\CMcal{C}^\infty$ unfoldings of the Hopf-zero singularity is studied in \cite{BV84}. In the first place, in that paper the authors show that, doing the normal form procedure up to order infinity and using Borel-Ritt theorem, one can write $X_{\mu,\nu}=X_{\mu,\nu}^\infty+F_{\mu,\nu}^\infty$, where $X_{\mu,\nu}^\infty$ has the same properties 1, 2 and 3 as the vector fields $X_{\mu,\nu}^n$ described above, and $F_{\mu,\nu}=F_{\mu,\nu}(x,y,z)$ is a flat function at $(x,y,z,\mu,\nu)=(0,0,0,0,0)$. Their main result is that, given a family $X_{\mu,\nu}^\infty$ there exist flat perturbations $p_{\mu,\nu}^\infty$ such that the family:
\begin{equation}
 X_{\mu,\nu}=X_{\mu,\nu}^\infty+p_{\mu,\nu}^\infty \label{familiaBV}
\end{equation}
possesses a sequence of \v{S}il'nikov bifurcations, occurring at parameter points $(\mu_l,\nu_l)$, $l\in\mathbb{N}$, which accumulate at $(\mu,\nu)=(0,0)$. Moreover, they prove that there is a dense subset of the unfoldings which do not have a \v{S}il'nikov bifurcation, but in the complement of this set this \v{S}il'nikov phenomenon occurs densely. Even if the authors give an existence theorem, they do not provide conditions to check if a concrete family $X_{\mu,\nu}$ possesses a \v{S}il'nikov bifurcation. Moreover, the fields of the family \eqref{familiaBV}, for which they prove the existence of such bifurcations, are  $\CMcal{C}^\infty$ but not analytic vector fields.

Our final goal will be to study real analytic unfoldings of the singularity $X^*$ and to provide specific and explicit conditions over the family $X_{\mu,\nu}$ that, under assumptions \eqref{condicionscompactes}, and when the parameters belong to the set defined by \eqref{condicionsparamsconcr}, ensure the existence of a \v{S}il'nikov bifurcation. We conjecture that a similar phenomenon as the one described in \cite{BV84} will happen for a generic analytic singularity $X^*$ and all unfoldings satisfying these assumptions.

However, before proving the existence of an homoclinic connection, one has to check that the heteroclinic connections are broken indeed. In this paper we give a generic and numerically checkable condition on $X^*$ which guarantees the breakdown of the one-dimensional heteroclinic connection for any universal analytic unfolding satisfying \eqref{condicionscompactes} and \eqref{condicionsparamsconcr}. This is just a first step towards proving the existence of \v{S}il'nikov bifurcations for universal analytic unfoldings of the Hopf-zero singularity.

The breakdown of this heteroclinic orbit has been proved, in the conservative setting, for the so-called regular case in \cite{BaSe06}. In this problem, the regular case consists in considering that the terms of order three in system \eqref{sistema} are all divisible by $\mu$. Under this assumption, the authors give an asymptotic formula of the splitting distance of the one-dimensional invariant manifolds when they meet the plane $z=0$, which is a suitable version of the Melnikov integral (see \cite{GH90,Mel63}). Moreover, this distance turns out to be exponentially small with respect to the perturbation parameter $\mu$. Note that, as we pointed out above, the breakdown of the heteroclinic orbit cannot be detected in the truncation of the normal form at any finite order and therefore, as it is usually called, it is a phenomenon \textit{beyond all orders}. Hence, the exponential smallness of the splitting distance is in fact what one expected.

Here we deal with a generic $X^*$ and universal unfoldings, and therefore with the singular case. We observe that this case is very relevant since the vector field $X^*$ and the unfoldings considered in the regular case are not generic. Indeed, on one hand, the fact that the terms of order three in system \eqref{sistema} must be all divisible by $\mu$ implies that the singularity $X^*$ cannot have any term of order three, which obviously is not a generic condition. On the other hand, it also implies that some coefficients of the Taylor expansion of the unfoldings $X_{\mu,\nu}$ must be equal to zero, and hence the result is not valid for generic unfoldings.

In this paper we give an asymptotic formula of the distance between the two one-dimensional invariant manifolds when they meet the plane $z=0$, for generic unfoldings and both in the dissipative and conservative settings. This distance is again exponentially small with respect to the parameter $\mu$. However, Melnikov theory is no longer valid, and one has to introduce some techniques that were not needed in \cite{BaSe06}, such as the study of the so-called \textit{inner equation} (for other examples of exponentially small phenomena where the prediction given by Melnikov theory is not true, see \cite{Gel97, Tre97, Sau01}). Moreover, from the asymptotic formula we obtain an explicit and checkable condition over the vector field $X^*$ (namely, that a given constant $C^*$ is not zero) which ensures that, for every member of the family $X_{\mu,\nu}$ satisfying \eqref{condicionscompactes} and \eqref{condicionsparamsconcr}, the one-dimensional invariant manifolds of $S_\pm(\mu,\nu)$ are not coincident. This constant $C^*$, which is usually called the Stokes constant (see \cite{St64,St02}), depends on the full jet of $X^*$ and therefore, up to now, it can only be computed numerically.

The main result of the paper is the following:

\begin{thm*}\label{maintheorem}
 Consider system \eqref{sistema}, with $\mu,\,\beta_1>0$ and $|\nu|<\beta_1\sqrt{\mu}$, which has two critical points $S_\pm(\mu,\nu)$ of saddle focus type. Then there exists a constant $C^*$, depending on the full jet of $X^*$, such that the distance $\bar d^{\uns,\sta}$ between the one-dimensional stable manifold of $S_-(\mu,\nu)$ and the one-dimensional unstable manifold of $S_+(\mu,\nu)$ when they meet the plane $\zb=0$ is given asymptotically by:
$$\bar d^{\uns,\sta}=\mu^{-\frac{\beta_1}{2}}e^{-\frac{\alpha_0\pi}{2\sqrt{\mu}}}e^{\frac{\pi}{2}(\alpha_0h_0-\frac{\alpha_1\nu}{\sqrt{\mu}}+\alpha_3)}\left(C^*+O\left(\frac{1}{\log(1/\sqrt{\mu})}\right)\right),$$
where $-h_0$ is the coefficient of $\zb^3$ in the third equation of system \eqref{sistema}.
\end{thm*}

\begin{rmk*}
In the conservative setting we have $\nu=0$ and $\beta_1=1$, and hence this distance is given by:
$$\bar d^{\uns,\sta}=\mu^{-1/2}e^{-\frac{\alpha_0\pi}{2\sqrt{\mu}}}e^{\frac{\pi}{2}(\alpha_0h_0+\alpha_3)}\left(C^*+O\left(\frac{1}{\log(1/\sqrt{\mu})}\right)\right).$$
\end{rmk*}

\begin{cor*}
 If $C^*\neq0$, the one-dimensional invariant manifolds of $S_+(\mu,\nu)$ and $S_-(\mu,\nu)$ do not intersect.
\end{cor*}

\section{Sketch of the proof}
The aim of this section is to give the main ideas of how Theorem \ref{maintheorem} is proved.
\subsection{Set-up and notation}
First of all we will rescale the variables and parameters so that the critical points are $O(1)$, and not $O(\sqrt{\mu})$ as we had in system \eqref{sistema}. We define the new parameters $\delta=\sqrt{\mu}$, $\param=\delta^{-1}\nu$, and the new variables $x=\delta^{-1}\xb$, $y=\delta^{-1}\yb$, $z=\delta^{-1}\zb$ and $t=\delta\tb$. Then, renaming the coefficients $b=\gamma_2$, $c=\alpha_3$ and $d=\beta_1$, system \eqref{sistema} becomes:
\begin{equation}\label{initsys}
 \begin{array}{rcl}
  \displaystyle\frac{dx}{dt}&=&\displaystyle x\left(\param-dz\right)+\left(\frac{\alpha(\delta\param)}{\delta}+cz\right)y+\delta^{-2}f(\delta x,\delta y, \delta z, \delta,\delta\param),\medskip\\
  \displaystyle\frac{dy}{dt}&=&\displaystyle-\left(\frac{\alpha(\delta\param)}{\delta}+cz\right)x+y\left(\param-dz\right)+\delta^{-2}g(\delta x,\delta y, \delta z, \delta,\delta\param),\medskip\\
  \displaystyle\frac{dz}{dt}&=&-1+b(x^2+y^2)+z^2+\delta^{-2}h(\delta x, \delta y, \delta z, \delta,\delta\param),
 \end{array}
\end{equation}
where $\coef>0$, $f$, $g$ and $h$ are real analytic functions of order three in all their variables, $\delta>0$ is a small parameter and $|\param|<\coef$. Moreover, $\alpha(\delta\param)$ is an analytic function such that $\alpha(0)=\alpha_0\neq0$ and $\alpha'(0)=\alpha_1$.

\begin{rmk}
Without loss of generality, we can assume that $\alpha_0$ and $c$ are both positive constants. In particular, for $\delta$ small enough, $\alpha(\delta\param)$ will be also positive.
\end{rmk}

\begin{rmk}
 From now on, in order to shorten the notation, we will not write explicitly the dependence of $\alpha$ with respect to $\delta\param$. That is, we will write $\alpha$ instead of $\alpha(\delta\param)$. In fact, $\alpha$ will be treated as a parameter independent of $\delta$ and $\param$, since for $\delta$ small enough:
$$0<K_1\leq\alpha(\delta\param)\leq K_2,$$
and both constants are independent of these two parameters.
\end{rmk}

Below we summarize some properties of the rescaled system \eqref{initsys}, which can be deduced similarly as in \cite{BaSe08}.
\begin{lem}
  For any value of $\delta>0$, the unperturbed system (system \eqref{initsys} with $f=g=h=0$) verifies:
\begin{enumerate}
 \item It possesses two hyperbolic fixed points $S_{\pm}^0=(0,0,\pm1)$ which are of saddle-focus type with eigenvalues $\param\mp\coef+|\frac{\alpha}{\delta}\pm c|i$,
$\param\pm\coef-|\frac{\alpha}{\delta}\pm c|i$, and $\pm2$.
\item The one-dimensional unstable manifold of $S_{+}^0$ and the one-dimensional stable manifold of $S_{-}^0$ coincide along the heteroclinic connection
$\{(0,0,z):\, -1<z<1\}.$ This heteroclinic orbit can be parameterized by
$$\Upsilon_0(t)=(0,0,\het(t))=(0,0,-\tanh t),$$
if we require $\Upsilon_0(0)=(0,0,0).$
\end{enumerate}
\end{lem}

\begin{lem}\label{lemaptscritics}
 If $\delta>0$ is small enough, system \eqref{initsys} has two fixed points $S_{\pm}(\delta,\param)$ of saddle-focus type $S_{\pm}(\delta,\param)=(x_\pm(\delta,\param), y_\pm(\delta,\param),z_\pm(\delta,\param)),$ with:
$$x_\pm(\delta,\param)=O(\delta^2,\delta^2\param^3)=O(\delta^2),\quad y_\pm(\delta,\param)=O(\delta^2,\delta^2\param^3)=O(\delta^2),\quad z_\pm(\delta,\param)=\pm1+O(\delta,\delta\param^3)=\pm1+O(\delta).$$
$S_+(\delta,\param)$ has a one-dimensional unstable manifold and a two-dimensional stable one. Conversely, $S_-(\delta,\param)$ has a one-dimensional stable manifold and a two-dimensional unstable one.

Moreover, there are no other fixed points of \eqref{initsys} in the closed ball $B(\delta^{-1/3}).$
\end{lem}

Next theorem, which we will prove in the following, is the version of Theorem \ref{maintheorem} in the new variables.
\begin{thm}\label{maintheoremrescaled}
Consider system \eqref{initsys}, with $\delta,\,\coef>0$ and $|\param|<\coef$. Then there exists a constant $C^*$, such that the distance $d^{\uns,\sta}$ between the one-dimensional stable manifold of $S_-(\delta,\param)$ and the one-dimensional unstable manifold of $S_+(\delta,\param)$, when they meet the plane $z=0$, is given asymptotically by:
$$d^{\uns,\sta}=\delta^{-(1+\beta_1)}e^{-\frac{\alpha_0\pi}{2\delta}}e^{\frac{\pi}{2}(\alpha_0h_0-\alpha_1\param+c)}\left(C^*+O\left(\frac{1}{\log(1/\delta)}\right)\right),$$
being $h_0=-\lim_{z\to 0} z^{-3} h(0,0,z,0,0)$.
\end{thm}

\begin{rmk}
The asymptotic formula provided in Theorem \ref{maintheoremrescaled} for the distance $d^{\uns,\sta}$ has the same qualitative behavior as the
one proved in \cite{BaSe06} in the conservative setting for the regular case. The main difference between both formulae is about the
constant $C^{*}$. While in Theorem \ref{maintheoremrescaled} this constant depends on the full jet of $f,g,h$ and (at the moment) can only be computed
numerically, in the
regular case $C^*$ is completely determined by means of the Borel transform of some adequate analytic functions depending on
$f(0,0,u,0), g(0,0,u,0)$.
\end{rmk}
 Before we proceed, we introduce some notation that we will use for the rest of the paper. On one hand, in $\mathbb{C}^n$ we will consider the norm $|.|$ as:
$$|(z_1,\dots,z_n)|=|z_1|+\dots+|z_n|,$$
where $|z|$ stands for the ordinary modulus of a complex number. On the other hand, $B(r_0)$ will stand for the open ball of any vector space centered at zero and of radius $r_0$. Moreover, we will write $B^n(r_0)$ to denote $B(r_0)\times\stackrel{n)}{\dots}\times B(r_0)$.

\subsection{Existence of complex parameterizations in the outer domains}
As it is usual in works where exponentially small phenomena must be detected, the first thing we have to do in order to prove Theorem \ref{maintheoremrescaled} is to provide parameterizations of the one-dimensional invariant manifolds of the critical points $S_\pm(\delta,\param)$. Moreover, these have to be defined in some complex domains that are close to the singularities of the heteroclinic connection of the unperturbed system.

However, first we will introduce some changes of variables that will simplify the proof. The first one consists in performing a change that keeps the corresponding critical point constant with respect to the parameters. For instance, to prove the existence of a complex parameterization of the unstable manifold of $S_+(\delta,\param)$ we perform the $O(\delta)$-close to the identity change $C_1^u$ defined by:
\begin{equation}\label{canviptcritic}
 (\tilde x, \tilde y, \tilde z)=C_1^\uns(x,y,z,\delta,\delta\param)=(x-x_+(\delta,\param),y-y_+(\delta,\param),z-z_+(\delta,\param)+1),
\end{equation}
obtaining a system of the form:
\begin{equation}\label{syseta3dinter}
 \begin{array}{rcl}
  \displaystyle\frac{d\tilde x}{dt}&=&\displaystyle \tilde x\left(\param-d\tilde z\right)+\left(\frac{\alpha(\delta\param)}{\delta}+c\tilde z\right)\tilde y+\delta^{-2}f^\uns(\delta \tilde x,\delta \tilde y, \delta \tilde z, \delta,\delta\param),\medskip\\
  \displaystyle\frac{d\tilde y}{dt}&=&\displaystyle-\left(\frac{\alpha(\delta\param)}{\delta}+c\tilde z\right)\tilde x+\tilde y\left(\param-d\tilde z\right)+\delta^{-2}g^\uns(\delta \tilde x,\delta \tilde y, \delta \tilde z, \delta,\delta\param),\medskip\\
  \displaystyle\frac{d\tilde z}{dt}&=&-1+b(\tilde x^2+\tilde y^2)+\tilde z^2+\delta^{-2}h^\uns(\delta \tilde x, \delta \tilde y, \delta \tilde z, \delta,\delta\param),
 \end{array}
\end{equation}
where $f^\uns(0,0,\delta,\delta,\delta\param)= g^\uns(0,0,\delta,\delta,\delta\param)= h^\uns(0,0,\delta,\delta,\delta\param)=0$ for all $\delta$, and hence has the critical point $S_+(\delta,\param)$ fixed at $(0,0,1)$. Moreover $f^\uns$, $g^\uns$ and $h^\uns$ are analytic and of order three in all their variables.

After that we do the change:
\begin{equation}\label{canviC2}(\veta,\vetab,\vv)=C_2(\tilde x, \tilde y, \tilde z)=(\tilde x+i\tilde y,\tilde x-i\tilde y,\het^{-1}(\tilde z)),\end{equation}
where $\het(t)=-\tanh t$ is the third component of the heteroclinic connection $\Upsilon_0(t)$ of the unperturbed system. Then we obtain a system of the form:
\begin{equation}\label{syseta3d}
 \begin{array}{rcl}
  \displaystyle\frac{d\veta}{dt}&=&\displaystyle-\left(\frac{\alpha }{\delta}+c\het(\vv)\right)i\veta+\veta\left(\param-\coef\het(\vv)\right)+\delta^{-2} F_1^\uns(\delta \veta,\delta \vetab, \delta \het(\vv), \delta,\delta\param),\medskip\\
  \displaystyle\frac{d\vetab}{dt}&=&\displaystyle\left(\frac{\alpha }{\delta}+c\het(\vv)\right)i\vetab+\vetab\left(\param-\coef\het(\vv)\right)+\delta^{-2} F_2^\uns(\delta \veta,\delta \vetab, \delta \het(\vv), \delta,\delta\param),\medskip\\
  \displaystyle\frac{d\vv}{dt}&=&\displaystyle 1+\frac{b\veta\vetab+\delta^{-2} H^\uns(\delta \veta,\delta \vetab, \delta \het(\vv), \delta,\delta\param)}{-1+\het^2(\vv)},
 \end{array}
\end{equation}
where, again, $F_1^\uns(0,0,\delta,\delta,\delta\param)= F_2^\uns(0,0,\delta,\delta,\delta\param)= H^\uns(0,0,\delta,\delta,\delta\param)=0$ for all $\delta$ and are of order three, since:
\begin{eqnarray}\label{defF1uF2uHu}
 F_1^\uns(\delta\veta,\delta\vetab,\het(\vv),\delta,\delta\param)&=&f^\uns\left(\frac{\delta(\veta+\vetab)}{2},\frac{\delta(\veta-\vetab)}{2},\het(\vv),\delta,\delta\param\right)+ig^\uns\left(\frac{\delta(\veta+\vetab)}{2},\frac{\delta(\veta-\vetab)}{2},\het(\vv),\delta,\delta\param\right),\nonumber\\
 F_2^\uns(\delta\veta,\delta\vetab,\het(\vv),\delta,\delta\param)&=&f^\uns\left(\frac{\delta(\veta+\vetab)}{2},\frac{\delta(\veta-\vetab)}{2},\het(\vv),\delta,\delta\param\right)-ig^\uns\left(\frac{\delta(\veta+\vetab)}{2},\frac{\delta(\veta-\vetab)}{2},\het(\vv),\delta,\delta\param\right),\nonumber\\
 H^\uns(\delta\veta,\delta\vetab,\het(\vv),\delta,\delta\param)&=&h^\uns\left(\frac{\delta(\veta+\vetab)}{2},\frac{\delta(\veta-\vetab)}{2},\het(\vv),\delta,\delta\param\right).
\end{eqnarray}


To prove the existence of the stable manifold of $S_-(\delta,\param)$, instead of the change $C_1^\uns$ defined in \eqref{canviptcritic}, we do the change:
 $$(\tilde x, \tilde y, \tilde z)=C_1^\sta(x,y,z,\delta,\delta\param)=(x-x_-(\delta,\param),y-y_-(\delta,\param),z-z_-(\delta,\param)+1),$$
and after that we do the change $C_2$. Then we obtain a system analogous to \eqref{syseta2d}, where instead of $F_i^\uns$ and $H^\uns$ we have functions $F_i^\sta$, $H^\sta$ such that $F_1^\sta(0,0,-\delta,\delta,\delta\param)= F_2^\sta(0,0,-\delta,\delta,\delta\param)= H^\sta(0,0,-\delta,\delta,\delta\param)=0$ for all $\delta$.

We will denote:
\begin{equation}\label{defvetaz+-}
 \veta_\pm=\veta_\pm(\delta,\param)=x_\pm(\delta,\param)+iy_\pm(\delta,\param),\qquad \vetab_\pm=\vetab_\pm(\delta,\param)=\overline{\veta_\pm(\delta,\param)},\qquad z_\pm=z_\pm(\delta,\param).
\end{equation}

\begin{rmk}\label{ordreF12H}
 Note that, as $f$, $g$ and $h$ are analytic functions, and since:
$$\delta\veta_\pm,\delta\vetab_\pm=O(\delta^3),\quad\delta(z_\pm\mp1)=O(\delta^2),$$
there exist some $r_0^{\uns,\sta}$, independent of $\delta$, such that for $\delta$ small enough $F_1^{\uns,\sta}$, $F_2^{\uns,\sta}$ and $H^{\uns,\sta}$ are analytic in $(\delta\veta,\delta\vetab, \delta z,\delta,\delta\param)\in B^3(r_0^{\uns,\sta})\times B(\delta_0)\times B(\param_0)$ respectively. Moreover, using that they are of order three, it is easy to see that if $\phi\in B^3(r_0^{\uns,\sta})\times B(\delta_0)\times B(\param_0)$, then:
\begin{equation*}
|F_1^{\uns,\sta}(\phi)|,|F_2^{\uns,\sta}(\phi)|,|H^{\uns,\sta}(\phi)|\leq K|(\phi_1,\phi_2, \phi_3\mp\delta,\phi_4,\phi_5)|^3,
\end{equation*}
respectively.
\end{rmk}

Finally, thinking of $\veta$ and $\vetab$ as functions of $\vv$ we get the following systems, respectively in the unstable and stable case:
\begin{equation}\label{syseta2d}
 \begin{array}{rcl}
  \displaystyle\frac{d\veta}{d\vv}&=&\displaystyle\frac{\displaystyle-\left(\frac{\alpha }{\delta}+c\het(\vv)\right)i\veta+\veta\left(\param-\coef\het(\vv)\right)+\delta^{-2} F_1^{\uns,\sta}(\delta \veta,\delta \vetab, \delta \het(\vv), \delta,\delta\param)}{\displaystyle 1+\frac{b\veta\vetab+\delta^{-2} H^{\uns,\sta}(\delta \veta,\delta \vetab, \delta \het(\vv), \delta,\delta\param)}{-1+\het^2(\vv)}},\bigskip\\
  \displaystyle\frac{d\vetab}{d\vv}&=&\displaystyle\frac{\displaystyle\left(\frac{\alpha }{\delta}+c\het(\vv)\right)i\vetab+\vetab\left(\param-\coef\het(\vv)\right)+\delta^{-2} F_2^{\uns,\sta}(\delta \veta,\delta \vetab, \delta \het(\vv), \delta,\delta\param)}{\displaystyle 1+\frac{b\veta\vetab+\delta^{-2} H^{\uns,\sta}(\delta \veta,\delta \vetab, \delta \het(\vv), \delta,\delta\param)}{-1+\het^2(\vv)}}.
 \end{array}
\end{equation}
We will look for solutions of system \eqref{syseta2d} $\vzeta^{\uns,\sta}(\vv)=(\veta^{\uns,\sta}(\vv),\vetab^{\uns,\sta}(\vv))$ such that:
\begin{equation}\label{cond}\lim_{\vv\to-\infty}\vzeta^\uns(\vv)=(0,0),\qquad\lim_{\vv\to+\infty}\vzeta^\sta(\vv)=(0,0).\end{equation}
After Theorem \ref{thmouterunbounded} we will justify that, indeed, $(\veta^{\uns,\sta}(\vv),\vetab^{\uns,\sta}(\vv),\het(\vv))$ lead to parameterizations of the unstable and stable manifolds of the critical points $(0,0,\pm1)$ of system \eqref{syseta3dinter}, respectively.

Once we have obtained a suitable system \eqref{syseta2d}, the next step is to prove the existence of solutions verifying \eqref{cond}. The main idea is that system \eqref{syseta2d} has a linear part which is dominant. More concretely, we denote $\vzeta=(\veta,\vetab)^T$, $F^{\uns,\sta}=(F_1^{\uns,\sta},F_2^{\uns,\sta})^T$, and we define:
$$A(\vv)=\left(\begin{array}{cc}
                  \displaystyle-\left(\frac{\alpha }{\delta}+c\het(\vv)\right)i+\param-\coef\het(\vv) & 0\medskip\\
		  0 & \displaystyle\left(\frac{\alpha }{\delta}+c\het(\vv)\right)i+\param-\coef\het(\vv)
                 \end{array}
\right),$$
and:
\begin{equation}\label{defR}
 R^{\uns,\sta}(\vzeta)(\vv)=\left(\frac{\displaystyle 1}{\displaystyle1+\frac{\displaystyle b\veta\vetab+\delta^{-2}H^{\uns,\sta}(\delta \vzeta, \delta \het(\vu), \delta,\delta\param)}{\displaystyle-1+\het^2(\vu)}}-1\right)A(\vv)\vzeta+\frac{\displaystyle\delta^{-2}F^{\uns,\sta}(\delta \vzeta,\delta\het(\vv),\delta,\delta\param)}{\displaystyle1+\frac{\displaystyle b\veta\vetab+\delta^{-2}H^{\uns,\sta}(\delta \vzeta, \delta \het(\vu), \delta,\delta\param)}{\displaystyle-1+\het^2(\vu)}}.
\end{equation}
Then, in the stable case, system \eqref{syseta2d} joint with \eqref{cond} can be rewritten as:
\begin{equation}\label{syscompu}
\frac{d\vzeta}{d\vv}=A(\vv)\vzeta+R^\uns(\vzeta)(\vv),\qquad\lim_{\vv\to-\infty}\vzeta^\uns(\vv)=(0,0)
\end{equation}
and the corresponding for the stable one as:
\begin{equation}\label{syscomps}
\frac{d\vzeta}{d\vv}=A(\vv)\vzeta+R^\sta(\vzeta)(\vv),\qquad\lim_{\vv\to+\infty}\vzeta^\sta(\vv)=(0,0).
\end{equation}

As we mentioned above, we will need to find parameterizations of the invariant manifolds defined not just for $\vv\in\mathbb{R}$, but in some complex domains that are close to the first singularities of the heteroclinic connection $\Upsilon_0$ of the unperturbed system, which in this case are $\vv=\pm i\pi/2$. We will now proceed to introduce these complex domains. We define (see Figure \ref{figuradoutu}):
\begin{equation*}
\doutb{\uns}=\left\{\vv\in\mathbb{C} \,:\, |\im\vv|\leq \pi/2-\overline{\dist}\delta\log(1/\delta)-\tan\beta\re\vv\right\},
\end{equation*}
where $0<\beta<\pi/2$, $T>0$ and $\overline{\dist}>0$ are constants independent of $\delta$.
\begin{figure}
 \center
 \includegraphics[width=7.5cm]{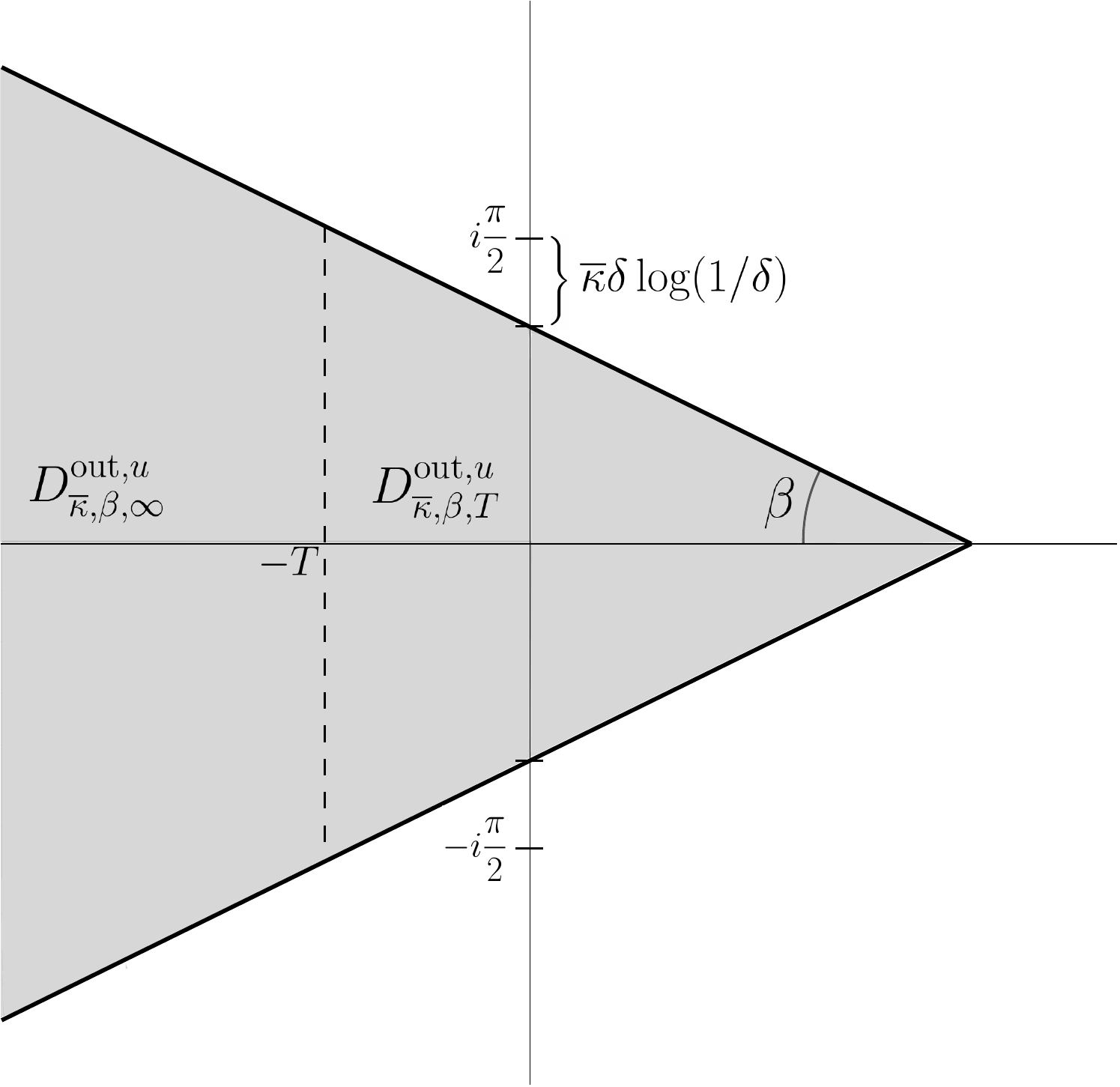}
 \caption{The outer domain $\doutb{\uns}$ for the unstable case with its subdomains $\doutbT{\uns}$ and $\doutbinf{\uns}$.}
 \label{figuradoutu}
\end{figure}
For technical reasons we will split the domain $\doutb{\uns}$ in two subsets, namely:
\begin{equation}\label{defDTDinf}
  \doutbinf{\uns}=\left\{\vv\in \doutb{\uns}\, : \, \re\vv\leq-T\right\},\qquad  \doutbT{\uns}=\left\{\vv\in \doutb{\uns}\, : \, \re\vv\geq-T\right\}.
\end{equation}
Analogously, we define:
$$\doutb{\sta}=-\doutb{\uns},\qquad \doutbinf{\sta}=-\doutbinf{\uns},\qquad \doutbT{\sta}=-\doutbT{\uns}.$$

\begin{thm}\label{thmouterunbounded}
Let $\overline{\dist}>0$ and $0<\beta<\pi/2$ be any fixed constants independent of $\delta$. Then, if $\delta>0$ is small enough, problem \eqref{syscompu} has a solution $\vzeta^\uns(\vv)=(\veta^\uns(\vv),\vetab^\uns(\vv))$ defined for $\vv\in\doutb{\uns}$, and \eqref{syscomps} has a solution $\vzeta^\sta(\vv)=(\veta^\sta(\vv),\vetab^\sta(\vv))$ defined for $\vv\in\doutb{\sta}$. Moreover there exists a constant $K$ independent of $\delta$ such that:
$$|\vzeta^{\uns,\sta}(\vv)|\leq\left\{\begin{array}{cl} K\delta^2|\het(\vv)-1| & \qquad\text{if } \vv\in \doutbinf{*},\medskip\\ K\delta^2|\het(\vv)-1|^3 & \qquad\text{if } \vv\in \doutbT{*},\end{array}\right.$$
where $*=\uns,\sta$ respectively.
\end{thm}
The proof of this result is postponed to Section \ref{prooftheoremouterunbounded}. Now we enunciate the following corollary:

\begin{cor}
Let $\overline{\kappa}$ and $0<\beta<\pi/2$ be two fixed constants independent of $\delta$. Consider $\veta^\uns$ and $\vetab^\uns$ the functions given by Theorem \ref{thmouterunbounded}, and let $\vv(t)$ be the solution of:
\begin{equation}\label{equaciov} \frac{d\vv}{dt}=\displaystyle 1+\frac{b\veta^\uns(\vv)\vetab^\uns(\vv)+\delta^{-2} H^\uns(\delta \veta^\uns(\vv),\delta \vetab^\uns(\vv), \delta \het(\vv), \delta,\delta\param)}{-1+\het^2(\vv)}=:1+\mathcal{F}(\vv),\end{equation}
such that $\vv(0)=0$. Then, $(\tilde x^{\uns}(t),\tilde y^{\uns}(t),\tilde z^{\uns}(t))$ defined by:
$$\tilde x^{\uns}(t)=\frac{\veta^{\uns}(\vv(t))+\vetab^\uns(\vv(t))}{2},\qquad \tilde y^{\uns}(t)=\frac{\veta^{\uns}(\vv(t))-\vetab^\uns(\vv(t))}{2}, \qquad \tilde z^\uns=\het(\vv(t)),$$
is a parameterization of the unstable manifold of the critical point $(0,0,1)$ of  system \eqref{syseta3dinter}. For the stable manifold of $(0,0,-1)$, one has an analogous result.
\end{cor}

\begin{proof}
Indeed, it is clear that $(\tilde x^\uns(t),\tilde y^\uns(t),\tilde z^\uns(t))$ is a solution of system \eqref{syseta3dinter}, since it consists in performing the inverse change of $C_2$, defined in \eqref{canviC2}, for a particular solution of system \eqref{syseta3d}. Hence, we just have to check that:
$$\lim_{t\to-\infty}(\tilde x^\uns(t),\tilde y^\uns(t),\tilde z^\uns(t))=(0,0,1).$$
Note that it is sufficient to prove that:
\begin{equation}\label{limitv}\lim_{t\to-\infty}\vv(t)=-\infty,\end{equation}
since, on the one hand $\het(\vv)=-\tanh(\vv)$ goes to 1 as $\vv$ goes to $-\infty$ and, on the other hand, from Theorem \ref{thmouterunbounded} we know that:
$$\lim_{\vv\to-\infty}(\veta^\uns(\vv),\vetab^\uns(\vv))=(0,0).$$

We will prove that \eqref{limitv} holds if $\vv(0)=0$ as follows. Indeed, from \eqref{equaciov} it is clear that:
$$t=\int_0^v\frac{1}{1+\mathcal{F}(\vw)}d\vw:=\CMcal{G}(\vv).$$
Now, from Theorem \ref{thmouterunbounded} and the fact that $|\het(\vv)-1|$ is bounded for $\vv\in\doutbT{u}\cap\mathbb{R}$, it is clear that for $\vv\in \doutb{u}\cap\mathbb{R}$:
$$|\veta^\uns(\vv)|,|\vetab^\uns(\vv)|\leq K\delta^2|\het(\vv)-1|,$$
for some constant $K$. Using these bounds, that $e^\vv$ is bounded for $\vv\in\doutb{u}\cap\mathbb{R}$ and Remark \ref{ordreF12H}, it can be easily seen that:
\begin{eqnarray*}
 |\mathcal{F}(\vv)|&=&\left|\frac{b\veta^\uns(\vv)\vetab^\uns(\vv)+\delta^{-2} H^\uns(\delta \veta^\uns(\vv),\delta \vetab^\uns(\vv), \delta \het(\vv), \delta,\delta\param)}{-1+\het^2(\vv)}\right|\leq\tilde K\left(\frac{\delta^4|\het(\vv)-1|^2+\delta|\het(\vv)-1|^3}{|-1+\het^2(\vv)|}\right) \\
&\leq&\tilde K\left(\delta^4e^{2\vv}+\delta e^{4\vv}\right)<\frac{1}{2},
\end{eqnarray*}
if $\delta$ is small enough. Then it is clear that $\CMcal{G}'(\vv)=(1+\mathcal{F}(\vv))^{-1}$ satisfies:
\begin{equation}\label{fitesGcal}\CMcal{G}'(\vv)\geq\frac{1}{1+1/2}=\frac{2}{3}>0.\end{equation}
On one hand, the fact that $\CMcal{G}'(\vv)$ is strictly positive implies that $\CMcal{G}(\vv)$ is strictly increasing. Then $\CMcal{G}$ is invertible in $\doutb{u}\cap\mathbb{R}$, and for $\vv\in\doutb{u}\cap\mathbb{R}$ we can write:
\begin{equation}\label{Ginversa}\vv=\CMcal{G}^{-1}(t).\end{equation}
Note that, as $\CMcal{G}$ is strictly increasing, so is $\CMcal{G}^{-1}$, and then we have that $v(t)\leq v(0)=0$ for $t\leq0$ . Hence it is clear that $v(t)\in\doutb{u}\cap\mathbb{R}$ for all $t\leq0$, and hence \eqref{Ginversa} has sense for all $t\leq0$.
On the other hand, we also have that:
$$\left(\CMcal{G}^{-1}\right)'=\frac{1}{\CMcal{G}'}\leq \frac{3}{2},$$
which implies that:
$$\vv=\int_0^t\left(\CMcal{G}^{-1}(\vs)\right)'d\vs\leq\frac{3}{2}t,$$
and hence we immediately obtain \eqref{limitv}.
\end{proof}

\subsubsection{Local parameterizations of the invariant manifolds}\label{localparams}
Theorem \ref{thmouterunbounded} provides us with complex parameterizations of the invariant manifolds, $\vzeta^{\uns,\sta}=(\veta^{\uns,\sta},\vetab^{\uns,\sta})$, which are solutions of problems \eqref{syscompu} and \eqref{syscomps} respectively. However, in order to study their difference, it is very useful that both manifolds are given by functions that satisfy the same system in a common domain. We proceed to undo the changes $C_1^\uns$ for $\vzeta^\uns$ and $C_1^\sta$ for $\vzeta^\sta$.

Consider:
$$V_\pm(\vu,\delta,\param)=\het^{-1}(\het(\vu)-z_\pm(\delta,\param)\pm1)-\vu.$$
Let $(\veta^{\uns,\sta}(\vv),\vetab^{\uns,\sta}(\vv))$ be solutions of system \eqref{syseta2d} and:
\begin{equation}\label{eqrecover}
 \vxi^{\uns,\sta}(\vu)=\veta^{\uns,\sta}(\vu+V_\pm(\vu,\delta,\param))+\veta_\pm(\delta,\param),\qquad\vxib^{\uns,\sta}(\vu)=\vetab^{\uns,\sta}(\vu+V_\pm(\vu,\delta,\param))+\vetab_\pm(\delta,\param).
\end{equation}
Then wherever $V_\pm(\vu,\delta,\param)$ is defined we have that $(\vxi^{\uns,\sta},\vxib^{\uns,\sta})$ are solutions of the following system:
\begin{equation}\label{sys2d}
 \begin{array}{rcl}
  \displaystyle\frac{d\vxi}{d\vu}&=&\displaystyle\frac{\displaystyle-\left(\frac{\alpha(\delta\param)}{\delta}+c\het(\vu)\right)i\vxi+\vxi\left(\param-\coef\het(\vu)\right)+\delta^{-2}F_1(\delta \vxi,\delta \vxib, \delta \het(\vu), \delta,\delta\param)}{\displaystyle 1+\frac{b\vxi\vxib+\delta^{-2}H(\delta \vxi,\delta \vxib, \delta \het(\vu), \delta,\delta\param)}{-1+\het^2(\vu)}},\bigskip\\
  \displaystyle\frac{d\vxib}{d\vu}&=&\displaystyle\frac{\displaystyle\left(\frac{\alpha(\delta\param)}{\delta}+c\het(\vu)\right)i\vxib+\vxib\left(\param-\coef\het(\vu)\right)+\delta^{-2}F_2(\delta \vxi,\delta \vxib, \delta \het(\vu), \delta,\delta\param)}{\displaystyle 1+\frac{b\vxi\vxib+\delta^{-2}H(\delta \vxi,\delta \vxib, \delta \het(\vu), \delta,\delta\param)}{-1+\het^2(\vu)}},
 \end{array}
\end{equation}
where:
\begin{eqnarray}
 F_1(\delta\vxi,\delta\vxib,\het(\vv),\delta,\delta\param)&=&f\left(\frac{\delta(\vxi+\vxib)}{2},\frac{\delta(\vxi-\vxib)}{2},\het(\vv),\delta,\delta\param\right)+ig\left(\frac{\delta(\vxi+\vxib)}{2},\frac{\delta(\vxi-\vxib)}{2},\het(\vv),\delta,\delta\param\right),\nonumber\\
 F_2(\delta\vxi,\delta\vxib,\het(\vv),\delta,\delta\param)&=&f\left(\frac{\delta(\vxi+\vxib)}{2},\frac{\delta(\vxi-\vxib)}{2},\het(\vv),\delta,\delta\param\right)-ig\left(\frac{\delta(\vxi+\vxib)}{2},\frac{\delta(\vxi-\vxib)}{2},\het(\vv),\delta,\delta\param\right),\nonumber\\
 H(\delta\vxi,\delta\vxib,\het(\vv),\delta,\delta\param)&=&h\left(\frac{\delta(\vxi+\vxib)}{2},\frac{\delta(\vxi-\vxib)}{2},\het(\vv),\delta,\delta\param\right).\label{defF1F2H}
\end{eqnarray}
\begin{rmk}\label{ordreF12Horiginals}
 From \eqref{defF1F2H}, it is clear that $F_1(\phi)$, $F_2(\phi)$ and $H(\phi)$ are of order three and analytic if $\phi=(\phi_1,\phi_2,\phi_3,\phi_4,\phi_5)\in B^3(r_0)\times B(\delta_0)\times B(\param_0)$. Then we have that there exists some constant $K$, independent of $\delta$, such that:
\begin{equation}
|F_1(\phi)|,|F_2(\phi)|,|H(\phi)|\leq K|(\phi_1,\phi_2,\phi_3,\phi_4,\phi_5)|^3.
\end{equation}
\end{rmk}

\begin{thm}\label{theoremouter}
Let $\dist>0$ and $0<\beta<\pi/2$ be any constants independent of $\delta$. Then the one-dimensional invariant manifolds of $S_\pm(\delta,\param)$ can be parameterized respectively by:
$$\vxi=\vxi^{\uns,\sta}(\vu),\qquad\vxib=\vxib^{\uns,\sta}(\vu),\qquad z=\het(\vu),\qquad  \vu\in\doutT{*},$$
where $*=\uns,\sta$ respectively, and $\vphi^{\uns,\sta}(\vu)=(\vxi^{\uns,\sta}(\vu),\vxib^{\uns,\sta}(\vu))$ are solutions of system \eqref{sys2d}. Moreover, there exists a constant $K$, independent of $\delta$, such that:
$$|\vphi^{\uns,\sta}(\vu)|\leq K\delta^2|\het(\vu)-1|^3,\quad \vu\in\doutT{*},\qquad \text{for }*=\uns,\sta.$$
\end{thm}
The proof of this result can be found in Section \ref{prooftheoremouterbounded}.

\subsection{The inner system}
As we mentioned before, our study requires the knowledge of the asymptotics of the parameterizations $\vphi^{\uns,\sta}(\vu)$, given in Theorem \ref{theoremouter}, for $\vu$ near the singularities $\pm i\pi/2$. However, for $\vu\sim i\pi/2$ one has that $\vphi^{\uns,\sta}(\vu)\sim \delta^{-1}$, so that they are no longer perturbative (recall that, in the variables $(\vxi,\vxib)$, the heteroclinic orbit of the unperturbed system is $(\vxi,\vxib)=(0,0)$). Hence, it is natural to look for good approximations of system \eqref{sys2d} near $\pm i\pi/2$ in a different way. Here we will focus on the singularity $i\pi/2$, but similar results (which we will also state explicitly) can be proved near the singularity $-i\pi/2$.

To study the solutions of system \eqref{sys2d} near $i\pi/2$, we define the new variables
$(\vpsi,\vpsib,\vs)=C_3(\vxi,\vxib,\vu,\delta)$ by:
\begin{equation}\label{canviinner}\vpsi=\delta\vxi,\quad \vpsib=\delta\vxib,\quad\vs=\frac{\vu-i\pi/2}{\delta}.\end{equation}
Recalling that $\het(\vu)=-\tanh\vu$, we can write:
\begin{equation}\nonumber
 \begin{array}{rclr}
  \het(i\pi/2+\delta\vs)&=&\displaystyle\frac{-1}{\delta\vs}+\f,\qquad &\text{with }\f[0]=0,\medskip\\
(-1+\het^2(i\pi/2+\delta\vs))^{-1}&=&\delta^2\vs^2+\delta^3\vs^3\g,\qquad &\text{with }\g[0]=0.
 \end{array}
\end{equation}
We note that both $l$ and $m$ are analytic if $|\delta\vs|<1$. Then system \eqref{sys2d} after performing the change $C_3$ becomes:
\begin{equation}\label{sysdelta}
 \begin{array}{rcl}
  \displaystyle\frac{d\vpsi}{d\vs}&=&\displaystyle\frac{-\left[\alpha+c(-\vs^{-1}+\delta\f)\right]i\vpsi-\vpsi(\delta\param-\coef\vs^{-1}+\delta\coef\f)+F_1(\vpsi,\vpsib,-\vs^{-1}+\delta\f, \delta,\delta\param)}{1+\left[b\vpsi\vpsib+H(\vpsi,\vpsib, -\vs^{-1}+\delta\f, \delta,\delta\param)\right](\vs^2+\delta\vs^3\g)},\medskip\\
  \displaystyle\frac{d\vpsib}{d\vs}&=&\displaystyle\frac{\left[\alpha+c(-\vs^{-1}+\delta\f)\right]i\vpsib-\vpsib(\delta\param-\coef\vs^{-1}+\delta\coef\f)+F_2(\vpsi,\vpsib,-\vs^{-1}+\delta\f, \delta,\delta\param)}{1+\left[b\vpsi\vpsib+H(\vpsi,\vpsib, -\vs^{-1}+\delta\f, \delta,\delta\param)\right](\vs^2+\delta\vs^3\g)}.
 \end{array}
\end{equation}
If we set $\delta=0$ in this system, we obtain the \textit{inner system}:
\begin{equation}\label{sysinner}
 \begin{array}{rcl}
\displaystyle\frac{d\vpsi}{d\vs}&=&\frac{\displaystyle-\left(\alpha-cs^{-1}\right)i\vpsi+\coef\vpsi s^{-1}+F_1(\vpsi,\vpsib,-\vs^{-1},0,0)}{\displaystyle1+\vs^2\left[b\vpsi\vpsib+H(\vpsi,\vpsib, -\vs^{-1},0,0)\right]},\medskip\\
\displaystyle\frac{d\vpsib}{d\vs}&=&\frac{\displaystyle\left(\alpha-cs^{-1}\right)i\vpsib+\coef\vpsib s^{-1}+F_2(\vpsi,\vpsib,-\vs^{-1},0,0)}{\displaystyle1+\vs^2\left[b\vpsi\vpsib+H(\vpsi,\vpsib, -\vs^{-1},0,0)\right]}.
 \end{array}
\end{equation}

Below, we will expose the results concerning the existence of two solutions $\Vpsi_0^{\uns,\sta}=(\vpsi_0^{\uns,\sta},\vpsib_0^{\uns,\sta})$ of system \eqref{sysinner} which, as we will see in Theorem \ref{theoremmatching}, will give good approximations for the invariant manifolds for $\vu$ near the singularity $i\pi/2$. Moreover, we will provide an asymptotic expression for the difference $\Vpsi_0^\uns-\Vpsi_0^\sta$, which will turn out to be very useful in Section \ref{prooftheoremsplit}.

Given $\beta_0,\rho>0$, we define the following inner domains (see Figure \ref{figuradinu}):
\begin{equation}\label{defdinus}
 \din{\uns}=\{\vs\in\mathbb{C}\,:\, |\im \vs|\geq\tan\beta_0\re \vs+\rho\},\qquad \din{\sta}=-\din{\uns}.
\end{equation}
and:
\begin{equation}\label{defein}
\ein=\din{\uns}\cap\din{\sta}\cap\{\vs\in\mathbb{C}\,:\,\im\vs<0\}.
\end{equation}

\begin{figure}
 \center
 \includegraphics[width=8cm]{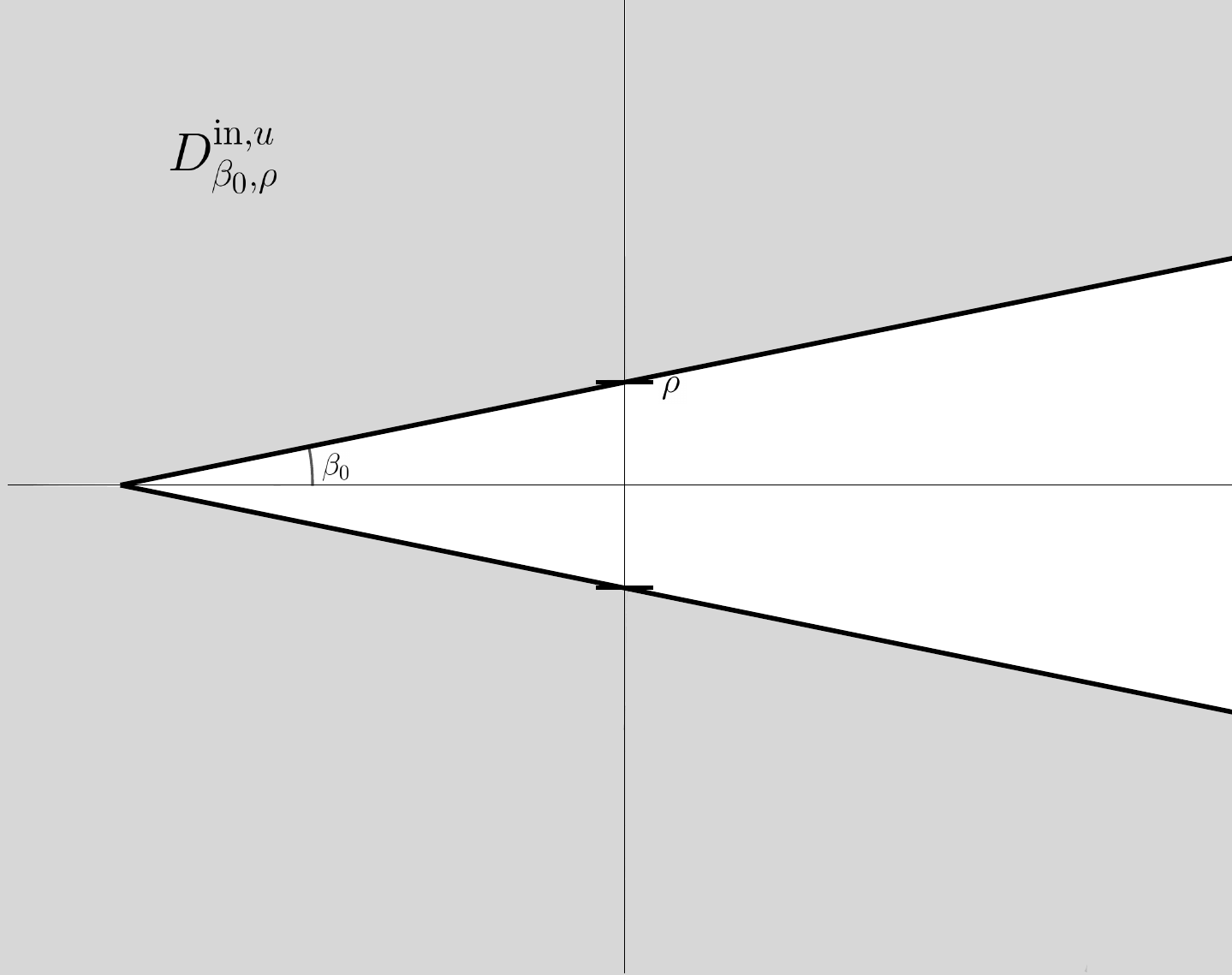}
 \caption{The inner domain, $\din{\uns}$.}
 \label{figuradinu}
\end{figure}
\begin{rmk}\label{remarkdoutsubsetdin}
 The inner domain $\din{\uns}$ expressed in the outer variables is:
$$\dinu{\uns}=\{\vu\in\mathbb{C}\,:\, |\im (\vu-i\pi/2)|\geq\tan\beta_0\re \vu+\rho\delta\}.$$
It is easy to check that for all $0<\beta_0,\beta<\pi/2$, if $\delta$ is small enough one has that $\dout{\uns}\subset\dinu{\uns}$. Analogously, we also have that $\dout{\sta}\subset\dinu{\sta}$, where $\dinu{\sta}=-\dinu{\uns}$.
\end{rmk}

\begin{thm}\label{theoreminner}
 Let $\beta_0>0$ and $\rho$ big enough. Then:
\begin{enumerate}
 \item \label{varietatsinner} System \eqref{sysinner} has two solutions $\Vpsi_0^{\uns,\sta}(\vs)=(\vpsi_0^{\uns,\sta}(\vs),\vpsib_0^{\uns,\sta}(\vs))$ defined for $\vs\in\din{*}$, with $*=\uns,\sta$ respectively. Moreover there exists a constant $K$, such that:
$$|\Vpsi_0^{\uns,\sta}(\vs)|\leq K|\vs|^{-3}.$$

\item \label{diferenciainner} Consider the difference:
$$\Delta\Vpsi_0(\vs)=\Vpsi_0^\uns(\vs)-\Vpsi_0^\sta(\vs),\qquad \vs\in\ein.$$
There exists $C_{\textup{in}}\in\mathbb{C}$ and a function $\chi:\ein\rightarrow\mathbb{C}^2$ such that:
\begin{equation}\label{asyVpsi0}
 \Delta\Vpsi_0(\vs)=\vs^\coef e^{-i(\alpha\vs-(c+\alpha h_0)\log\vs)}\left(\left(\begin{array}{c}C_{\textup{in}}\\0\end{array}\right)+\chi(\vs)\right),
\end{equation}
where $h_0=\lim_{\re s\to\infty}{s^{3}H(0,0,-s^{-1},0,0)}$ and $\chi=(\chi_1,\chi_2)$ satisfies:
$$|\chi_1(\vs)|\leq K|\vs|^{-1},\qquad |\chi_2(\vs)|\leq K|\vs|^{-2}.$$
Moreover, $C_{\textup{in}}\neq0$ if and only if $\Delta\Vpsi_0\neq0$.
\end{enumerate}
\end{thm}

The inner system corresponding to system \eqref{sys2d} with $\coef=1$ was exhaustively studied in \cite{BaSe08}. Moreover, the authors used an extra parameter $\varepsilon$ (not necessarily small) which we take $\varepsilon=1$. Since the proof for the case where $\coef$ is a free parameter and $\varepsilon=1$ is completely analogous, in Section \ref{sketchtheoreminner} we will give just the main ideas of how Theorem \ref{theoreminner} can be proved for this case without going into details.

\begin{rmk}\label{remarkinnerasota}
 The change \eqref{canviinner} allows us to study some approximations of the invariant manifolds and their difference near the singularity $i\pi/2$. However, if we want to approximate these manifolds and their difference near the singularity $-i\pi/2$, instead of change \eqref{canviinner} one has to introduce the following change:
\begin{equation}\label{canviinner2}\vpsi=\delta\vxi,\quad \vpsib=\delta\vxib,\quad\vs=\frac{\vu+i\pi/2}{\delta}.\end{equation}
In this case, one can prove the existence of two solutions $\tilde{\Vpsi}_0^{\uns,\sta}(\vs)$ of the inner system obtained after doing change \eqref{canviinner2}, which are defined for $\vs\in\overline{\din{*}}$, with $*=\uns,\sta$, where:
$$\overline{\din{*}}=\{\vs\in\mathbb{C}\,:\, \overline{\vs}\in\din{*}\}.$$
Moreover, for:
$$\vs\in\overline{\ein}:=\overline{\din{\uns}}\cap\overline{\din{\sta}}\cap\{\vs\in\mathbb{C}\,:\,\im\vs>0\},$$
the difference between these two solutions, $\Delta\tilde{\Vpsi}_0(\vs)$, is given asymptotically by:
$$\Delta\tilde{\Vpsi}_0(\vs)=\vs^\coef e^{i(\alpha\vs-(c+\alpha h_0)\log\vs)}\left(\left(\begin{array}{c}0\\\overline{C_{\textup{in}}}\end{array}\right)+\tilde{\chi}(\vs)\right),$$
where $\overline{C_{\textup{in}}}$ is the conjugate of the constant $C_{\textup{in}}$ in Theorem \ref{theoreminner} and $\tilde{\chi}=(\tilde{\chi}_1,\tilde{\chi}_2)$ satisfies that $|\tilde{\chi}_1(\vs)|\leq|\vs|^{-2}$ and $|\tilde{\chi}_2(\vs)|\leq|\vs|^{-1}$.
\end{rmk}

\subsection{Study of the matching error}

Let us recall the domains $\doutT{\uns}$ and $\doutT{\sta}$, defined in \eqref{defDTDinf}, where the parameterizations $\vphi^{\uns,\sta}$ of the invariant manifolds given by Theorem \ref{theoremouter} are defined, for some fixed $\dist>0$ and $0<\beta<\pi/2$. We also recall the domains $\din{\uns}$ and $\din{\sta}$, defined in \eqref{defdinus}, with $\rho>0$ and $0<\beta_0<\pi/2$ fixed, where the solutions $\Vpsi_0^{\uns,\sta}$ given in Theorem \ref{theoreminner} are defined. Now we take $\beta_1,\beta_2$ two constants independent of $\delta$, such that:
\begin{equation}\label{condicionsbeta12}
0<\beta_1<\beta<\beta_2<\pi/2.
\end{equation}
We define $\vu_j\in\mathbb{C}$, $j=1,2$ as the two points that satisfy (see Figure \ref{figuradmatchmes}):
\begin{itemize}
 \item $\im \vu_j=-\tan\beta_j\re \vu_j+\pi/2-\dist\delta\log(1/\delta)$,
\item $|\vu_j-i(\pi/2-\dist\delta\log(1/\delta))|=\delta^\gamma$, where $\gamma\in(0,1)$ is a constant independent of $\delta$,
\item $\re \vu_1<0$, $\re\vu_2>0$.
\end{itemize}
We also consider the following domains (see Figure \ref{figuradmatchmes}):
\begin{align}
\dmatch{\uns}=&\left\{\vu\in\mathbb{C}\,:\, \im u\leq -\tan\beta_1\re \vu+\pi/2-\dist \delta\log(1/\delta),\,\im \vu\leq -\tan\beta_2\re \vu+\pi/2-\dist \delta\log(1/\delta),\right.\nonumber\\
&\,\,\textstyle\left.\im \vu\geq \im u_1-\tan\left(\frac{\beta_1+\beta_2}{2}\right)(\re \vu-\re \vu_1)\right\}, \nonumber
\end{align}
and:
 $$\dmatch{\sta}=\{\vu\in\mathbb{C}\,:\, -\overline{\vu}\in\dmatch{\uns}\}.$$
\begin{figure}
 \center
 \subfigure[The matching domain $\dmatch{\uns}$.]{\includegraphics[height=4.5cm]{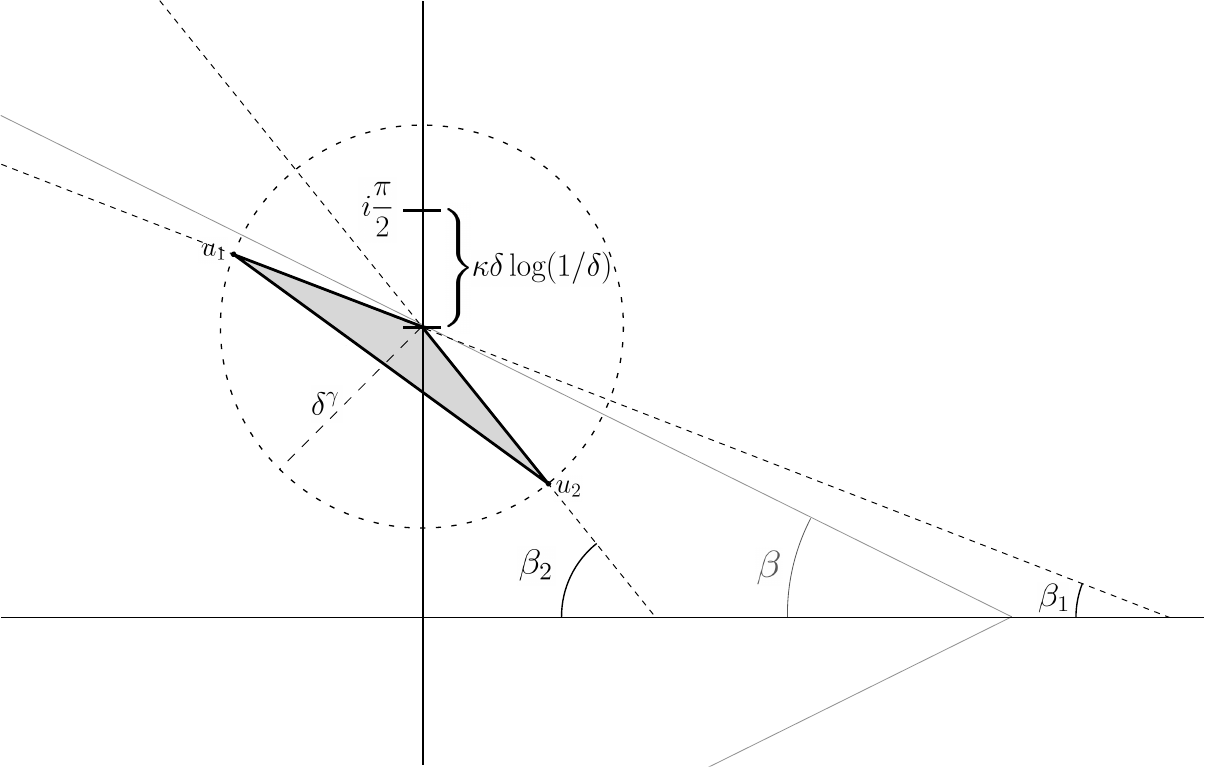}}
\hspace{1cm}
 \subfigure[The matching domain $\dmatch{\sta}$.]{\includegraphics[height=4.5cm]{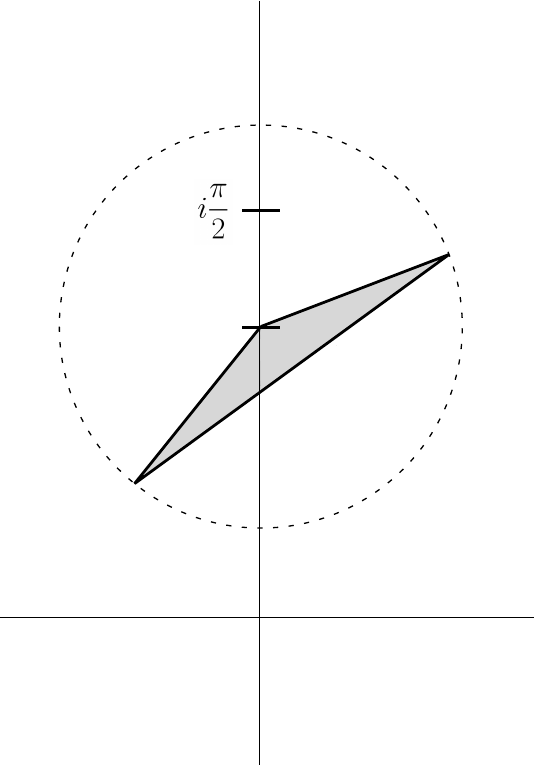}}
 \caption{The matching domains in the outer variables.}
 \label{figuradmatchmes}
\end{figure}
We note that there exist two constants $K_1$ and $K_2$, independent of $\delta$, such that:
$$K_1\delta^\gamma\leq|\vu_j-i\pi/2|\leq K_2\delta^\gamma,\qquad j=1,2.$$
Moreover, for all $\vu\in\dmatch{*}$, $*=\uns,\sta$, we have:
\begin{equation}\label{fitaudmatch}
\dist\cos\beta_1\delta\log(1/\delta)\leq|\vu-i\pi/2|\leq K_2\delta^\gamma.
\end{equation}
Note that from \eqref{condicionsbeta12} and \eqref{fitaudmatch} we have $\dmatch{\uns}\subset \doutT{\uns}$ and $\dmatch{\sta}\subset \doutT{\sta}$, if $\delta$ is small enough .

We also define the matching domains in the inner variables:
$$\dmatchs{*}=\{\vs\in\mathbb{C}\,:\,i\pi/2+\vs\delta\in\dmatch{*}\}, \qquad *=\uns,\sta$$
and:
\begin{equation}\label{etsj}
 \vs_j=\frac{\vu_j-i\pi/2}{\delta},\quad j=1,2.
\end{equation}
It is clear that:
\begin{equation}\label{fitasj}
 K_1\delta^{\gamma-1}\leq|\vs_j|\leq K_2\delta^{\gamma-1},\qquad j=1,2,
\end{equation}
and that for all $\vs\in\dmatchs{*}$, where $*=\uns,\sta$, we have:
$$\dist\cos\beta_1\log(1/\delta)\leq|\vs|\leq K_2\delta^{\gamma-1}.$$
Using that $\dmatch{*}\subset\doutT{*}$ and Remark \ref{remarkdoutsubsetdin} it is clear that $\dmatchs{*}\subset\din{*}$ if $\delta$ is small enough, $*=\uns,\sta$.

The main result of this section is the following.
\begin{thm}\label{theoremmatching}
 Let $\Vpsi^{\uns,\sta}(\vs)=\delta\vphi^{\uns,\sta}(\delta\vs+i\pi/2)$, where $\vphi^{\uns,\sta}$ are the parameterizations given by Theorem \ref{theoremouter}. Then, if $\vs\in\dmatchs{*}$, for $*=\uns,\sta$, one has $\Vpsi^{\uns,\sta}(\vs)=\Vpsi_0^{\uns,\sta}(\vs)+\Vpsi_1^{\uns,\sta}(\vs)$, where $\Vpsi_0^{\uns,\sta}(\vs)$ are the two solutions of the inner system \eqref{sysinner} given by Theorem \ref{theoreminner}, $\Vpsi_1^{\uns,\sta}$ and there exists a constant $K$, independent of $\delta$, such that:
$$|\Vpsi_1^{\uns,\sta}(\vs)|\leq  K\delta^{1-\gamma}|\vs|^{-2}.$$
\end{thm}
This Theorem is proved in Section \ref{prooftheoremmatching}. From this result, the following corollary is clear:
\begin{cor}\label{corollarymatching}
 For $\vu\in\dmatch{*}$, where $*=\uns,\sta$, we have that:
$$\vphi^{\uns,\sta}(\vu)=\frac{1}{\delta}\left(\Vpsi_0^{\uns,\sta}\left(\frac{\vu-i\pi/2}{\delta}\right)+\Vpsi_1^{\uns,\sta}\left(\frac{\vu-i\pi/2}{\delta}\right)\right),$$
where $\Vpsi_0^{\uns,\sta}$ are the two solutions of the inner system \eqref{sysinner} given by Theorem \ref{theoreminner} and:
$$\left|\Vpsi_1^{\uns,\sta}\left(\frac{\vu-i\pi/2}{\delta}\right)\right|\leq K\delta^{1-\gamma}\left|\frac{\vu-i\pi/2}{\delta}\right|^{-2},$$
for some constant $K$. Note that, as for $u\in\dmatch{*}$, $|u-i\pi/2|\geq K\delta\log(1/\delta)$, from this last inequality we obtain:
$$\left|\Vpsi_1^{\uns,\sta}\left(\frac{\vu-i\pi/2}{\delta}\right)\right|\leq \frac{K\delta^{1-\gamma}}{\log^2(1/\delta)},$$
and since $\gamma\in(0,1)$ we obtain that $\Vpsi_0^{\uns,\sta}$ are good approximations of $\vphi^{\uns,\sta}$ in $\dmatch{\uns}$ and $\dmatch{\sta}$ respectively.
\end{cor}
\begin{rmk}
 Theorem \ref{theoremmatching} and, more precisely, Corollary \ref{corollarymatching} provide us with a bound of the difference between the invariant manifolds $\vphi^{\uns,\sta}(\vu)$ of Theorem \ref{theoremouter} and the functions  $\Vpsi_0^{\uns,\sta}((\vu-i\pi/2)/\delta)$ given by Theorem \ref{theoreminner},  when $\vu$ is near the singularity $i\pi/2$. One can proceed similarly to study this difference near the singularity $-i\pi/2$ as we pointed out in Remark \ref{remarkinnerasota}. In this case, defining:
$$\overline{\dmatch{*}}=\{\vs\in\mathbb{C}\,:\,\overline{\vs}\in\dmatch{*}\},\qquad\qquad \textup{for }*=\uns,\sta$$
we would obtain that for $\vu\in\overline{\dmatch{*}}$, one has:
$$\vphi^{\uns,\sta}(\vu)=\frac{1}{\delta}\left(\tilde\Vpsi_0^{\uns,\sta}\left(\frac{\vu+i\pi/2}{\delta}\right)+\tilde\Vpsi_1^{\uns,\sta}\left(\frac{\vu+i\pi/2}{\delta}\right)\right),$$
where $\tilde\Vpsi_0^{\uns,\sta}$ are the two solutions of the inner system derived from the change \eqref{canviinner2} in Remark \ref{remarkinnerasota}, and:
$$\left|\tilde\Vpsi_1^{\uns,\sta}\left(\frac{\vu+i\pi/2}{\delta}\right)\right|\leq K\delta^{1-\gamma}\left|\frac{\vu+i\pi/2}{\delta}\right|^{-2},$$
for some constant $K$.
\end{rmk}

\subsection{Asymptotic formula for the splitting distance}\label{subsectionsplit}

\begin{thm}\label{theoremsplit}
 Let $\vphi^\uns$, $\vphi^\sta$ be the parameterizations given by Theorem \ref{theoremouter}. For $\vu\in\doutT{\uns}\cap\doutT{\sta}$, we define its difference:
\begin{equation}\label{defDeltavphi}\Delta\vphi(\vu)=\vphi^\uns(\vu)-\vphi^\sta(\vu).\end{equation}
Let $C_{\textup{in}}\in\mathbb{C}$ be the constant in Theorem \ref{theoreminner}. If $C_{\textup{in}}\neq0$, then:
$$\Delta\vphi(0)=\delta^{-(1+\coef)}e^{-\frac{\alpha_0\pi}{2\delta}}e^{\frac{\pi}{2}(c+\alpha_0 h_0-\alpha_1\param)}\left(\left(\begin{array}{c}C_{\textup{in}}e^{-i\left(\frac{\param\pi}{2}+\frac{\alpha_0h_0}{2}+(c+\alpha_0h_0)\log\delta\right)}\medskip\\\overline{C_{\textup{in}}}e^{i\left(\frac{\param\pi}{2}+\frac{\alpha_0h_0}{2}+(c+\alpha_0h_0)\log\delta\right)}\end{array}\right)+O\left(\frac{1}{\log(1/\delta)}\right)\right),$$
where $h_0=-\lim_{z\to0}z^{-3}H(0,0,z,0,0)$, $\alpha_0=\alpha(0)$ and $\alpha_1=\alpha'(0)$.
\end{thm}
\begin{rmk}
 Note that from Theorem \ref{theoremsplit}, doing the inverse of change $C_2$ (defined in \eqref{canviC2}) and taking norms, we obtain Theorem \ref{maintheoremrescaled}, with $C^*=|C_{\textup{in}}|$.
\end{rmk}

In this section we will give the main ideas of how Theorem \ref{theoremsplit} can be proved. The full proof can be found in Section \ref{prooftheoremsplit}.

First of all recall that both $\vphi^\uns$ and $\vphi^\sta$ satisfy equations \eqref{sys2d}. We will decompose system \eqref{sys2d} in a more convenient form. For that we define:
\begin{eqnarray}\mathcal{A}(\vu)&=&\frac{1}{1-\frac{\delta h_0\het^3(\vu)}{-1+\het^2(\vu)}}\left(\begin{array}{cc}
                  \displaystyle-\left(\frac{\alpha }{\delta}+c\het(\vu)\right)i+\param-\coef\het(\vu) & 0\medskip\\
		  0 & \displaystyle\left(\frac{\alpha }{\delta}+c\het(\vu)\right)i+\param-\coef\het(\vu)
                 \end{array}\right),\label{defAsplit}\medskip\\
\nonumber&&\medskip\\
 \mathcal{R}(\vphi)(\vu)&=&\frac{\delta^{-2}F(\delta\vphi,\delta\het(\vu),\delta,\delta\param)}{1+\frac{b\vxi\vxib+\delta^{-2}H(\delta \vphi, \delta \het(\vu), \delta, \delta\param)}{-1+\het^2(\vu)}}+\left(\frac{1}{1+\frac{b\vxi\vxib+\delta^{-2}H(\delta \vphi, \delta \het(\vu), \delta, \delta\param)}{-1+\het^2(\vu)}}-\frac{1}{1-\frac{\delta h_0\het^3(\vu)}{-1+\het^2(\vu)}}\right)\mathcal{A}(\vu)\vphi.\label{defRsplit}
\end{eqnarray}
Then, system \eqref{sys2d} can be written as:
\begin{equation}\label{sys2dmodified}\frac{d\vphi}{d\vu}=\mathcal{A}(\vu)\vphi+\mathcal{R}(\vphi)(\vu).\end{equation}
Since $\vphi^\uns$ and $\vphi^\sta$ satisfy system \eqref{sys2dmodified}, it is clear that its difference $\Delta\vphi=\vphi^\uns-\vphi^\sta$ satisfies:
$$\frac{d\Delta\vphi}{d\vu}=\mathcal{A}(\vu)\Delta\vphi+\mathcal{R}(\vphi^\uns)(\vu)-\mathcal{R}(\vphi^\sta)(\vu).$$
Following \cite{Sau01}, using the mean value theorem we can still rewrite this equation as the following linear equation:
\begin{equation}\label{eqsplit}
 \frac{d\Delta\vphi}{d\vu}=\mathcal{A}(\vu)\Delta\vphi+\mathcal{B}(\vu)\Delta\vphi,
\end{equation}
with:
\begin{equation}\label{defBsplit}
\mathcal{B}(\vu)=\int_0^1D\mathcal{R}((1-\lambda)\vphi^\sta-\lambda\vphi^\uns)(\vu)d\lambda.
\end{equation}
We observe that we can think of the matrix $\mathcal{B}$ as just depending on $\vu$, because the existence of $\vphi^\uns$ and $\vphi^\sta$ has been already proved in Theorem \ref{theoremouter}.

The point of writing the system for $\Delta\vphi$ as \eqref{eqsplit} is that, as we shall see, we split it into a dominant part, the one corresponding to the matrix $\mathcal{A}(\vu)$, and a \textit{small perturbation}, which corresponds to the the matrix $\mathcal{B}(\vu)$. This will allow us to find an asymptotic expression for $\Delta\vphi(\vu)$, with its dominant term given by the solution of the system:
\begin{equation}\label{hsys}
 \frac{d\Delta\vphi}{d\vu}=\mathcal{A}(\vu)\Delta\vphi.
\end{equation}

\begin{lem}\label{lemamatriufonsplit}
 For $\vu\in \doutT{\uns}\cap\doutT{\sta}$, a fundamental matrix of the homogeneous system \eqref{hsys} is:
\begin{equation}\label{defMsplit}
 \mathcal{M}(\vu)=\left(\begin{array}{cc} m_1(\vu) & 0 \\ 0 & m_2(\vu)\end{array}\right),
\end{equation}
with:
\begin{equation}\label{asyformm1}\begin{array}{rcl}
m_1(\vu)&=&\cosh^\coef\vu e^{-\alpha i\vu/\delta}e^{\param\vu}e^{\alpha  h_0i\left[-\frac{1}{2}\sinh^2\vu+\log\cosh\vu\right]}e^{ic\log\cosh\vu}\left(1+O\left(\frac{1}{\log(1/\delta)}\right)\right),\medskip\\
m_2(\vu)&=&\cosh^\coef\vu e^{\alpha i\vu/\delta}e^{\param\vu}e^{-\alpha  h_0i\left[-\frac{1}{2}\sinh^2\vu+\log\cosh\vu\right]}e^{-ic\log\cosh\vu}\left(1+O\left(\frac{1}{\log(1/\delta)}\right)\right).\end{array}
\end{equation}
\end{lem}
The proof of Lemma \ref{asyformm1} will be given in Section \ref{prooftheoremsplit}.

In the following we shall give an heuristic idea of how the asymptotic formula given in Theorem \ref{theoremsplit} can be found. For simplicity, we will focus just on the first component of $\Delta\vphi$, that is $\Delta\vxi$.

Let us omit the influence of $\CMcal{B}$, that is assume that $\CMcal{B}(\vu)\equiv 0$. Then, any solution $\Phi$ of \eqref{eqsplit} can be written as:
$$\Phi(\vu)=\mathcal{M}(\vu)\left(\begin{array}{c}c_1\\c_2\end{array}\right),$$
for certain $c_1$, $c_2$, and its first component is $m_1(\vu)c_1$. Hence, $\Delta\vxi(\vu)=m_1(\vu)c_1$ for a certain $c_1$. The main idea is that $\Delta\vxi(\vu)$ is bounded when $\vu\in\doutT{\uns}\cap\doutT{\sta}$. The first thing we observe is that from the asymptotic expression of $m_1(\vu)$ in Lemma \ref{lemamatriufonsplit} we can already see that $\Delta\vxi(\vu)$ has an exponentially small bound if $\vu\in\mathbb{R}$. Indeed, it is clear that when $\vu\sim i\pi/2$ we have that $m_1(\vu)\sim e^{\frac{\alpha\pi}{2\delta}}$, that is exponentially big. Then, $c_1$ has to be $\sim e^{-\frac{\alpha\pi}{2\delta}}$ for $\Delta\vxi(\vu)$ to be bounded, i.e. it must be exponentially small. As a consequence, when $\vu\in\mathbb{R}$ we have that $\Delta\vxi(\vu)=m_1(\vu)c_1$ is exponentially small.

However, we do not want a bound of $\Delta\vxi$ but an asymptotic formula. Thus we have to find the constant $c_1$ that corresponds to $\Delta\vxi$, or more concretely a good approximation $c_1^0$ of it. We recall that near the singularity $i\pi/2$ we have a good approximation $\Delta\vpsi_0$ of $\Delta\vxi$ given by the study of the inner equation in Theorem \ref{theoreminner}. Then, if we consider the point:
$$\vu_+=i\left(\frac{\pi}{2}-\kappa\delta\log(1/\delta)\right),$$
it is clear that the initial condition $c_1$ satisfies:
$$c_1m_1(\vu_+)=\Delta\vxi(\vu_+)\approx\delta^{-1}\Delta\vpsi_0\left( \frac{\vu_+ - i\pi/2}{\delta}\right).$$
From Theorem \ref{theoreminner} we know that:
$$\delta^{-1}\Delta\vpsi_0\left ( \frac{\vu_+ -i\pi/2}{\delta}\right )=\frac{(-i\kk)^\coef}{\delta}e^{-\kk\alpha-(c+\alpha h_0)\log(-i\kk)}(C_{\textup{in}}+\chi_1(-i\kk)),$$
where $\kk=\dist\log(1/\delta)$, and therefore:
$$c_1=m_1(\vu_+)^{-1}\Delta\vxi(\vu_+)\approx m_1^{-1}(\vu_+)\delta^{-1}\Delta\vpsi_0\left (\frac{\vu_+ - i\pi/2}{\delta} \right )
\approx m_1^{-1}(\vu_+)\frac{(-i\kk)^\coef}{\delta} e^{-\alpha \kk+i(c+\alpha  h_0)\log(-i\kk)}C_{\textup{in}},$$
so that taking:
$$c_1^0= m_1^{-1}(\vu_+)\frac{(-i\kk)^\coef}{\delta} e^{-\alpha \kk+i(c+\alpha  h_0)\log(-i\kk)}C_{\textup{in}},$$
we obtain a good approximation $\Delta\vxi_0(\vu)$ of $\Delta\vxi(\vu)$ defined by:
$$\Delta\vxi_0(\vu):=m_1(\vu)c_1^0.$$
Using the bound of the matching error given in Theorem \ref{theoremmatching}, it can be proved that $\Delta\vxi_0(\vu)$ is the dominant part of $\Delta\vxi(\vu)$. Then, computing explicitly the asymptotic formula of $\Delta\vxi_0(0)$ one obtains the first component of the dominant term of the formula given in Theorem \ref{theoremsplit}. As we will see in Section \ref{prooftheoremsplit}, $c_1^0$ does not depend on  $\dist$.

For the second component, $\Delta\vxib$, we can repeat the same arguments, but using the singularity $-i\pi/2$. Finally, this procedure can be adapted to the whole system \eqref{eqsplit}, using the fact that, indeed, $\mathcal{B}(\vu)$ is small.

\section{Proof of Theorem \ref{thmouterunbounded}}\label{prooftheoremouterunbounded}
In this section we will prove Theorems \ref{thmouterunbounded} and  \ref{theoremouter}. However, in order to do that, first we need to define suitable Banach spaces in which we will work, which are the following:
\begin{equation*}
 \bsout{*}=\left\{\phi:\doutb{*}\rightarrow\mathbb{C}\, :\, \phi \textup{ analytic, } \|\phi\|_{\textup{out}}^*<\infty\right\},
\end{equation*}
where $*=\uns,\sta$, and the norm $\|.\|_{\textup{out}}^{\uns,\sta}$ defined as:
\begin{equation}\label{defnorm}
\begin{array}{lll}
 \|\phi\|_{\textup{out}}^\sta&=&\displaystyle\sup_{\vv\in \doutbinf{\sta} }|(\het(\vv)+1)^{-1}\phi(\vv)|+\sup_{\vv\in \doutbT{\sta}}|(\het(\vv)+1)^{-3}\phi(\vv)|,\medskip\\
 \|\phi\|_{\textup{out}}^\uns&=&\displaystyle\sup_{\vv\in \doutbinf{\uns} }|(\het(\vv)-1)^{-1}\phi(\vv)|+\sup_{\vv\in \doutbT{\uns}}|(\het(\vv)-1)^{-3}\phi(\vv)|.
\end{array}
\end{equation}

In the product space $\bsout{*}\times\bsout{*}$, with $*=\uns,\sta$, we take the norm:
$$\|(\phi_1,\phi_2)\|_{\textup{out},\times} ^{\uns,\sta}=\|\phi_1\|_{\textup{out}}^{\uns,\sta}+\|\phi_2\|_{\textup{out}}^{\uns,\sta},\qquad (\phi_1,\phi_2)\in\bsout{*}\times\bsout{*}.$$

Below we will introduce some notation that will allow us to see $\vzeta^{\uns,\sta}$ as fixed points of a certain operator. Given $\alpha $ and $c$, we define the linear operators acting on functions $\phi_1\in\bsout{*}$:
\begin{equation}\label{operadorLalphac}
 L_{\alpha ,c}^{\uns,\sta}(\phi_1)(\vv)=\cosh^\coef\vv\int_{\mp\infty}^0{\frac{1}{\cosh^\coef(\vv+r)}e^{i\alpha  r/\delta}e^{\param r}g_c^{\uns,\sta}(\vv,r)\phi_1(\vv+r)dr},
\end{equation}
where $*=\uns,\sta$, $-$ corresponds to $\uns$ and $+$ to $\sta$, and:
$$g_c^\uns(\vv,r)=e^{ic(r+\log((1+e^{2\vv})/2)-\log((1+e^{2(\vv+r)})/2)},\qquad g_c^\sta(\vv,r)=e^{ic(-r+\log((1+e^{-2\vv})/2)-\log((1+e^{-2(\vv+r)})/2)},$$
\begin{rmk}
 One might think that instead of taking $g_c$, it would be more natural to take:
 $$\hat g_c(\vv,r)=e^{ic(\log\cosh\vv-\log\cosh(\vv+r))}.$$
Although $g_c^\uns(\vv,r)=g_c^\sta(\vv,r)=\hat g_c(\vv,r)$ if $\vv,r\in\mathbb{R}$, this is not the case when $\vv,r\in\mathbb{C}$. In particular, one can see that if $\vv,r\in\doutb{*}$, $*=\uns,\sta$, the function $\hat g_c$ is not well defined. On the contrary, the function $g_c^\uns$ is always well defined for $\vv,r\in\doutb{\uns}$ and $g_c^\sta$ is well defined for $\vv,r\in\doutb{\sta}$.
\end{rmk}

Now, given a function $\phi=(\phi_1,\phi_2)\in\bsout{*}\times\bsout{*}$ we define the linear operator:
\begin{equation}\label{operadorL}
L^{\uns,\sta}(\phi)=(L_{\alpha ,c}^{\uns,\sta}(\phi_1),L_{-\alpha ,-c}^{\uns,\sta}(\phi_2)).
\end{equation}
\begin{lem}
With the above notation, if a bounded and continuous function $\vzeta^{\uns,\sta}:\doutb{*}\rightarrow\mathbb{C}^3$, with $*=\uns,\sta$ respectively, satisfies the fixed point equation
\begin{equation}\label{fixedpoint}
\vzeta^{\uns,\sta}=L^{\uns,\sta}\circ R^{\uns,\sta}(\vzeta^{\uns,\sta}),
\end{equation}
then it is a solution of \eqref{syscompu}, \eqref{syscomps} respectively.
\end{lem}

In the rest of this section we will prove the following result, which is equivalent to Theorem \ref{thmouterunbounded}:
\begin{prop}\label{propouterunbounded}
Let $\overline{\dist}>0$ and $0<\beta<\pi/2$ be any fixed constants independent of $\delta$. Then, if $\delta>0$ is small enough, problem \eqref{syscompu} has a solution $\vzeta^\uns$ defined in $\doutb{\uns}$, and \eqref{syscomps} has a solution $\vzeta^\sta$ defined in  $\doutb{\sta}$, both satisfying that $\vzeta^{\uns,\sta}=\vzeta_0^{\uns,\sta}+\vzeta_1^{\uns,\sta}$
with the following properties:
\begin{enumerate}
 \item $\vzeta_0^{\uns,\sta}=L^{\uns,\sta}\circ R^{\uns,\sta}(0)\in\bsout{*}\times\bsout{*}$ and there exists a constant $K$ independent of
$\delta$ such that:
$$\|\vzeta_0^{\uns,\sta}\|_{\textup{out},\times} ^{\uns,\sta}\leq  K\delta^2. $$
\item $\vzeta_1^{\uns,\sta}\in\bsout{*}\times\bsout{*}$, and there exists a constant $K$ independent of
$\delta$ such that:
$$\|\vzeta_1^{\uns,\sta}\|_{\textup{out},\times} ^{\uns,\sta}\leq \frac{K}{\log(1/\delta)}\|\vzeta_0^{\uns,\sta}\|_{\textup{out},\times} ^{\uns,\sta}$$
\end{enumerate}
where $*=\uns,\sta$ respectively.
\end{prop}
From Proposition \ref{propouterunbounded} we obtain solutions $\vzeta^\uns$ and $\vzeta^\sta$ of systems given in \eqref{syscompu} and \eqref{syscomps} respectively. Note that by the definitions of $\bsout{\uns}$ and $\bsout{\sta}$, and since $\vzeta^\uns\in\bsout{\uns}$ and $\vzeta^\sta\in\bsout{\sta}$, we know that:
$$\lim_{\re\vv\to-\infty}\vzeta^\uns(\vv)=(0,0),\qquad\lim_{\re\vv\to+\infty}\vzeta^\sta(\vv)=(0,0).$$

From now on, we will focus just on the parameterization of the unstable manifold, $\vzeta^\uns$, being the proof for the stable one completely analogous. For this reason, if there is not danger of confusion, we will omit the superindices $-\uns-$ of $\vzeta$, $\bsout[]{}$, $\doutb[\dist][]{}$, etc. Moreover, we will not write explicitly the dependence on $\vv$ of $\vzeta$ (or any function belonging to $\bsout[]{}$). Finally, in the rest of the paper, if no confusion is possible, we will denote by $K$ any constant independent of $\delta$. Obviously, these constants $K$ will depend on $\overline{\dist}$ and $\beta$, which we will consider fixed.

Before proving Proposition \ref{propouterunbounded} we will present some technical results.

\begin{lem}\label{lemadifdiv}
 Let $\phi_1,\phi_2\in\mathbb{C}^n$, such that $|\phi_1|,|\phi_2|<1/2$. Then:
$$\left|\frac{1}{1+\phi_1}-\frac{1}{1+\phi_2}\right|\leq 4|\phi_1-\phi_2|.$$
\end{lem}

\begin{lem}\label{lemafitaDFiDH}
Let $f:\mathbb{C}^n\rightarrow\mathbb{C}$ be any function that is analytic in some open ball $B(r_0^1)\times\cdots\times B(r_0^n)\subset\mathbb{C}^n$, and assume that there exists some $\phi^*\in\mathbb{C}^n$ such that for all $\phi\in B(r_0^1)\times\cdots\times B(r_0^n)$:
\begin{equation}
\label{fitaF}|f(\phi)|\leq K|\phi-\phi^*|^k,
\end{equation}
for some constants $K>0$ and $k\in\mathbb{N}$. Take $\phi\in B(r_0^1/2)\times\cdots\times B(r_0^n/2)$ and assume that $\phi-\phi^*\in B(r_0^1)\times\cdots\times B(r_0^n)$. Then there exists a constant $\tilde K$ such that:
\begin{equation}\label{ineqDF}
|D_jf(\phi)|\leq \tilde K|\phi-\phi^*|^{k-1},
\end{equation}
where $D_j$ denotes the derivative with respect to the $j-th$ component $\phi_j$.
\end{lem}


\begin{cor}\label{corollaryfitaDFiDH}
Let $F_1^\uns, F_2^\uns$ and $H^\uns$ the functions defined in \eqref{defF1uF2uHu}. If $\vzeta\in\bsout[]{}\times\bsout[]{}$ is such that $\|\vzeta\|_{\textup{out},\times} \leq \delta^2C$ for some constant $C$, we have that for $\delta$ small enough:
\begin{equation}\label{fitaDFiDH}
 |D_jF_i^\uns(\delta\vzeta,\delta\het(\vv),\delta,\delta\param)|,|D_jH^\uns(\delta\vzeta,\delta\het(\vv),\delta,\delta\param)|\leq\left\{\begin{array}{ll}  K\delta^2 & \quad\textrm {if } \vv\in \doutbinf[]{},\medskip\\  K\delta^2|\het(\vv)-1|^2 & \quad\textrm {if } \vv\in \doutbT[]{}.\end{array}\right.
\end{equation}
\end{cor}

\begin{proof}
We will prove this result just for $D_1F_1^\uns$, being the other cases analogous. Note that for $\delta$ small enough $(\delta\vzeta,\delta\het(\vv))\in B^3(r_0^\uns/2)$ since by the fact that $\|\vzeta\|_{\textup{out},\times} \leq \delta^2C$ and the definition \eqref{defnorm} of the norm $\|.\|_{\textup{out},\times}$ we have:
\begin{equation}\label{fitadeltazeta}
|\delta\vzeta(\vv)|\leq \left\{\begin{array}{ll} C\delta^3|\het(v)-1|\leq\delta^3 C<r_0^\uns/2& \quad\textrm {if } \vv\in \doutbinf[]{},\medskip\\\displaystyle C\delta^3|\het(v)-1|^3\leq \frac{CK}{\log^3(1/\delta)}<r_0^\uns/2, & \quad\textrm {if }\vv\in \doutbT[]{},\end{array}\right.
\end{equation}
and:
$$
|\delta\het(\vv)|\leq\left\{\begin{array}{ll}  K\delta<r_0^\uns/2& \quad\textrm {if } \vv\in \doutbinf[]{},\medskip\\\displaystyle \frac{K}{\log(1/\delta)}<r_0^\uns/2, & \quad\textrm {if }\vv\in \doutbT[]{},\end{array}\right. \nonumber
$$
For the same reason it is clear that $\delta\het(\vv)-\delta\in B(r_0^\uns)$ for $\delta$ small enough, and then $\phi-\phi^*\in B^3(r_0^\uns)\times B(\delta_0)\times B(\param_0)$. Then we just have to take $\phi=(\delta\vzeta,\delta\het(\vv),\delta,\delta\param)$ and $\phi^*=(0,0,0,\delta,0,0)$ in Lemma \ref{lemafitaDFiDH}, since by Remark \ref{ordreF12H}, we have:
\begin{equation*}
|F_1^\uns(\delta\vzeta,\delta\het(\vv),\delta,\delta\param)|\leq K|(\delta\vzeta,\het(\vv)-\delta,\delta,\delta\param)|^3=K|\phi-\phi^*|^3,
\end{equation*}
and then by Lemma \ref{lemafitaDFiDH} we have:
\begin{equation}\label{pasintermig}
 |D_1F_1^\uns(\delta\vzeta,\delta\het(\vv),\delta,\delta\param)|=|D_1F_1^\uns(\phi)|\leq K|\phi-\phi^*|^2= K|(\delta\vzeta,\delta(\het(\vv)-1),\delta,\delta\param)|^2.
\end{equation}
Moreover, since, by \eqref{fitadeltazeta}, $|\delta\vzeta|\leq\delta^3C|\het(\vv)-1|$ if $\vv\in \doutbinf[]{}$ and $|\delta\vzeta|\leq\delta^3C|\het(\vv)-1|^3$ if $\vv\in \doutbT[]{}$, it is clear that:
$$|(\delta\vzeta,\delta(\het(\vv)-1),\delta,\delta\param)|\leq\left\{\begin{array}{ll}  K\delta & \quad\textrm {if } \vv\in \doutbinf[]{},\medskip\\  K\delta|\het(\vv)-1| & \quad\textrm {if } \vv\in \doutbT[]{},\end{array}\right.
$$
for some constant $K$.
With this bound and \eqref{pasintermig} we obtain immediately bound \eqref{fitaDFiDH}.
\end{proof}

\begin{cor}\label{corollaryfitaFiH}
Let $F_1^\uns, F_2^\uns$ and $H^\uns$ the functions defined in \eqref{defF1uF2uHu}. If $\delta$ is small enough and $\|\vzeta\|_{\textup{out},\times}\leq \delta^2C$ for some constant $C$, there exists a constant $K$ independent of $\delta$ such that, for $i=1,2$:
\begin{eqnarray*}
 |F_i^\uns(\delta\vzeta,\delta\het(\vv),\delta,\delta\param)|, |H^\uns(\delta\vzeta,\delta\het(\vv),\delta,\delta\param)|&\leq& \left\{\begin{array}{ll} \delta^3 K|\het(\vv)-1|& \quad\textrm {if } \vv\in \doutbinf[]{},\medskip\\\displaystyle \delta^3 K|\het(\vv)-1|^3, & \quad\textrm {if }\vv\in \doutbT[]{}.\end{array}\right.
\end{eqnarray*}
\end{cor}
\begin{proof}
Again, we will do the proof for $F_1^\uns$. Reasoning as in the proof of Corollary \ref{corollaryfitaDFiDH}, we know that $(\delta\vzeta,\delta\het(\vv))\in B^3(r_0^\uns)$ if $\delta$ is sufficiently small. Then, by the mean value theorem we have:
\begin{eqnarray}\label{pasintermiglemafitaFiH}
|F_1^\uns(\delta\vzeta,\delta\het(\vv),\delta,\delta\param)|&=&|F_1^\uns(\delta\vzeta,\delta\het(\vv),\delta,\delta\param)-F_1^\uns(0,0,\delta,\delta,\delta\param)|\nonumber\\
&\leq&\int_0^1|DF_1^\uns(\lambda\delta\vzeta,\delta+\lambda\delta(\het(\vv)-1),\delta,\delta\param)|d\lambda\cdot|\delta\vzeta(\vv),\delta(\het(\vv)-1)|,
\end{eqnarray}
provided that $F_1^\uns(0,0,\delta,\delta,\delta\param)=0$.
By inequality \eqref{fitadeltazeta} and the fact that, for $\vv\in \doutbT[]{}$, $|\het(\vv)-1|\leq K\delta^{-1}$, one can easily deduce that $|\delta\vzeta(\vv)|\leq K|\delta(\het(\vv)-1)|$. Using this fact and reasoning as in the proof of Corollary \ref{corollaryfitaDFiDH} to bound $|DF_1^\uns(\lambda\delta\vzeta,\delta+\lambda\delta(\het(\vv)-1),\delta,\delta\param)|$, inequality \eqref{pasintermiglemafitaFiH} yields:
\begin{eqnarray*}
 |F_1^\uns(\delta\vzeta,\delta\het(\vv),\delta,\delta\param)|&\leq&\left\{\begin{array}{ll} \delta^3K|\het(\vv)-1| & \quad\textrm {if } \vv\in \doutbinf[]{},\medskip\\ \delta^3 K|\het(\vv)-1|^3 & \quad\textrm {if } \vv\in \doutbT[]{},\end{array}\right.
\end{eqnarray*}
and the claim is proved.
\end{proof}

\begin{lem}\label{cosexp3}
Let $\vw\in\doutb[\kappa][]{}$. Then:
 \begin{enumerate}
 \item\label{lemacosh} If $w\in \doutbinf[]{}$, one has:
$$|\cosh \vw|\geq\frac{e^{|\re \vw|}}{4}.$$
  \item\label{lemaexp} If $\vw\in \doutb[\dist][]{}$, then:
    $$|e^{\pm ic\log((1+e^{2\vw})/2)}|<e^{c\pi}.$$
 \end{enumerate}
\end{lem}
%
%

\begin{lem}\label{lemamaxmin}
 There exist constants $K_1$, $K_2$, $K_3$ and $K_4$, independent of $\delta$, such that
 \begin{enumerate}
  \item If $\vw\in \doutbT[]{}$ and $\im\vw\geq0$, then:
    \begin{enumerate}
     \item\label{apa} $K_1|\vw-i\pi/2|\leq|\cosh\vw|\leq K_2|\vw-i\pi/2|,$
     \item $K_3|\vw-i\pi/2|\leq|\het(\vw)-1|^{-1}\leq K_4|\vw-i\pi/2|,$
    \end{enumerate}
  \item If $\vw\in \doutbT[]{}$ and $\im\vw\leq0$, then:
    \begin{enumerate}
     \item $K_1|\vw+i\pi/2|\leq|\cosh\vw|\leq K_2|\vw+i\pi/2|,$
     \item $K_3|\vw+i\pi/2|\leq|\het(\vw)-1|^{-1}\leq K_4|\vw+i\pi/2|,$
    \end{enumerate}
 \end{enumerate}
\end{lem}

%

\begin{lem}\label{lematriangles}
 If $\vv\in \doutb[\dist][]{}$ and $\vw=\vv+re^{i(\pi-s)}$, with $r\in\mathbb{R}$, $r\geq0$ and $s\in\left(0,\beta/2\right]$, then there exists a constant $K\neq0$ independent of $\delta$ such that:
  $$|\vw\pm i\pi/2|\geq K|\vv\pm i\pi/2|.$$
\end{lem}

\begin{lem}\label{lemaz0}
 If $\vv\in \doutb[\dist][]{}$ and $\vw=\vv+re^{i(\pi-s)}$, with $r\in\mathbb{R}$, $r\geq0$ and $s\in\left(0,\beta/2\right]$, then there exists a constant $K$
independent of $\delta$ such that:
  \begin{enumerate}
   \item\label{lemaz0part1}
    \begin{enumerate}
\item\label{lemaz0p1pa} $|\cosh\vv|\leq K|\cosh\vw|.$
\item\label{lemaz0p1pb} Moreover, if $\vw\in \doutbinf[]{}$ then:
  $$\frac{|\cosh\vv|}{|\cosh\vw|}\leq Ke^{-|r\cos(\pi-\beta/2)|}.$$
\end{enumerate}
   \item\label{lemaz0part2} $|\het(\vw)-1|\leq K|\het(\vv)-1|.$
  \end{enumerate}
\end{lem}

\begin{lem}\label{lemacanvicami}
 Let $R>0$ be a constant big enough, and $\vv\in \doutb[\dist][]{}$. We define the complex path:
\begin{equation}\label{path1}
\Gamma^R_1=\{\vw\in\mathbb{C}:\,\vw=re^{i(\pi-\beta/2)},\,r\in[0,R]\},
\end{equation}
Then, if $\alpha $, $c$, $\delta>0$, the linear operator $L_{\alpha,c}$ defined in \eqref{operadorLalphac} can be rewritten as:
\begin{equation*}
 L_{\alpha ,c}(\phi)=-\lim_{R\to+\infty}\int_{\Gamma^R_1}{f_c(\vv, \vw)\phi(\vv+\vw)d\vw},
\end{equation*}
where $\phi\in\bsout[]{}$ and:
\begin{equation}\label{fc}
 f_c(\vv, \vw)=\frac{\cosh^\coef\vv}{\cosh^\coef(\vv+\vw)}e^{i\alpha \vw/\delta}e^{\param\vw}e^{ic\left[\vw+\log((1+e^{2\vv})/2)-\log((1+e^{2(\vv+\vw)})/2)\right]}
\end{equation}

\end{lem}

\begin{rmk}
 For  $L_{-\alpha ,-c}(\phi)$ we get the same result but in curves of the form $\overline{\Gamma}^R_1:=\left\{\vw\in\mathbb{C}\,:\,\overline{\vw}\in\Gamma^R_1\right\}$.
\end{rmk}

With these previous lemmas we can prove the following proposition, which characterizes how the operator $L=(L_{\alpha,c},L_{-\alpha,-c})$, defined in \eqref{operadorL}, acts on $\bsout[]{}\times\bsout[]{}$.
\begin{lem}\label{lemaL}
 The operator $L:\bsout[]{}\times\bsout[]{}\rightarrow\bsout[]{}\times\bsout[]{}$ is well defined and there exists a constant $K$ independent of $\delta$ such that for all $\phi\in\bsout[]{}\times\bsout[]{}$:
$$\|L(\phi)\|_{\textup{out},\times} \leq K\delta\|\phi\|_{\textup{out},\times} .$$
\end{lem}
\begin{proof}
 We just need to bound $\|L_{\alpha ,c}(\phi)\|_{\textup{out}}$, since the case for $\|L_{-\alpha ,-c}(\phi)\|_{\textup{out}}$ is completely analogous. Note that by Lemma \ref{lemacanvicami} we have that:
$$|L_{\alpha ,c}(\phi)(\vv)|=\left|\lim_{R\to+\infty}\int_{\Gamma^R_1}f_c(\vv, \vw)\phi(\vv+\vw)d\vw\right|,$$
where $\Gamma_1^R$ was defined in \eqref{path1} and $f_c$ was defined in \eqref{fc}. Now, parameterizing the curve $\Gamma^R_1$ by $\gamma(r)=re^{i(\pi-\beta/2)}$, with $r\in[0,R]$, we get:
\begin{eqnarray}
 |L_{\alpha ,c}(\phi)(\vv)|&=&\left|\lim_{R\to+\infty}\int_0^R e^{i(\pi-\beta/2)}f_c(\vv,re^{i(\pi-\beta/2)})\phi(\vv+re^{i(\pi-\beta/2)})dr\right|\nonumber\\
&=&\left|\cosh^\coef\vv\int_0^{+\infty}\frac{e^{i(\pi-\beta/2)}e^{\frac{i\alpha  re^{i(\pi-\beta/2)}}{\delta}}e^{\param re^{(\pi-\beta/2)}}}{\cosh^\coef(\vv+re^{i(\pi-\beta/2)})}\tilde g_c(\vv, r)\phi(\vv+re^{i(\pi-\beta/2)})dr\right|,\nonumber
\end{eqnarray}
where:
$$\tilde g_c(\vv, r)=g_c(v, re^{i(\pi-\beta/2)})=e^{ic(re^{i(\pi-\beta/2)}+\log((1+e^{2\vv})/2)-\log((1+e^{2(\vv+re^{i(\pi-\beta/2)})})/2))}.$$
First we will see that there exists a constant $K$ such that:
\begin{equation}\label{fitacoshe}
 \left|\frac{\cosh^\coef\vv \,e^{\param re^{(\pi-\beta/2)}}}{\cosh^\coef(\vv+re^{i(\pi-\beta/2)})}\right|\leq K.
\end{equation}
On one hand, if $re^{i(\pi-\beta/2)}\in \doutbinf[]{}$ then by part \ref{lemaz0p1pb} of Lemma \ref{lemaz0} we have that:
\begin{equation*}
 \left|\frac{\cosh^\coef\vv\,e^{\param re^{(\pi-\beta/2)}}}{\cosh^\coef(\vv+re^{i(\pi-\beta/2)})}\right|\leq Ke^{-\coef|r\cos(\pi-\beta/2)|}|e^{\param re^{(\pi-\beta/2)}}| \leq Ke^{(|\param|-\coef)|r\cos(\pi-\beta/2)|}\leq K,
\end{equation*}
because $|\param|-\coef<0$. On the other hand, if $re^{i(\pi-\beta/2)}\in \doutbT[]{}$ it implies that $r\leq r^*$ for some $r^*<+\infty$ independent of $\delta$. Then, by part \ref{lemaz0p1pa} of Lemma \ref{lemaz0}, we have that:
\begin{equation*}
\left|\frac{\cosh^\coef\vv\,e^{\param re^{(\pi-\beta/2)}}}{\cosh^\coef(\vv+re^{i(\pi-\beta/2)})}\right|\leq K|e^{\param re^{(\pi-\beta/2)}}|\leq Ke^{|\param||r\cos(\pi-\beta/2)|}\leq Ke^{|\param||r^*\cos(\pi-\beta/2)|}\leq Ke^{d|r^*\cos(\pi-\beta/2)|}.
\end{equation*}
This finishes the proof of \eqref{fitacoshe}.

Now, to bound $\tilde g_c(\vv, r)$ we just use item \ref{lemaexp} of Lemma \ref{cosexp3}:
\begin{equation}\label{fitaexpic}
|\tilde g_c(\vv, r)|=|e^{ic(re^{i(\pi-\beta/2)}+\log(1+e^{2\vv})-\log(1+e^{2(\vv+re^{i(\pi-\beta/2)})}))}|\leq e^{-c\im re^{i(\pi-\beta/2)}}e^{2c\pi}=e^{-cr\sin(\beta/2)}\leq e^{2c\pi}.
\end{equation}

Hence, using bounds \eqref{fitacoshe} and \eqref{fitaexpic} we have, for $\vv\in\doutb[\kappa]{}$
\begin{eqnarray}
&&|L_{\alpha ,c}(\phi)(\vv)|\leq Ke^{2c\pi}\int_0^{+\infty}\left|e^{i\alpha  re^{i(\pi-\beta/2)}/\delta}\right||\phi(\vv+re^{i(\pi-\beta/2)})|dr.
\end{eqnarray}
Now we distinguish between the cases $\vv\in\doutbinf[]{}$ and $\vv\in\doutbT[]{}$. On one hand, if $\vv\in \doutbinf[]{}$ we have $\vv+re^{i(\pi-\beta/2)}\in \doutbinf[]{}$ and then by part \ref{lemaz0part2} of Lemma \ref{lemaz0}:
\begin{eqnarray}\label{integralD1}
|L_{\alpha ,c}(\phi)(\vv)|&\leq& Ke^{2c\pi}\|\phi\|_{\textup{out},\times} \int_0^{+\infty}e^{-\alpha  r\sin(\pi-\beta/2)/\delta}|\het(\vv+re^{i(\pi-\beta/2)})-1| dr\nonumber\\
&\leq& Ke^{2c\pi}\|\phi\|_{\textup{out},\times} |\het(\vv)-1|\int_0^{+\infty}e^{-\alpha  r\sin(\pi-\beta/2)/\delta} dr.
\end{eqnarray}
On the other hand, if $\vv\in \doutbT[]{}$, let $r^*$ be the value such that $\vv+re^{i(\pi-\beta/2)}\in \doutbinf[]{}\cap \doutbT[]{}$. Then:
\begin{eqnarray}
|L_{\alpha ,c}(\phi)(\vv)|&\leq& Ke^{2c\pi}\|\phi\|_{\textup{out},\times} \left(\int_0^{r^*}e^{-\alpha  r\sin(\pi-\beta/2)/\delta}|\het(\vv+re^{i(\pi-\beta/2)})-1|^3dr\right.\nonumber\\
&&+\left.\int_{r^*}^{+\infty}e^{-\alpha  r\sin(\pi-\beta/2)/\delta}|\het(\vv+re^{i(\pi-\beta/2)})-1| dr\right)\nonumber\\
&\leq& Ke^{2c\pi}\|\phi\|_{\textup{out},\times} \left(\int_0^{r^*}e^{-\alpha  r\sin(\pi-\beta/2)/\delta}|\het(\vv)-1|^3dr\right.\nonumber\\
&&+\left.\int_{r^*}^{+\infty}e^{-\alpha  r\sin(\pi-\beta/2)/\delta}|\het(\vv)-1| dr\right),\nonumber
\end{eqnarray}
where we have used part \ref{lemaz0part2} of Lemma \ref{lemaz0} again. Now, since for $\vv\in \doutbT[]{}$ we have that $|\het(\vv)-1|\leq K|\het(\vv)-1|^3$, this last inequality yields:
\begin{equation}\label{integralD2D3}
|L_{\alpha ,c}(\phi)(\vv)|\leq Ke^{2c\pi}\|\phi\|_{\textup{out},\times}|\het(\vv)-1|^3\int_0^{\infty}e^{-\alpha  r\sin(\pi-\beta/2)/\delta}dr.
\end{equation}
Hence, from \eqref{integralD1} and \eqref{integralD2D3} we can write:
$$|L_{\alpha ,c}(\phi)(\vv)|\leq Ke^{2c\pi}\|\phi\|_{\textup{out},\times}|\het(\vv)-1|^\nu\int_0^{\infty}e^{-\alpha  r\sin(\pi-\beta/2)/\delta} dr,$$
where $\nu=1$ if $\vv\in \doutbinf[]{}$ and $\nu=3$ otherwise.

If we compute the last integral explicitly we get that:
$$|L_{\alpha ,c}(\phi)(\vv)|\leq\delta \frac{Ke^{2c\pi}}{\alpha \sin(\beta/2)}\|\phi\|_{\textup{out},\times} |\het(\vv)-1|^\nu,$$
and then, by definition \eqref{defnorm} of the norm $\|.\|_{\textup{out}}$, the result is clear.
\end{proof}

With Lemma \ref{lemaL}, the first part of Proposition \ref{propouterunbounded} will be proved. Concretely, we will demonstrate the following:
\begin{lem}\label{lemaprimaprox}
 The function $\vzeta_0=L\circ R(0)$, where $R$ was defined in \eqref{defR} and $L$ in \eqref{operadorL}, belongs to $\bsout[]{}\times\bsout[]{}$, and there exists a constant $K$ independent of $\delta$ such that:
$$\|\vzeta_0\|_{\textup{out},\times} \leq  K\delta^2.$$
\end{lem}
\begin{proof}
 By Lemma \ref{lemaL} it is clear that we just need to prove that $\|R(0)\|_{\textup{out},\times} \leq K\delta.$ Again, we will just bound the norm of the first component of $R(0)$, that is $R_1(0)$, being the second one analogous.

By \eqref{defR} we have:
$$|R_1(0)(\vv)|=\frac{\delta^{-2}|F^\uns_1(0,\delta\het(\vv),\delta,\delta\param)|}{\left|1+\frac{\delta^{-2}H^\uns(0,\delta\het(\vv),\delta,\delta\param)}{-1+\het^2(\vv)}\right|}\leq\frac{\delta^{-2}|F^\uns_1(0,\delta\het(\vv),\delta,\delta\param)|}{\left|1-\left|\frac{\delta^{-2}H^\uns(0,\delta\het(\vv),\delta,\delta\param)}{-1+\het^2(\vv)}\right|\right|}.$$
First we will prove that, for $\vv\in\doutb[\kappa][]{}$:
\begin{equation}\label{divH0}
 \frac{1}{\left|1-\left|\frac{\delta^{-2}H^\uns(0,\delta\het(\vv),\delta,\delta\param)}{-1+\het^2(\vv)}\right|\right|}\leq2.
\end{equation}
Indeed, if $\vv\in \doutbinf[]{}$, by Corollary \ref{corollaryfitaFiH}:
\begin{equation}\label{part1}
 \left|\frac{\delta^{-2}H^\uns(0,\delta\het(\vv),\delta,\delta\param)}{-1+\het^2(\vv)}\right|\leq \frac{ K\delta|\het(\vv)-1|}{|-1+\het^2(\vv)|}=2 K\delta|e^{\vv}\cosh\vv|\leq  K\delta <\frac{1}{2},
\end{equation}
where we have used that $2e^v\cosh\vv=e^{2\vv}+1$ is bounded in $\doutbinf[]{}$ and that $\delta$ is sufficiently small. Otherwise, if $\vv\in \doutbT[]{}$, again by Corollary \ref{corollaryfitaFiH} we have:
$$\left|\frac{\delta^{-2}H^\uns(0,\delta\het(\vv),\delta,\delta\param)}{-1+\het^2(\vv)}\right|\leq \frac{ K\delta|\het(\vv)-1|^3}{|-1+\het^2(\vv)|}=\frac{8 K\delta e^{3\vv}}{|\cosh\vv|}.$$
Now, using Lemma \ref{lemamaxmin} we have:
$$\frac{1}{|\cosh\vv|}\leq\frac{1}{K_1|\vv\mp i\pi/2|}\leq\frac{1}{K_1\delta\log(1/\delta)},$$
since $|\vv\mp i\pi/2|\geq  K\delta\log(1/\delta)$ in $\doutb[\dist][]{}$. Moreover, for $\vv\in \doutbT[]{}$ it is clear that $e^{3\vv}$ is bounded. Therefore it is straightforward to see that:
\begin{equation}\label{part2}
 \left|\frac{\delta^{-2}H^\uns(0,\delta\het(\vv),\delta,\delta\param)}{-1+\het^2(\vv)}\right|\leq\frac{K}{\log(1/\delta)}<\frac{1}{2}
\end{equation}
if $\delta$ is small enough. Then, from \eqref{part1} and from \eqref{part2}, bound \eqref{divH0} holds true.

Finally, from \eqref{divH0} and using again Corollary \ref{corollaryfitaFiH} it is clear that:
$$|R_1(0)(\vv)|\leq2|\delta^{-2}F_1^\uns(0,\delta\het(\vv),\delta,\delta\param)|\leq \left\{\begin{array}{ll}  K\delta|\het(\vv)-1|& \quad\textrm {if } \vv\in \doutbinf[]{},\medskip\\\displaystyle  K\delta|\het(\vv)-1|^3, & \quad\textrm {if }\vv\in \doutbT[]{},\end{array}\right.$$
and then from the definition \eqref{defnorm} of the norm $\|.\|_{\textup{out}}$ we obtain the statement immediately.
\end{proof}

We enunciate the following technical lemma, which due to Angenent \cite{ang} will simplify the proof of the second part of Proposition \ref{propouterunbounded}.
\begin{lem}[\cite{ang}]\label{lemaang}
 Let $E$ be a complex Banach space, and let $f:B_r\rightarrow B_{\theta r}$ be a holomorphic mapping, where $B_\rho=\{x\in E:\|x\|<\rho\}$.

If $\theta<1/2$, then $f_{|_{B_{\theta r}}}$ is a contraction, and hence has a unique fixed point in $B_{\theta r}$.
\end{lem}

The following result will allow us to finish the proof of Proposition \ref{propouterunbounded}.
\begin{lem}\label{lemafixedpt}
 Let $\CMcal{F}:=L\circ R$ and $B(r)$ be the ball of $\bsout[]{}\times\bsout[]{}$ centered at the origin of radius
$r=8\|\vzeta_0\|_{\textup{out},\times} $. Then, $\CMcal{F}:B(r)\rightarrow B(r/4)$ is well defined. Moreover, there exists a constant $K$ independent of $\delta$
such that if $\vzeta\in B(r)$:
$$\|\CMcal{F}(\vzeta)-\vzeta_0\|_{\textup{out},\times} \leq\frac{1}{\log(1/\delta)}K\|\vzeta\|_{\textup{out},\times} .$$
\end{lem}
\begin{proof}
 Note that it is sufficient to prove the inequality. Indeed, suppose that it holds, then taking $\vzeta\in B(r)$ and $\delta$ sufficiently small we have:
\begin{eqnarray}
 \|\CMcal{F}(\vzeta)\|_{\textup{out},\times} \leq\|\CMcal{F}(\vzeta)-\vzeta_0\|_{\textup{out},\times} +\|\vzeta_0\|_{\textup{out},\times} \leq\frac{1}{\log(1/\delta)}K\|\vzeta\|_{\textup{out},\times} +\|\vzeta_0\|_{\textup{out},\times} \leq\frac{1}{8}r+\frac{1}{8}r=\frac{1}{4}r,\nonumber
\end{eqnarray}
that is $\CMcal{F}(\vzeta)\in B(r/4)$.

Now, recall that:
$$\CMcal{F}(\vzeta)-\vzeta_0=L\circ R(\vzeta)-L\circ R(0)=L\circ(R(\vzeta)-R(0)),$$
where $R$ was defined in \eqref{defR}. In order to make the proof clearer, we will decompose $R$ as:
$$R(\vzeta)(\vv)=S(\vzeta)(\vv)+T(\vzeta)(\vv)\cdot\vzeta,$$
where:
\begin{equation*}
 S(\vzeta)(\vv)= \frac{\displaystyle\delta^{-2}F^\uns(\delta\vzeta,\delta\het(\vv),\delta,\delta\param)}{\displaystyle1+\frac{\displaystyle b\veta\vetab+\delta^{-2}H^\uns(\delta\vzeta, \delta \het(\vv), \delta, \delta\param)}{\displaystyle-1+\het^2(\vv)}},\,
 T(\vzeta)(\vv)= \left(\frac{\displaystyle 1}{\displaystyle1+\frac{\displaystyle b\veta\vetab+\delta^{-2}H^\uns(\delta \vzeta, \delta \het(\vv), \delta, \delta\param)}{\displaystyle-1+\het^2(\vv)}}-1\right)A(\vv).
\end{equation*}
Then we have that:
$$R(\vzeta)(\vv)-R(0)(\vv)=S(\vzeta)(\vv)-S(0)(\vv)+T(\vzeta)(\vv)\cdot\vzeta.$$
Now we shall bound these two last terms separately. We will begin by $S(\vzeta)-S(0)$, and we will do it using the mean value theorem:
$$S(\vzeta)(\vv)-S(0)(\vv)=\int_0^1DS(\lambda\vzeta)(\vv)d\lambda\cdot\vzeta.$$
So we just need to bound $DS(\lambda\vzeta)(\vv)$ with $\lambda\in[0,1]$. We claim that:
\begin{equation}\label{fitaDS}|DS(\lambda\vzeta)(\vv)|\leq \frac{K}{\delta\log^2(1/\delta)}.\end{equation}
To prove that, we introduce the auxiliary function:
$$\tilde S(\phi)= \frac{\displaystyle\delta^{-2}F^\uns(\phi_1,\phi_2,\phi_3,\phi_4,\phi_5)}{\displaystyle1+\frac{\displaystyle \delta^{-2}\left(b\phi_1\phi_2+H^\uns(\phi_1,\phi_2,\phi_3,\phi_4,\phi_5)\right)}{\displaystyle-1+\delta^{-2}\phi_3^2}}.$$
Observe that one has:
\begin{equation}\label{dSdtildeS}
 DS(\lambda\vzeta)(\vv)=\delta\left(\partial_{\phi_1}\tilde S(\phi),\partial_{\phi_2}\tilde S(\phi)\right).
\end{equation}

Now, it is clear that $\tilde S$ is analytic in the open set:
$$\CMcal{U}:=B^3(r_0^\uns)\times B(\delta_0)\times B(\param_0)\cap\left\{(\phi_1,\phi_2,\phi_3,\phi_4,\phi_5)\in\mathbb{C}^3\times\mathbb{R}^2\,:\, \left|\frac{\displaystyle \delta^{-2}\left(b\phi_1\phi_2+H^\uns(\phi_1,\phi_2,\phi_3, \phi_4, \phi_5)\right)}{\displaystyle-1+\delta^{-2}\phi_3^2}\right|<\frac{1}{2}\right\}.$$
For $\delta$ small enough this set is not empty since:
$$(0,0,0,\delta,0)\in B^3(r_0^\uns)\times B(\delta_0)\times B(\param_0),$$
and, since $H^\uns(x,y,z,\delta,\overline{\param})=O_3(x,y,z-\delta,\delta,\overline{\param})$ and is analytic in $ B^3(r_0^\uns)\times B(\delta_0)\times B(\param_0)$, we have:
$$|\delta^{-2}H^\uns(0,0,0,\delta,0)|\leq K\delta<\frac{1}{2}.$$
Let $0<\tilde r_0^\uns\leq r_0^\uns$ such that $B^3(\tilde r_0^\uns)\times B(\delta_0)\times B(\param_0)\subset \CMcal{U}$. Then, by definition of $\CMcal{U}$, the fact that $F^\uns$ is analytic in $\CMcal{U}$ and that $F^\uns(x,y,z,\delta,\overline{\param})=O_3(x,y,z-\delta,\delta,\overline{\param})$, it is clear that:
$$|\tilde S(\phi)|\leq 2\delta^{-2}|F^\uns(\phi_1,\phi_2,\phi_3,\phi_4,\phi_5)|\leq \delta^{-2}K|(\phi_1,\phi_2,\phi_3-\delta,\phi_4,\phi_5)|^3.$$
Then taking $\phi\in B^3(\tilde r_0^\uns/2)\times B(\delta_0/2)\times B(\param_0/2)$ and $\phi^*=(0,0,\delta,0,0)$, since we have that $\phi-\phi^*\in B^3(\tilde r_0^\uns)\times B(\delta_0)\times B(\param_0)$, if $\delta$ is small enough. Hence, by Lemma \ref{lemafitaDFiDH} we obtain that:
\begin{equation}\label{fitapartialtildeS}
 \left|\partial_{\phi_i}\tilde S_j(\phi)\right|\leq\delta^{-2}K|\phi-\phi^*|^2=|(\phi_1,\phi_2,\phi_3-\delta,\phi_4,\phi_5)|^2,\qquad\qquad i,j=1,2
\end{equation}
where $\tilde S_j$ denotes the j-th component of $\tilde S$.

Note that we can apply \eqref{fitapartialtildeS} to $\phi=(\delta\lambda\vzeta,\delta\het(\vv),\delta,\delta\param)$. Indeed, since $\|\lambda\vzeta\|_{\textup{out},\times}\leq\|\vzeta_0\|_{\textup{out},\times}\leq K\delta^2$, it can be easily seen that for $\delta$ small enough and $\vv\in\doutb{\uns}$, we have that $(\delta\lambda\vzeta,\delta\het(\vv),\delta,\delta\param)\in B^3(\tilde r_0^\uns/2)\times B(\delta_0/2)\times B(\param_0/2)$, for all $\lambda\in[0,1]$. Then, since $|\delta|,|\delta\param|,|\delta\lambda\vzeta|\leq K\delta|\het(\vv)-1|$, we obtain:
\begin{equation*}
 \left|\partial_{\phi_i}\tilde S_j(\delta\lambda\vzeta,\delta\het(\vv),\delta,\delta\param)\right|\leq\delta^{-2}K|(\delta\lambda\vzeta,\delta(\het(\vv)-1),\delta,\delta\param)|^2\leq K|\het(\vv)-1|^2\leq\frac{K}{\delta^2\log^2(1/\delta)}.
\end{equation*}
Using this last bound in equality \eqref{dSdtildeS}, bound \eqref{fitaDS} is clear. Then, \eqref{fitaDS} and the mean value theorem yield:
$$|S(\vzeta)(\vv)-S(0)(\vv)|\leq\int_0^1|DS(\lambda\vzeta,\vv)|d\lambda\cdot|\vzeta(\vv)|\leq\frac{K}{\delta\log^2(1/\delta)}|\vzeta(\vv)|,$$
that is:
\begin{equation}\label{fitadeltaS}
 \|S(\vzeta)-S(0)\|_{\textup{out},\times} \leq\frac{K}{\delta\log^2(1/\delta)}\|\vzeta\|_{\textup{out},\times} .
\end{equation}

Now we shall proceed to bound $T(\vzeta)(\vv)\cdot\vzeta$. Recall that:
\begin{equation}\label{fitanormabola}
 |\vzeta(\vv)|\leq\|\vzeta\|_{\textup{out},\times}|\het(\vv)-1|^\nu\leq r|\het(\vv)-1|^\nu=8\|\vzeta_0\|_{\textup{out},\times}|\het(\vv)-1|^\nu\leq K\delta^2|\het(\vv)-1|^\nu,
\end{equation}
where $\nu=1$ if $\vv\in \doutbinf[]{}$ and $\nu=3$ otherwise. Then, on the one hand,  Corollary \ref{corollaryfitaFiH} and the fact that for $\vv\in\doutbinf[]{}$ we have $|\het(\vv)+1|\geq K$ yield:
\begin{eqnarray}\label{fitau}
 \left|\frac{b\veta\vetab+\delta^{-2}H^\uns(\delta\vzeta,\delta\het(\vv),\delta,\delta\param)}{-1+\het^2(\vv)}\right|&\leq&K\frac{\delta^4|\het(\vv)-1|^2+\delta|\het(\vv)-1|}{|-1+\het^2(\vv)|}\nonumber\\
&\leq& K\left(\delta^4|\het(\vv)-1|^2+\delta\right)\leq K\delta<\frac{1}{2},
\end{eqnarray}
if $\delta$ is small enough. On the other hand, if $\vv\in \doutbT[]{}$, by Corollary \ref{corollaryfitaFiH} we obtain that:
\begin{eqnarray}\label{fitados}
 \left|\frac{b\veta\vetab+\delta^{-2}H^\uns(\delta\vzeta,\delta\het(\vv),\delta,\delta\param)}{-1+\het^2(\vv)}\right|&\leq&K\left|\frac{\delta^4|\het(\vv)-1|^6+\delta|\het(\vv)-1|^3}{-1+\het^2(\vv)}\right|\leq K\left(\delta^4|\het(\vv)-1|^4+\delta|\het(\vv)-1|\right)\nonumber\\
&\leq&\frac{K}{\log(1/\delta)},
\end{eqnarray}
for $\delta$ small enough. Hence, from \eqref{fitau} and \eqref{fitados}, using Lemma \ref{lemadifdiv} it is clear that for all $\vv\in\doutb[\dist][]{}$ we have:
\begin{equation*}
 \left|\frac{1}{1+\frac{b\veta\vetab+\delta^{-2}H^\uns(\delta\vzeta,\delta\het(\vv),\delta,\delta\param)}{-1+\het^2(\vv)}}-1\right|\leq \frac{K}{\log(1/\delta)}.
\end{equation*}
Moreover, by the definition of the matrix $A$ it is straightforward to see that $|A(\vv)\vzeta(\vv)|\leq K\delta^{-1}|\vzeta(\vv)|,$ and hence:
\begin{equation}\label{fitaTzeta}
 \|T(\vzeta)\cdot\vzeta\|_{\textup{out},\times} \leq\frac{K}{\delta\log(1/\delta)}\|\vzeta\|_{\textup{out},\times} .
\end{equation}

In conclusion, using \eqref{fitadeltaS} and \eqref{fitaTzeta} it is clear that:
\begin{eqnarray*}
 \|R(\vzeta)-R(0)\|_{\textup{out},\times} &=& \|S(\vzeta)-S(0)-T(\vzeta)\cdot\vzeta\|_{\textup{out},\times} \leq K\|\vzeta\|_{\textup{out},\times} \left(\frac{1}{\delta\log^2(1/\delta)}+\frac{1}{\delta\log(1/\delta)}\right)\\
&\leq&\frac{K}{\delta\log(1/\delta)}\|\vzeta\|_{\textup{out},\times} .
\end{eqnarray*}
Finally, by Lemma \ref{lemaL} we obtain the desired bound:
$$\|L\circ(R(\vzeta)-R(0))\|_{\textup{out},\times} \leq K\delta\|R(\vzeta)-R(0)\|_{\textup{out},\times} \leq \frac{K}{\log(1/\delta)}\|\vzeta\|_{\textup{out},\times} .$$
\end{proof}

\begin{proof}[End of the proof of Proposition \ref{propouterunbounded}]
 As we already mentioned, the first part of Proposition \ref{propouterunbounded} is proved in Lemma \ref{lemaprimaprox}.

On the other hand, note that Lemma \ref{lemaang} and Lemma \ref{lemafixedpt} imply that the operator $\CMcal{F}=L\circ R$ has a unique fixed point $\vzeta^\uns$ in the ball of $\bsout[]{}\times\bsout[]{}$ of radius $8\|\vzeta_0^\uns\|_{\textup{out},\times}$. Then we just need to define $\vzeta_1^\uns=\vzeta^\uns-\vzeta_0^\uns$. It is clear that $\vzeta^\uns=\vzeta_0+\vzeta_1$ and that by Lemma \ref{lemafixedpt}:
\begin{equation*}
\|\vzeta_1^\uns\|_{\textup{out},\times}=\|\vzeta^\uns-\vzeta_0^\uns\|_{\textup{out},\times}=\|\CMcal{F}(\vzeta^\uns)-\CMcal{F}(0)\|_{\textup{out},\times}\leq \frac{K}{\log(1/\delta)}\|\vzeta^\uns\|_{\textup{out},\times}\leq \frac{K}{\log(1/\delta)}\|\vzeta_0^\uns\|_{\textup{out},\times},
\end{equation*}
and then the second part of Proposition \ref{propouterunbounded} is clear.
\end{proof}


\section{Proof of Theorem \ref{theoremouter}}\label{prooftheoremouterbounded}
Again, we will just focus on the proof for the unstable manifold, $\vphi^\uns$, being the one for the stable manifold analogous. We will also omit the superindices $-\uns-$ whenever it does not lead to confusion.
\begin{lem}\label{lemaVpetita}
 Let $V_\pm(\vu,\delta,\param)=\het^{-1}(\het(\vu)-z_\pm(\delta,\param)\pm1)-\vu$, where $z_\pm(\delta,\param)$ is the third component of the critical point $S_\pm(\delta,\param)$. Then, for all $\vu\in \doutT[]{}$, there exists a constant $C_{V_\pm}$ independent of $\delta$ such that:
$$|V_\pm(\vu,\delta,\param)|\leq\delta C_{V_\pm}.$$
Moreover, given any constant $\dist$, if $\vu\in \doutT[]{}$, then for $\delta>0$ sufficiently small:
$$u+V_\pm(\vu,\delta,\param)\in D_{\dist/2,\beta}^{\textup{out}}.$$
\end{lem}
\begin{proof}
Consider the function $f(t):=\het^{-1}(\het(\vu)+t(-z_\pm(\delta,\param)\pm1))$. It is clear that $V_\pm(\vu,\delta,\param)=f(1)-f(0)$. Moreover, for any $\vu\in \doutb[]{}$ and $\delta>0$, the function $f$ is analytic. Using that:
$$\left|\frac{1}{-1+\het^2(\vu)}\right|=\left|\cosh^2\vu\right|\leq M\qquad \text{if }\vu\in\doutT[]{},$$
and that, by Lemma \ref{lemaptscritics}, $|-z_\pm(\delta,\param)\pm1|\leq K\delta$, one can easily see that $|f'(t)|\leq \delta C_{V_\pm}$. Then, by the mean value theorem, the first part of the lemma is proved. Moreover, using the bound of $V_\pm(\vu,\delta,\param)$ it is straightforward to check that the second part of the lemma also holds.

\end{proof}

\begin{proof}[End of the proof of Theorem \ref{theoremouter}]
We just need to take $\overline{\dist}=\dist/2$. Then, by Lemma \ref{lemaVpetita}, if $u\in \doutT[]{}$, $u+V_\pm(\vu,\delta,\param)$ belongs to $D_{\dist/2,\beta}^{\textup{out}}=\doutb[\dist][]{},$ where we know by Proposition \ref{propouterunbounded} that the parameterizations $\vzeta^\uns$ and $\vzeta^\sta$ are defined. Then we just have to define $\vphi^\sta$ and $\vphi^\uns$ as:
\begin{equation}
\begin{array}{rcl}
  \vphi^\sta(\vu)&=&\vzeta^\sta(\vu+V_-(\vu,\delta,\param))+\vzeta_-(\delta,\param),\\
  \vphi^\uns(\vu)&=&\vzeta^\uns(\vu+V_+(\vu,\delta,\param))+\vzeta_+(\delta,\param),
\end{array}\qquad \vu\in \doutT[]{}
\end{equation}
where $\vzeta_\pm(\delta,\param)=(\veta_\pm(\delta,\param),\vetab_\pm(\delta,\param))$, and $\veta_\pm$, $\vetab_\pm$ were defined in \eqref{defvetaz+-}. As we pointed out in Subsection \ref{localparams}, both $\vphi^\sta(\vu)$ and $\vphi^\uns(\vu)$ satisfy system \eqref{sys2d}, and that they are parameterizations of the stable and unstable manifolds of $S_-(\delta,\param)$ and $S_+(\delta,\param)$ respectively.

Finally, note that, for $\vu\in \doutT[]{}$, one has:
$$|\vphi^{\uns,\sta}(\vu)||\het(\vu)-1|^3\leq|\vzeta^{\uns,\sta}(\vu+V_\pm(\vu,\delta,\param))||\het(\vu)-1|^3+|\vzeta_\pm(\delta,\param)||\het(\vu)-1|^3,$$
for some constant $K$. Now, on one hand, by Proposition \ref{propouterunbounded} and using that for $\vu\in\doutT[]{}$:
$$\left|\frac{\het(\vu)-1}{\het(\vu+V_\pm(\delta,\param))-1}\right|\leq\frac{K}{1-\log(1/\delta)},$$
we have:
\begin{equation*}
 |\vzeta^{\uns,\sta}(\vu+V_\pm(\vu,\delta,\param))||\het(\vu)-1|^3\leq K|\vzeta^{\uns,\sta}(\vu+V_\pm(\vu,\delta,\param))||\het(\vu+V_\pm(\vu,\delta,\param)-1|^3\leq K\|\vzeta^{\uns,\sta}\|_{\textup{out}}^{\uns,\sta}\leq  K\delta^2.
\end{equation*}
On the other hand, recall that $\vzeta_\pm(\delta,\param)=(\veta_\pm(\delta,\param),\vetab_\pm(\delta,\param))$, where $\veta_\pm(\delta,\param)=x_\pm(\delta,\param)+iy_\pm(\delta,\param),$ and then by Lemma \ref{lemaptscritics}, since $|\cosh\vu|$ is bounded in $\doutT[]{}$, we obtain $|\vzeta_\pm(\delta,\param)\cosh^3\vu|\leq  K\delta^2,$ and thus the last statement of Theorem \ref{theoremouter} is clear.
\end{proof}

\section{Sketch of the proof of Theorem \ref{theoreminner}}\label{sketchtheoreminner}
In this section we present the main ideas of how Theorem \ref{theoreminner} is proved. As we already mentioned, the proof is analogous as the one found in \cite{BaSe08}, and hence for more details we refer the reader to that paper.

\subsection{Existence of solutions $\Vpsi_0^{\uns,\sta}$}
First we will introduce the Banach spaces in which we will work. For $*=\uns,\sta$, we define:
$$\bsin{*}=\{\phi:\din{*}\rightarrow\mathbb{C},\, \phi \textrm{ analytic, }\|\phi\|_{\textup{in},\nu}^{\uns,\sta}:=\sup_{\vs\in\din{*}}|\vs^\nu\phi(\vs)|<\infty\}.$$
As usual, in the product space $\bsin{*}\times\bsin{*}$ we will take the norm:
\begin{equation}\label{defnormin}\|(\phi_1,\phi_2)\|_{\textup{in},\nu,\times}^{\uns,\sta}=\|\phi_1\|_{\textup{in},\nu}^{\uns,\sta}+\|\phi_2\|_{\textup{in},\nu}^{\uns,\sta}.\end{equation}

Now, if we call $\Vpsi=(\vpsi,\vpsib)$ the solutions of \eqref{sysinner}, $F=(F_1,F_2)$ and define:
\begin{eqnarray}
h_0&=&\lim_{\re s\to\infty}{s^{3}H(0,0,-s^{-1},0,0)},\label{defh0}\medskip\\
\hin(\Vpsi,\vs)&=&\vs^2\left[b\vpsi\vpsib+H(\Vpsi,-\vs^{-1},0,0)\right],\label{defhinner}\medskip\\
\Ain(\vs)&=&\left(\begin{array}{cc}
                 -(\alpha-c\vs^{-1})i+\coef\vs^{-1}& 0\bigskip\\
		    0 & (\alpha-c\vs^{-1})i+\coef\vs^{-1}
                \end{array}\right), \label{defAinner}
\end{eqnarray}
and
\begin{equation}\label{defRin}\Rin(\Vpsi)(\vs)=\left(\frac{1}{1+\hin(\Vpsi,\vs)}-\frac{1}{1+h_0\vs^{-1}}\right)\Ain(\vs)\Vpsi+\frac{F(\Vpsi,-\vs^{-1},0,0)}{1+\hin(\Vpsi,\vs)},\end{equation}
then system \eqref{sysinner} can be written as:
\begin{equation}\label{sysinner2dcompact}
 \frac{d\Vpsi}{ds}=\frac{1}{1+h_0\vs^{-1}}\Ain(\vs)\Vpsi+\Rin(\Vpsi)(\vs).
\end{equation}

\begin{lem}
 A fundamental matrix of the linear homogeneous system
$$\frac{d\Vpsi}{d\vs}=\frac{1}{1+h_0\vs^{-1}}\Ain(\vs)\Vpsi,$$
is:
\begin{equation}\label{matriuMinner}
\CMcal{M}(\vs)=\left(\begin{array}{cc} m_1(\vs) & 0\\ 0 & m_2(\vs)\end{array}\right)=\vs^\coef(1+h_0\vs^{-1})^\coef\left(\begin{array}{cc}e^{-i(\alpha\vs+\beta(\vs))} & 0\\ 0 & e^{i(\alpha\vs+\beta(\vs))}\end{array}
\right),
\end{equation}
where $\beta(\vs)=-(c+\alpha h_0)\log(\vs(1+h_0\vs^{-1}))$.
\end{lem}

The functional equation that $\Vpsi_0^{\uns,\sta}$ have to satisfy is the following:
\begin{equation}\label{fixedpointinner}
 \Vpsi_0^{\uns,\sta}(\vs)=\CMcal{M}(\vs)\int_{\mp\infty}^0\CMcal{M}(\vs+t)^{-1}\Rin(\Vpsi_0^{\uns,\sta})(\vs+t)dt,
\end{equation}
where $\Rin$ was defined in \eqref{defRin}, and $+\infty$ corresponds to the stable case and $-\infty$ to the unstable one. For functions $\Phi\in\bsin{*}\times\bsin{*}$, we introduce the linear operators:
$$\Bin^{\uns,\sta}(\Phi)(\vs)=\CMcal{M}(\vs)\int_{\mp\infty}^0\CMcal{M}(\vs+t)^{-1}\Phi(\vs+t)dt,$$
so that the fixed point equation \eqref{fixedpointinner} can be written as:
\begin{equation}\label{fixedpointinnercurta}
 \Vpsi_0^{\uns,\sta}=\CMcal{F}^{\uns,\sta}(\Vpsi_0^{\uns,\sta}):=\Bin^{\uns,\sta}\circ\Rin(\Vpsi_0^{\uns,\sta}).
\end{equation}

The main result in this subsection, which is equivalent to item 1 of Theorem \ref{theoreminner}, is the following:
\begin{prop}\label{propvarietatsinner}
 Given $\beta_0>0$, there exists $\rho>0$ big enough such that system \eqref{sysinner2dcompact} has two solutions $\Vpsi_0^{\uns,\sta}$ belonging to $\bsin{*}\times\bsin{*}$, $*=\uns,\sta$, of the form:
$$\Vpsi_0^{\uns,\sta}=\Vpsi_{0,0}^{\uns,\sta}+\Vpsi_{0,1}^{\uns,\sta},$$
with $\Vpsi_{0,0}^{\uns,\sta}=\Bin^{\uns,\sta}\circ\Rin(0)\in\bsin[3]{*}\times\bsin[3]{*}$, $\Vpsi_{0,1}^{\uns,\sta}\in\bsin[4]{*}\times\bsin[4]{*}$, satisfying $\|\Vpsi_{0,1}^{\uns,\sta}\|_{\textup{in},3,\times}^{\uns,\sta}<\|\Vpsi_{0,0}^{\uns,\sta}\|_{\textup{in},3,\times}^{\uns,\sta}$.

Moreover, these functions $\Vpsi_0^{\uns,\sta}$ are the unique solutions of \eqref{sysinner2dcompact} satisfying the asymptotic condition $\lim_{\re\vs\to\mp\infty}\Vpsi_0^{\uns,\sta}(\vs)=0$, where $-$ corresponds to $\uns$ and $+$ to $\sta$.
\end{prop}

This proposition is proved in \cite{BaSe08} in the case $\coef=1$, and the case $\coef\neq1$ can be proved identically.

\subsection{Asymptotic expression for the difference $\Delta\Vpsi_0$}
Below we sketch how formula \eqref{asyVpsi0} can be found, which is an adaptation of the results of \cite{BaSe08} for the case $\coef\neq1$. The first step is to realize that, since $\Vpsi_0^\sta$ and $\Vpsi_0^\uns$ satisfy equation \eqref{sysinner2dcompact}, its difference $\Delta\Vpsi_0=(\Delta\vpsi_0,\Delta\vpsib_0)$ satisfies the following homogeneous linear equation:
\begin{equation}\label{homoeqinner}
 \frac{d\Delta\Vpsi}{d\vs}=\left[\frac{1}{1+h_0\vs^{-1}}\Ain(\vs) +\Rinasy(\vs)\right]\Delta\Vpsi,
\end{equation}
where $\Rinasy$ is the matrix defined by:
$$\Rinasy(\vs)=\int_0^1D\Rin(\Vpsi_0^\sta(\vs)+\lambda(\Vpsi^\uns(\vs)-\Vpsi^\sta(\vs)))d\lambda,$$
and $\Rin$ was defined in \eqref{defRin}. As in \cite{BaSe08}, one deduces that any analytic solution of equation \eqref{homoeqinner} that is bounded in the domain $\ein$, defined in \eqref{defein}, can be written as the following integral equation.
\begin{eqnarray}
 \Delta\vpsi_0(\vs)&=&\vs^\coef(1+h_0\vs^{-1})^\coef e^{-i(\alpha\vs+\beta(\vs))}\left[\kappa_0+\int_{-i\rho}^{\vs}\frac{e^{i(\alpha t+\beta(t))}}{t^\coef(1+h_0t^{-1})^\coef}\langle \Rinasy_1(t),\Delta\Vpsi_0(t)\rangle dt\right],\nonumber\\\label{primeracompdeltavpsi}\\
 \Delta\vpsib_0(\vs)&=&\vs^\coef(1+h_0\vs^{-1})^\coef e^{i(\alpha\vs+\beta(\vs))}\int_{-i\infty}^{\vs}\frac{e^{-i(\alpha t+\beta(t))}}{t^\coef(1+h_0t^{-1})^\coef}\langle \Rinasy_2(t),\Delta\Vpsi_0(t)\rangle dt,\label{segonacompdeltavpsi}
\end{eqnarray}
where $\beta(\vs)=-(c+\alpha h_0)\log(\vs(1+h_0\vs^{-1})).$

Now we define the linear operator $\Gin$ by the expression:
$$\Gin(\Phi)(\vs)=\vs^\coef(1+h_0\vs^{-1})^\coef\left(\begin{array}{c}
\displaystyle e^{-i(\alpha\vs+\beta(\vs))}\int_{-i\rho}^{\vs}\frac{e^{i(\alpha t+\beta(t))}}{t^\coef(1+h_0t^{-1})^\coef}\langle \Rinasy_1(t),\Phi(t)\rangle dt\\
\displaystyle e^{i(\alpha\vs+\beta(\vs))}\int_{-i\infty}^{\vs}\frac{e^{-i(\alpha t+\beta(t))}}{t^\coef(1+h_0t^{-1})^\coef}\langle \Rinasy_2(t),\Phi(t)\rangle dt
\end{array}\right)
$$
and the function:
$$\Delta\Vpsi_{0,0}(\vs)=\vs^\coef(1+h_0\vs^{-1})^\coef\left(\begin{array}{c}\kappa_0e^{-i(\alpha\vs+\beta(\vs))}\\ 0\end{array}\right).$$
Then we can rewrite \eqref{primeracompdeltavpsi} and \eqref{segonacompdeltavpsi} in the compact form:
\begin{equation}\label{eqdif}\Delta\Vpsi_0(\vs)=\Delta\Vpsi_{0,0}(\vs)+\Gin(\Delta\Vpsi_0)(\vs).\end{equation}

Adapting the steps followed in \cite{BaSe08}, one can see that the operator $\text{Id}-\Gin$ is invertible in a suitable Banach space, and therefore we can write:
\begin{equation}\label{exprdeltainner}
 \Delta\Vpsi_0=(\textup{Id}-\Gin)^{-1}(\Delta\Vpsi_{0,0})=\sum_{n\geq0}\Gin^n(\Delta\Vpsi_{0,0}).
\end{equation}

The last step, once we know that $\Delta\Vpsi_0$ can be obtained form formula \eqref{exprdeltainner}, is to study how the operator $\Gin$ and its iterates $\Gin^n$ act on $\Delta\Vpsi_{0,0}$. What one can prove is that there exists some constant $\overline{K}(\rho)$ such that:
$$\pi_1\Gin(\Delta\Vpsi_{0,0})(\vs)=\vs^de^{-i(\alpha\vs+\beta(\vs))}(\overline{K}(\rho)+O(\vs^{-1})).$$
and that:
$$\pi_2\Gin(\Delta\Vpsi_{0,0})(\vs)=O\left(\vs^{d-2}e^{-i(\alpha\vs+\beta(\vs))}\right).$$
Using standard functional analysis, formula \eqref{asyVpsi0} for $\Delta\Vpsi_0$ is found, finishing the proof of Theorem \ref{theoreminner}.

\section{Proof of Theorem \ref{theoremmatching}}\label{prooftheoremmatching}
Theorem \ref{theoremouter} provides parameterizations of the invariant manifolds satisfying the same equation \eqref{sys2d}. Nevertheless, it does not give enough information about the behavior of these manifolds near the singularities $\pm i\pi/2$. To obtain this information we will use the solutions $\Vpsi_0^{\uns,\sta}$ of the inner equation \eqref{sysinner} given in Theorem \ref{theoreminner}. For this reason, in this section we will deal not with system \eqref{sys2d} but with \eqref{sysdelta} (which comes from \eqref{sys2d} after a change of variables). Moreover, we will restrict ourselves to the matching domains $\dmatch{\uns}$ and $\dmatch{\sta}$ (see Figure \ref{figuradmatchmes}).

Let us consider the Banach space:
$$\bsmatch{*}=\{\phi:\dmatchs{*}\rightarrow\mathbb{C},\,\phi\textup{ analytic, }\sup_{\vs\in\dmatchs{*}}|s|^2|\phi(\vs)|<\infty\},\qquad *=\uns,\sta$$
with the norm:
$$\|\phi\|_{\textup{mch}}^{\uns,\sta}=\sup_{\vs\in\dmatchs{*}}|s|^2|\phi(\vs)|,$$
and we endow the product space $\bsmatch{*}\times\bsmatch{*}$ with the norm:
$$\|(\phi_1,\phi_2)\|_{\textup{mch},\times}^{\uns,\sta}=\|\phi_1\|_{\textup{mch}}^{\uns,\sta}+\|\phi_2\|_{\textup{mch}}^{\uns,\sta}.$$

Now we present the main result of this section, which is equivalent to Theorem \ref{theoremmatching}:
\begin{prop}\label{propmatching}
 Let $\Vpsi^{\uns,\sta}(\vs)=\delta\vphi^{\uns,\sta}(\delta\vs+i\pi/2)$, where $\vphi^{\uns,\sta}$ are the parameterizations given by Theorem \ref{theoremouter}. If $\vs\in\dmatchs{*}$, for $*=\uns,\sta$, one has $\Vpsi^{\uns,\sta}(\vs)=\Vpsi_0^{\uns,\sta}(\vs)+\Vpsi_1^{\uns,\sta}(\vs)$, where $\Vpsi_0^{\uns,\sta}$ are the two solutions of the inner system \eqref{sysinner} given by Theorem \ref{theoreminner} and:
$$\|\Vpsi_1^{\uns,\sta}\|_{\textup{mch},\times}^{\uns,\sta}\leq  K\delta^{1-\gamma},$$
for some constant $K$.
\end{prop}

Now we shall proceed to prove Proposition \ref{propmatching} for the unstable case. The stable case is analogous. As usual, we will omit the superindices $-\uns-$ of the domain $\dmatchs[]{}$, the Banach space $\bsmatch[]{}$ and the norm $\|.\|_{\textup{mch}}$, whenever there is no danger of confusion.

Before proceeding, we will explain the main steps to prove Proposition \ref{propmatching}.

\subsection{Notation and outline of the proof of Proposition \ref{propmatching}}

First of all, let us introduce some notation. We will call $\Vpsi=(\vpsi,\vpsib)$ the solutions of \eqref{sysdelta}. Recalling the definitions \eqref{defhinner} of $\hin(\Vpsi,\vs)$ and \eqref{defAinner} of $\Ain(\vs)$, we define:
\begin{eqnarray}
h(\Vpsi,\vs,\delta,\param)&=&\left[b\vpsi\vpsib+H(\Vpsi,-\vs^{-1}+\delta\f,\delta,\delta\param)\right](\vs^2+\delta\vs^3\g),\label{defhmatch}\medskip\\
\CMcal{A}(\vs,\delta,\param)&=&\left(\begin{array}{cc}a_1(\vs,\delta,\param) & 0\\ 0 & a_2(\vs,\delta,\param)\end{array}\right)\label{defAmatch}\medskip\\
\xz(\Vpsi,\vs)&=&\displaystyle\frac{1}{1+\hin(\Vpsi,\vs)}\left[\Ain(\vs)\Vpsi+F(\Vpsi,-\vs^{-1},0,0)\right],\label{defX0}\medskip\\
  \xu(\Vpsi,\vs,\delta,\param)&=&\displaystyle\frac{1}{1+h(\Vpsi,\vs,\delta,\param)}\left[\CMcal{A}(\vs,\delta,\param)\Vpsi+F(\Vpsi,-\vs^{-1}+\delta\f,\delta,\delta\param)\right]\nonumber\medskip\\
&&\displaystyle-\frac{1}{1+\hin(\Vpsi,\vs)}\left[\Ain(\vs)\Vpsi+F(\Vpsi,-\vs^{-1},0,0)\right],\label{defX1}
\end{eqnarray}
where:
\begin{equation}\label{defa1match}
 \begin{array}{rcl}a_1(\vs,\delta,\param)&=&-(\alpha+c(-\vs^{-1}+\delta\f)i-\delta\param+\coef\vs^{-1}-\delta\coef\f,\\
 a_2(\vs,\delta,\param)&=&(\alpha+c(-\vs^{-1}+\delta\f)i-\delta\param+\coef\vs^{-1}-\delta\coef\f\end{array}
\end{equation}
Note that $\CMcal{A}(\vs,0,\param)=\Ain(\vs)$, and $h(\Vpsi,\vs,0,\param)=\hin(\Vpsi,\vs)$.

Then, the full system \eqref{sysdelta} can be written as:
\begin{equation}\label{sysdeltacompact}
 \frac{d\Vpsi}{d\vs}=\xz(\Vpsi,\vs)+\xu(\Vpsi,\vs,\delta,\param),
\end{equation}
and the inner system \eqref{sysinner} reads:
\begin{equation}\label{sinnercompact}
 \frac{d\Vpsi}{d\vs}=\xz(\Vpsi,\vs).
\end{equation}

Let us consider $\Vpsi^\uns$ defined as the parameterization of the one-dimensional unstable manifold of system \eqref{sys2d} given by Theorem \ref{theoremouter} in the new coordinates, that is $\Vpsi^\uns(\vs)=\delta\vphi^\uns(\delta \vs+i\pi/2),$ which is a solution of \eqref{sysdeltacompact}. Moreover, consider the solution $\Vpsi_0^\uns$ of the inner system \eqref{sinnercompact} given by Theorem \ref{theoreminner}. Then, if we define their difference:
\begin{equation}\label{defdif}\Vpsi_1^\uns=\Vpsi^\uns-\Vpsi_0^\uns,\end{equation}
we have that $\Vpsi_1^\uns$ satisfies:
\begin{equation*}
  \displaystyle\frac{d\Vpsi_1^\uns}{d\vs}=\xz(\Vpsi_0^\uns+\Vpsi_1^\uns,\vs)+\xu(\Vpsi_0^\uns+\Vpsi_1^\uns,\vs,\delta,\delta\param)-\xz(\Vpsi_0^\uns,\vs)
=\displaystyle\frac{1}{1+h_0\vs^{-1}}\Ain(\vs)\Vpsi_1^\uns+\CMcal{R}(\Vpsi_1^\uns,\delta,\param)(\vs)
\end{equation*}
where $h_0$ was defined in \eqref{defh0} and:
\begin{eqnarray}
 \CMcal{R}(\Vpsi_1^\uns,\delta,\param)(\vs)&=&\xz(\Vpsi_0^\uns+\Vpsi_1^\uns,\vs)-\xz(\Vpsi_0^\uns,\vs)-D_\Vpsi\xz(\Vpsi_0^\uns,\vs)\Vpsi_1^\uns\nonumber\\
&&+\xu(\Vpsi_0^\uns+\Vpsi_1^\uns,\vs,\delta,\param)+\left[\frac{1}{1+\hin(\Vpsi_0^\uns,\vs)}-\frac{1}{1+h_0\vs^{-1}}\right]\Ain(\vs)\Vpsi_1^\uns\nonumber\\
&&+\frac{1}{1+\hin(\Vpsi_0^\uns,\vs)}D_\Vpsi F(\Vpsi_0^\uns,\vs^{-1},0,0)\Vpsi_1^\uns\nonumber\\
&&+D_\Vpsi\left[\frac{1}{1+\hin(\Vpsi_0^\uns,\vs)}\right](\Ain(\vs)\Vpsi_0^\uns+F(\Vpsi_0^\uns,\vs^{-1},0,0))\Vpsi_1^\uns.\label{defRmatch}
\end{eqnarray}


Now consider the linear operator acting on functions $(\phi_1,\phi_2)\in\bsmatch[]{}\times\bsmatch[]{}$:
\begin{equation}\label{operadorLmatch}
 \CMcal{L}(\phi_1,\phi_2)(\vs)=\CMcal{M}(\vs)\left(\begin{array}{c}\displaystyle \int_{\Gamma(\vs_1,\vs)}m_1^{-1}(\vw)\phi_1(\vw)d\vw\medskip\\ \displaystyle\int_{\Gamma(\vs_2,\vs)}m_2^{-1}(\vw)\phi_2(\vw)d\vw\end{array}\right),
\end{equation}
where the matrix $\CMcal{M}(\vs)$ was defined in \eqref{matriuMinner}, $s_i$, $i=1,2$, were defined in \eqref{etsj} and $\Gamma(\vs_i,\vs)$ is any curve in $\dmatchs[]{}$ going from $\vs_i$ to $\vs$. Note that, since for $\vw\in\dmatchs[]{}$ the functions $m_1(\vw)^{-1}\phi_1(\vw)$ and $m_2(\vw)^{-1}\phi_2(\vw)$ are analytic, by Cauchy's theorem the integrals in \eqref{operadorLmatch} do not depend on the choice of the curves $\Gamma(\vs_i,\vs)$.

With this notation, it is clear that $\Vpsi_1^\uns$ satisfies the fixed point equation:
\begin{equation}\label{fixedpointmatch}
 \Vpsi_1^\uns(\vs)=\CMcal{I}(c_1,c_2)(\vs)+\CMcal{L}\circ \CMcal{R}(\Vpsi_1^\uns,\delta,\param)(\vs),
\end{equation}
being:
\begin{equation}\label{defImatch}
 \CMcal{I}(k_1,k_2)(\vs)=\CMcal{M}(\vs)\left(\begin{array}{c}k_1\\k_2 \end{array}\right)
\end{equation}
and
\begin{equation}\label{c1c2fixedpointmatch}c_1=m_1^{-1}(\vs_1)\vpsi_1(\vs_1),\qquad c_2=m_2^{-1}(\vs_2)\vpsib_1(\vs_2).\end{equation}

Now, let us explain the main steps to check Proposition \ref{propmatching}. First of all, we note that the fixed point equation \eqref{fixedpointmatch} is equivalent to:
\begin{eqnarray}
 \Vpsi_1^\uns&=\CMcal{I}(c_1,c_2)+\CMcal{L}\circ \CMcal{R}(0,\delta,\param)+\CMcal{L}\circ [\CMcal{R}(\Vpsi_1^\uns,\delta,\param)-\CMcal{R}(0,\delta,\param)],
\end{eqnarray}
where:
\begin{equation*}
 \CMcal{L}\circ [\CMcal{R}(\Vpsi_1^\uns,\delta,\param)-\CMcal{R}(0,\delta,\param)]=\CMcal{M}(\vs)\left(\begin{array}{c} \int_{\Gamma(\vs_1,\vs)}m_1^{-1}(\vw)[\CMcal{R}_1(\Vpsi_1^\uns,\delta,\param)(\vw)-\CMcal{R}_1(0,\delta,\param)(\vw)]d\vw\\ \int_{\Gamma(\vs_2,\vs)}m_2^{-1}(\vw)[\CMcal{R}_2(\Vpsi_1^\uns,\delta,\param)(\vw)-\CMcal{R}_2(0,\delta,\param)(\vw)]d\vw\end{array}\right).
\end{equation*}
Note that:
\begin{equation*}
 \CMcal{R}(\Vpsi_1^\uns,\delta,\param)(\vw)-\CMcal{R}(0,\delta,\param)(\vw)=\int_0^1{D_\Vpsi \CMcal{R}(\lambda\Vpsi_1^\uns,\delta,\param)(\vw)d\lambda}\Vpsi_1^\uns(\vw)=\int_0^1{D_\Vpsi \CMcal{R}(\lambda(\Vpsi^\uns-\Vpsi_0^\uns),\delta,\param)(\vw)d\lambda}\Vpsi_1^\uns(\vw).
\end{equation*}
Now, since we already proved the existence of both parameterizations $\Vpsi^\uns$ and $\Vpsi_0^\uns$, we can think of the integral term as independent of $\Vpsi_1^\uns$, that is:
$$\CMcal{R}(\Vpsi_1^\uns,\delta,\param)(\vw)-\CMcal{R}(0,\delta,\param)(\vw)=B(\vw)\Vpsi_1^\uns(\vw),$$
where the matrix $B(\vw)$ is given by:
$$B(\vw)=\int_0^1{D_\Vpsi \CMcal{R}(\lambda(\Vpsi^\uns-\Vpsi_0^\uns),\delta,\param)(\vw)d\lambda}.$$
Therefore, for $\Vpsi\in\bsmatch[]{}\times\bsmatch[]{}$, we can define the linear operators:
\begin{equation}\label{defBmatchGmatch}\CMcal{B}(\Vpsi)(\vw)=B(\vw)\Vpsi(\vw),\qquad\CMcal{G}(\Vpsi)(\vs)=\CMcal{L}\circ \CMcal{B}(\Vpsi)(\vs),\end{equation}
and then equation \eqref{fixedpointmatch} can be rewritten as:
\begin{equation}\label{fixedpointmatchcompacte}
 \left(\textup{Id}-\CMcal{G}\right)\Vpsi_1^\uns=\CMcal{I}(c_1,c_2)+\CMcal{L}\circ \CMcal{R}(0,\delta,\param).
\end{equation}
We will proceed to study this equation as follows. First, in Subsections \ref{subsecL} and \ref{subsecB} we will study the linear operators $\CMcal{L}$ and $\CMcal{B}$ respectively. Then, in Subsection \ref{subsecindepterm} we will study the independent term of \eqref{fixedpointmatchcompacte}, that is $\CMcal{I}(c_1,c_2)+\CMcal{L}\circ \CMcal{R}(0,\delta,\param)$. Finally, in Subsection \ref{subsecendproof} we will see that joining the results of the previous subsections allows us to guarantee that the operator $\textup{Id}-\CMcal{G}$ is invertible in $\bsmatch[]{}\times\bsmatch[]{}$ and to obtain the desired bound for the norm of $\Vpsi_1^\uns$.

\subsection{The linear operator $\CMcal{L}$}\label{subsecL}
As we already mentioned, in this subsection we will study the operator $\CMcal{L}$. However, before we present two technical lemmas. The first one is completely analogous to Lemma \ref{lematriangles}, and can be proved in the same way.
\begin{lem}\label{lematriangles2}
 Let $\vs\in\dmatchs[]{}$ and $\vw=\vs_1+t(\vs-\vs_1)$, $\tilde\vw=\vs_2+t(\vs-\vs_2)$, with  $t\in[0,1]$. Then there exists $K\neq0$ independent of $\delta$ such that:
$$|\vw|,|\tilde\vw|\geq K|\vs|.$$
\end{lem}

\begin{lem}\label{lemafitaintegrand}
Let $\vs\in\dmatchs[]{}$ and $\vw=\vs_1+t(\vs-\vs_1)$, $\tilde\vw=\vs_2+t(\vs-\vs_2)$, with  $t\in[0,1]$. Then there exists $K$ independent of $\delta$ such that:
$$|m_1(\vs)m_1^{-1}(\vw)|\leq Ke^{\alpha (1-t)\im(\vs-\vs_1)},\qquad|m_2(\vs)m_2^{-1}(\tilde\vw)|\leq Ke^{\alpha (1-t)\im(\vs_2-\vs)},$$
where $m_1$ and $m_2$ are defined in \eqref{matriuMinner}.
\end{lem}
\begin{proof}
 We will do just the case for $\vw$. We have:
$$m_1(\vs)m_1^{-1}(\vw)=\frac{\vs^\coef(1+h_0\vs^{-1})^\coef}{\vw^\coef(1+h_0\vw^{-1})^\coef}e^{-i[\alpha(\vs-\vw)+\beta(\vs)-\beta(\vw)]}.$$

First of all note that by Lemma \ref{lematriangles2} we have that:
$$\left|\frac{\vs^\coef(1+h_0\vs^{-1})^\coef}{\vw^\coef(1+h_0\vw^{-1})^\coef}\right|\leq K\left|\frac{1+h_0\vs^{-1}}{1+h_0\vw^{-1}}\right|^\coef.$$
Moreover, for $\delta$ small enough we have that $|\vs|,|\vw|\geq K\log(1/\delta)\geq2h_0$ and hence:
\begin{equation}\label{3k}
 \left|\frac{\vs^\coef(1+h_0\vs^{-1})^\coef}{\vw^\coef(1+h_0\vw^{-1})^\coef}\right|\leq K\frac{\left(1+|h_0\vs^{-1}|\right)^\coef}{\left(1-|h_0\vw^{-1}|\right)^\coef}\leq K.
\end{equation}

On the other hand, we have that:
$$\left|e^{-i[\alpha(\vs-\vw)+\beta(\vs)-\beta(\vw)]}\right|\leq e^{\alpha\im(\vs-\vw)}e^{|\im\beta(\vs)|+|\im\beta(\vw)|}.$$
Recall that $\beta(\vs)=-(c+\alpha h_0)\log(\vs+h_0)$ and therefore $\im\beta(\vs)=-(c+\alpha h_0)\arg(\vs+h_0)$,
obtaining for $\im\beta(\vw)$ an analogous expression. It is clear that for $\vs\in\dmatchs[]{}$ we have $\im\vs\leq\im\vs_1<0,$ and then, since $h_0$ is real, we also have that $\im(\vs+h_0)<0.$ Consequently, we have $\arg(\vs+h_0)\in(\pi,2\pi)$ and hence:
$$|\im\beta(\vs)|,|\im\beta(\vw)|\leq(c+\alpha|h_0|)2\pi.$$
Then it is clear that:
\begin{equation}\label{ebeta}
\left|e^{-i[\alpha(\vs-\vw)+\beta(\vs)-\beta(\vw)]}\right|\leq e^{\alpha\im(\vs-\vw)}e^{4\pi(c+\alpha|h_0|)}=e^{\alpha (1-t)\im(\vs-\vs_1)}e^{4\pi(c+\alpha|h_0|)}.
\end{equation}
In conclusion, from \eqref{3k} and \eqref{ebeta} we obtain the initial statement.
\end{proof}

The following lemma studies how the linear operator $\CMcal{L}$ acts on functions belonging to $\bsmatch[]{}\times\bsmatch[]{}$.
\begin{lem}\label{lemLmatch}
 The operator $\CMcal{L}:\bsmatch[]{}\times\bsmatch[]{}\rightarrow\bsmatch[]{}\times\bsmatch[]{}$ is well defined and there exists a constant $K$ such that for any $\phi\in\bsmatch[]{}\times\bsmatch[]{}$, then:
$$\|\CMcal{L}\circ\phi\|_{\textup{mch},\times}\leq K\|\phi\|_{\textup{mch},\times}.$$
\end{lem}
\begin{proof}
 We will check the bound for the first component. We have:
$$\pi^1\CMcal{L}\circ\phi(\vs)= m_1(\vs)\int_{\Gamma(\vs_1,\vs)}m_1^{-1}(\vw)\phi_1(\vw)d\vw.$$
Taking $\Gamma(\vs_1,\vs)$ as the segment from $\vs_1$ to $\vs$ and parameterizing it by $\gamma(t)=\vs_1+t(\vs-\vs_1)$, $t\in[0,1]$, we have:
\begin{eqnarray*}
 |\pi^1\CMcal{L}\circ\phi(\vs)|&\leq& |\vs-\vs_1|\int_0^1|m_1(\vs)m_1^{-1}(\vs_1+t(\vs-\vs_1))\phi_1(\vs_1+t(\vs-\vs_1))|dt\\
&\leq&K|\vs_1-\vs|\|\phi_1\|_{\textup{mch}}\int_0^1|m_1(\vs)m_1^{-1}(\vs_1+t(\vs-\vs_1))|\vs_1+t(\vs-\vs_1)|^{-2}dt.
\end{eqnarray*}
Using Lemmas \ref{lematriangles2} and \ref{lemafitaintegrand} it is clear that:
\begin{eqnarray*}
 |\pi^1\CMcal{L}\circ\phi(\vs)|&\leq&K|\vs_1-\vs|\|\phi_1\|_{\textup{mch}}|\vs|^{-2}\left|\int_0^1e^{\alpha(1-t)\im(\vs-\vs_1)}dt\right|\\
&=&\frac{K|\vs-\vs_1|}{\alpha|\im(\vs_1-\vs)|}\|\phi_1\|_{\textup{mch}}|\vs|^{-2}\left|1-e^{\alpha\im(\vs-\vs_1)}\right|.
\end{eqnarray*}
Finally, we note that as $\im(\vs-\vs_1)\leq0$ we have that $\left|1-e^{\alpha\im(\vs-\vs_1)}\right|\leq1$. Moreover, from the definition of $\dmatchs[]{}$, using standard geometric arguments, it is easy to see that there exists a constant $C(\beta_1,\beta_2)$ such that $|\vs-\vs_1|\leq C(\beta_1,\beta_2)\left|\im\vs-\im\vs_1\right|$. Then it is clear that $|\pi^1\CMcal{L}\circ\phi(\vs)|\leq K\|\phi_1\|_{\textup{mch}}|\vs|^{-2}$, and consequently $\|\pi^1\CMcal{L}\circ\phi\|_{\textup{mch}}\leq K\|\phi_1\|_{\textup{mch}}$.
\end{proof}

\subsection{The linear operator $\CMcal{B}$}\label{subsecB}
Now we proceed to study the operator $\CMcal{B}$, defined in \eqref{defBmatchGmatch}. However, before we will need to study the vector field $\xu$.

\begin{lem}\label{lemaxu}
 Consider the vector field $\xu$ defined in \eqref{defX1}, and let $\Vpsi\in\bsmatch[]{}\times\bsmatch[]{}$, such that $\|\Vpsi\|_{\textup{mch},\times}\leq1$. Then there exists a constant $K$ such that for all $\vs\in\dmatchs[]{}$:
$$|\xu(\Vpsi,\vs,\delta,\param)|\leq K\delta|\vs|^{-2}.$$
\end{lem}
\begin{proof}
 First of all we will rewrite $\xu$, which was defined in \eqref{defX1}, in a more convenient way:
\begin{eqnarray*}
 \xu(\Vpsi,\vs,\delta,\param)&=&\left[\frac{1}{1+h(\Vpsi,\vs,\delta,\param)}-\frac{1}{1+\hin(\Vpsi,\vs)}\right]\left[\CMcal{A}(\vs,\delta,\param)\Vpsi+F(\Vpsi,-\vs^{-1}+\delta\f,\delta,\param)\right]\\
&&+\frac{1}{1+\hin(\Vpsi,\vs)}\left[(\CMcal{A}(\vs,\delta,\param)-\Ain(\vs))\Vpsi+F(\Vpsi,-\vs^{-1}+\delta\f,\delta,\param)-F(\Vpsi,-\vs^{-1},0,0)\right],
\end{eqnarray*}
where $\hin(\Vpsi,\vs)$ was defined in \eqref{defhinner}, $\Ain(\vs)$ in \eqref{defAinner}, $h(\Vpsi,\vs,\delta,\param)$ in \eqref{defhmatch} and $\CMcal{A}(\vs,\delta,\param)$ in \eqref{defAmatch}.
In the following we shall bound each term.

Our first claim is that:
\begin{equation}\label{fitahmenysh0}
 \left|\left[\frac{1}{1+h(\Vpsi,\vs,\delta,\param)}-\frac{1}{1+\hin(\Vpsi,\vs)}\right]\right|\leq  K\delta.
\end{equation}
First of all, note that by Remark \ref{ordreF12Horiginals} and the fact that $|\Vpsi(\vs)|\leq\|\Vpsi\|_{\textup{mch},\times}|\vs|^{-2}$ we have:
\begin{eqnarray}\label{fita1mig}
 |h(\Vpsi,\vs,\delta,\param)|&\leq& \left(b\|\Vpsi\|_{\textup{mch},\times}|\vs|^{-4}+K|\vs|^{-3}\right)\left(|\vs|^2+K\delta^2|\vs^{4}|\right)\leq K(|\vs|^{-2}+|\vs|^{-1}+\delta^2+\delta^2|\vs|)\nonumber\\
&\leq& K\left(\frac{1}{\log^2(1/\delta)}+\frac{1}{\log(1/\delta)}+\delta^2+\delta^{1+\gamma}\right)\leq\frac{1}{2}.
\end{eqnarray}
Note that this bound is also valid for $\hin(\Vpsi,\vs)=h(\Vpsi,-\vs^{-1},0,0)$. Then, by Lemma \ref{lemadifdiv}
we obtain:
\begin{eqnarray}\label{hola0}
 \left|\left[\frac{1}{1+h(\Vpsi,\vs,\delta,\param)}-\frac{1}{1+\hin(\Vpsi,\vs)}\right]\right|&\leq&4|h(\Vpsi,\vs,\delta,\param)-\hin(\Vpsi,\vs)|\nonumber\\
&\leq&4\left|b\vpsi\vpsib+H(\Vpsi,-\vs^{-1}+\delta\f,\delta,\delta\param)\right||\delta\vs^3\g|\nonumber\\
&&+4\left|H(\Vpsi,-\vs^{-1}+\delta\f,\delta,\delta\param)-H(\Vpsi,-\vs^{-1},0,0)\right||\vs|^2.
\end{eqnarray}
Now, on one hand, we have:
\begin{equation}\label{hola1}
\left|b\vpsi\vpsib+H(\Vpsi,-\vs^{-1}+\delta\f,\delta,\delta\param)\right||\delta\vs^3\g|\leq K\left(|\vs|^{-4}+|\vs^{-3}|\right)\delta^2|\vs|^4\leq K\delta.
\end{equation}
On the other hand, note that:
\begin{eqnarray*}
 &&\left|H(\Vpsi,-\vs^{-1}+\delta\f,\delta,\delta\param)-H(\Vpsi,-\vs^{-1},0,0)\right|\\
&&\leq|(\delta\f,\delta,\delta\param)|\int_0^1|\partial_{(z,\delta,\param)}H(\Vpsi,-\vs^{-1}+\lambda\delta\f,\lambda\delta,\lambda\delta\param)|d\lambda.
\end{eqnarray*}
Since for $\lambda\in[0,1]$ and for $\delta$ small enough one has $\phi=(\Vpsi,-\vs^{-1}+\lambda\delta\f,\lambda\delta,\lambda\delta\param)\in B^3(r_0/2)\times B(\delta_0/2)\times B(\param_0/2)$, from Remark \ref{ordreF12Horiginals} and applying again Lemma \ref{lemafitaDFiDH} (with $\phi^*=0$) we can bound all the derivatives of $H$ by $K|\phi|^2$, and then it is straightforward to see that:
\begin{eqnarray*}
 &&\left|H(\Vpsi,-\vs^{-1}+\delta\f,\delta,\delta\param)-H(\Vpsi,-\vs^{-1},0,0)\right|\nonumber\\
&&\leq\int_0^1K|(\Vpsi,-\vs^{-1}+\lambda\delta\f,\lambda\delta,\lambda\delta\param)|^2d\lambda\cdot|(\delta\f,\delta,\delta\param)|\leq K|\vs|^{-2}|(\delta\f,\delta,\delta\param)|\nonumber\\
&&\leq K\delta|\vs|^{-2},
\end{eqnarray*}
where we have used that $|\Vpsi(\vs)|\leq K|\vs|^{-2}$. Hence it is clear that:
\begin{equation}\label{hola2}
 \left|H(\Vpsi,-\vs^{-1}+\delta\f,\delta,\delta\param)-H(\Vpsi,-\vs^{-1},0,0)\right||\vs|^2\leq K\delta
\end{equation}
Substituting \eqref{hola1} and \eqref{hola2} in inequality \eqref{hola0}, claim \eqref{fitahmenysh0} is proved.

Our second claim is that:
\begin{equation}\label{fitaAmesF}
 |\CMcal{A}(\vs,\delta,\param)\Vpsi+F(\Vpsi,-\vs^{-1}+\delta\f,\delta,\param)|\leq K|s|^{-2}.
\end{equation}
This is straightforward to check, since the matrix $\CMcal{A}(\vs,\delta,\param)$ is bounded for $\vs\in\dmatchs[]{}$ (which is clear from \eqref{defAmatch} and \eqref{defa1match}), $\Vpsi\in\bsmatch[]{}\times\bsmatch[]{}$ and $|F(\Vpsi,-\vs^{-1}+\delta\f,\delta,\param)|\leq K|s|^{-3}$.

Our third claim is that, since as we already mentioned $|\hin(\Vpsi,\vs)|\leq1/2$, then:
\begin{equation}\label{fita1divh0}
 \left|\frac{1}{1+\hin(\Vpsi,\vs)}\right|\leq 2.
\end{equation}

The last claim is that:
\begin{equation}\label{afmenysaf0}
 |(\CMcal{A}(\vs,\delta,\param)-\Ain(\vs))\Vpsi+F(\Vpsi,-\vs^{-1}+\delta\f,\delta,\param)-F(\Vpsi,-\vs^{-1},0,\param)|\leq K\delta|s|^{-2}.
\end{equation}
First, we note that:
$$\CMcal{A}(\vs,\delta,\param)-\Ain(\vs)=\left(\begin{array}{cc}(-1-ic)\delta\f -\delta\param& 0\\ 0 & +(-1+ic)\delta\f-\delta\param\end{array}\right)$$
and since $|\delta\f|=O(\delta^{1+\gamma})$, it is clear that:
$$|(\CMcal{A}(\vs,\delta,\param)-\Ain(\vs))\Vpsi|\leq K\delta\|\Vpsi\|_{\textup{mch},\times}|s|^{-2}.$$
On the other hand, using the mean value theorem and Lemma \ref{lemafitaDFiDH} it is also easy to see that:
\begin{equation*}|F(\Vpsi,-\vs^{-1}+\delta\f,\delta,\param)-F(\Vpsi,-\vs^{-1},0,\param)|\leq \int_0^1|D_{(z,\delta)}F(\Vpsi,-\vs^{-1}+t\delta\f,t\delta,\param)|dt|(\delta\f,\delta)|\leq K\delta|s|^{-2},
\end{equation*}
so inequality \eqref{afmenysaf0} is clear.

In conclusion, from bounds \eqref{fitahmenysh0}, \eqref{fitaAmesF}, \eqref{fita1divh0} and \eqref{afmenysaf0} we obtain:
$$|\xu(\Vpsi,\vs,\delta,\param)|\leq  K\delta|s|^{-2}$$
\end{proof}

Now we can proceed to study how the operator $\CMcal{B}$ acts on $\bsmatch[]{}\times\bsmatch[]{}$.
\begin{lem}\label{lemanormaBmatch}
 If $\gamma\in(0,1)$ and $\Vpsi\in\bsmatch[]{}\times\bsmatch[]{}$, with $\|\Vpsi\|_{\textup{mch},\times}\leq 1$, then $\CMcal{B}(\Vpsi)\in\bsmatch[]{}\times\bsmatch[]{}$ and there exists a constant $K$ such that:
$$\|\CMcal{B}(\Vpsi)\|_{\textup{mch},\times}\leq\frac{K}{\log^2(1/\delta)}.$$
\end{lem}

\begin{proof}
 Recall that:
\begin{eqnarray*}
 \CMcal{R}(\Vpsi,\delta,\param)(\vw)&=&\xz(\Vpsi_0^\uns+\Vpsi,\vw)-\xz(\Vpsi_0^\uns,\vw)-D_\Vpsi\xz(\Vpsi_0^\uns,\vw)\Vpsi\\
&&+\xu(\Vpsi_0^\uns+\Vpsi,\vw,\delta,\param)+\left[\frac{1}{1+\hin(\Vpsi_0^\uns,\vw)}-\frac{1}{1+h_0\vw}\right]\CMcal{A}(\vw)\Vpsi\\
&&+\frac{1}{1+\hin(\Vpsi_0^\uns,\vw)}D_\Vpsi F(\Vpsi_0^\uns,\vw^{-1},0,0)\Vpsi\\
&&+D_\Vpsi\left[\frac{1}{1+\hin(\Vpsi_0^\uns,\vw)}\right](\CMcal{A}(\vw)\Vpsi_0^\uns+F(\Vpsi_0^\uns,\vw^{-1},0,0))\Vpsi,
\end{eqnarray*}
and hence:
\begin{eqnarray*}
 D_\Vpsi\CMcal{R}(\Vpsi,\delta,\param)(\vw)&=&D_\Vpsi\xz(\Vpsi_0^\uns+\Vpsi,\vw)-D_\Vpsi\xz(\Vpsi_0^\uns,\vw)+D_\Vpsi\xu(\Vpsi_0^\uns+\Vpsi,\vw,\delta,\param)\\
&&+\left[\frac{1}{1+\hin(\Vpsi_0^\uns,\vw)}-\frac{1}{1+h_0\vw}\right]\CMcal{A}(\vw)+\frac{1}{1+\hin(\Vpsi_0^\uns,\vw)}D_\Vpsi F(\Vpsi_0^\uns,\vw^{-1},0,0)\\
&&+D_\Vpsi\left[\frac{1}{1+\hin(\Vpsi_0^\uns,\vw)}\right](\CMcal{A}(\vw)\Vpsi_0^\uns+F(\Vpsi_0^\uns,\vw^{-1},0,0)).
\end{eqnarray*}
We will see that, for $\vw\in\dmatch[\dist][]{}$:
\begin{equation}\label{fitaDpsiR}
 \left|D_\Vpsi\CMcal{R}(\Vpsi,\delta,\param)(\vw)\right|\leq\frac{K}{\log^2(1/\delta)}.
\end{equation}

First of all we claim that:
\begin{eqnarray}\label{fitarestaDxz}
|D_\Vpsi\xz(\Vpsi_0^\uns+\Vpsi,\vw)-D_\Vpsi\xz(\Vpsi_0^\uns,\vw)|\leq\frac{K}{\log^2(1/\delta)}.
\end{eqnarray}
This can be shown using the mean value theorem in each column of the matrix $D_\Vpsi\xz$. For example, we will prove the result for the first one, $D_\vpsi\xz$. Writing $\Vpsi_\lambda=\Vpsi_0^\uns+\lambda\Vpsi$, the mean value theorem gives us the following bound:
\begin{eqnarray}\label{fitaambintegrand}
|D_\vpsi\xz(\Vpsi_0^\uns+\Vpsi,\vw)-D_\vpsi\xz(\Vpsi_0^\uns,\vw)|&\leq& \int_0^1{\left|D_\Vpsi D_\vpsi\xz(\Vpsi_\lambda,\vw)\right|d\lambda} |\Vpsi(\vw)|\nonumber\\
&\leq& \int_0^1{\left|D_\Vpsi D_\vpsi\xz(\Vpsi_\lambda,\vw)\right|d\lambda}\|\Vpsi\|_{\textup{mch},\times}|\vw|^{-2}\nonumber\\
&\leq&\int_0^1{\left|D_\Vpsi D_\vpsi\xz(\Vpsi_\lambda,\vw)\right|d\lambda}\frac{K\|\Vpsi\|_{\textup{mch},\times}}{\log^2(1/\delta)}.
\end{eqnarray}
Then it is clear that in order to prove \eqref{fitarestaDxz} it is only necessary to prove that the integral is bounded, or equivalently, that the integrand $D_\Vpsi D_\vpsi\xz(\Vpsi_\lambda,\vw)$ (which is a $2\times2$ matrix) is bounded for $\lambda\in[0,1]$. Note that, from definition \eqref{defX0} of $\xz$, fixing $w\in\dmatch[\dist][]{}$ it is clear that $\xz(\phi,\vw)$ is bounded and analytic if:
$$\phi\in B^2(\tilde r_0)\subset B^2(r_0)\cap\{\phi\in\mathbb{C}^2\,:\,|\hin(\phi,\vw)|<1/2\},$$
for some $\tilde r_0$. Then, Cauchy's theorem implies that the derivatives of $\xz$ with respect to $\vpsi$ and $\vpsib$ are bounded. Hence, applying again the same arguments, we prove that all the derivatives of order two are also bounded. Since for $\Vpsi\in\bsmatch{}\times\bsmatch{}$ we have that $\lambda\Vpsi(\vw)\in B^2(\tilde r_0/2)$ for $\delta$ small enough and $\lambda\in[0,1]$, we obtain that $D_\Vpsi D_\vpsi\xz(\Vpsi_\lambda,\vw)$ is bounded and thus \eqref{fitarestaDxz} is proved.

Our next step will be to prove that:
\begin{equation}\label{fitadxu}
 |D_\Vpsi\xu(\Vpsi_0^\uns+\Vpsi,\vw,\delta,\param)|\leq K\delta\leq \frac{K}{\log^2(1/\delta)}.
\end{equation}
In fact, we will prove the result just for the derivative with respect to $\vpsi$, being the one with respect to $\vpsib$ analogous. Note that if $\Vpsi+\Vpsi_0^\uns=(\vpsi+\vpsi_0^\uns,\vpsib+\vpsib_0^\uns)\in\bsmatch[]{}\times\bsmatch[]{}$, then $(\vpsi+\vpsi_0^\uns+|\vw|^{-2}e^{i\theta},\vpsib+\vpsib_0^\uns)\in\bsmatch[]{}\times\bsmatch[]{}$ too. Then first using Cauchy's theorem and later Lemma \ref{lemaxu}, we have:
\begin{eqnarray*}
 |D_\vpsi\xu(\Vpsi_0^\uns+\Vpsi,\vw,\delta,\param)|&\leq&\frac{1}{2\pi|\vw|^{-2}}\int_0^{2\pi}|\xu(\vpsi+\vpsi_0^\uns+|\vw|^{-2}e^{i\theta},\vpsib+\vpsib_0^\uns,\delta,\param)|d\theta\\
&\leq&\frac{1}{2\pi|\vw|^{-2}}\int_0^{2\pi}{ K\delta|\vw|^{-2}d\theta}= K\delta,
\end{eqnarray*}
and the claim is proved.

Now we claim that:
\begin{equation}\label{fitahpsimenysh0A}
 \left|\left[\frac{1}{1+\hin(\Vpsi_0^\uns,\vw)}-\frac{1}{1+h_0\vw^{-1}}\right]\CMcal{A}(\vw)\right|\leq K |\vw|^{-2}\leq \frac{K}{\log^2(1/\delta)}.
\end{equation}
Indeed, on the one hand, note that $\CMcal{A}(\vw)$ is bounded. On the other hand, we observe that for $\delta$ small enough:
$$
 |\hin(\Vpsi_0^\uns,\vw)|\leq\frac{1}{2}, \qquad |h_0\vw|^{-1}\leq\frac{K}{\log(1/\delta)}\leq\frac{1}{2},
$$
and then by Lemma \ref{lemadifdiv} and definition \eqref{defh0} of $h_0$ we obtain that:
$$\left|\frac{1}{1+\hin(\Vpsi_0^\uns,\vw)}-\frac{1}{1+h_0\vw^{-1}}\right|\leq 4|\hin(\Vpsi_0^\uns,\vw)-h_0\vw^{-1}|\leq K|\vw|^{-2}$$
and then bound \eqref{fitahpsimenysh0A} is clear.

Our next claim, which can be easily proved using Lemma \ref{lemafitaDFiDH} and the fact that $|\Vpsi_0^\uns(\vw)|\leq K|\vw|^{-3}$, is that:
\begin{equation}\label{1divhDF}
 \left|\frac{1}{1+\hin(\Vpsi_0^\uns,\vw)}D_\Vpsi F(\Vpsi_0^\uns,\vw^{-1},0)\right|\leq 2K|\vw|^{-2}\leq \frac{K}{\log^2(1/\delta)}.
\end{equation}

Finally, we claim that:
\begin{equation}\label{dhAmesF}
 \left|D_\Vpsi\left[\frac{1}{1+\hin(\Vpsi_0^\uns,\vw)}\right](\CMcal{A}(\vw)\Vpsi_0^\uns+F(\Vpsi_0^\uns,\vw^{-1},0))\right|\leq K|\vw|^{-2}\leq \frac{K}{\log^2(1/\delta)}.
\end{equation}
Indeed, we note that $D_\Vpsi(1+h(\Vpsi_0^\uns,\vw))^{-1}$ is bounded. We have to use that $(1+h(\Vpsi_0^\uns,\vw))^{-1}$ is bounded and analytic in a ball of radius $\tilde r_0$, and use Cauchy's theorem in a ball of radius $\tilde r_0/2$ (where $\Vpsi_0$ belongs to) to prove that the derivative with respect to $\Vpsi$ is bounded. Finally, \eqref{dhAmesF} follows from the following bounds:
$$|\CMcal{A}(\vw)\Vpsi_0^\uns|\leq K\|\Vpsi_0^\uns\|_{\textup{mch},\times}|\vw|^{-2},\qquad|F(\Vpsi_0^\uns,\vw^{-1},0,0))|\leq K|\vw|^{-3}.$$
In the second bound we have used Remark \ref{ordreF12Horiginals}.

With bounds \eqref{fitarestaDxz}, \eqref{fitadxu}, \eqref{fitahpsimenysh0A}, \eqref{1divhDF} and \eqref{dhAmesF} we obtain that:
$$|B(\vw)\Vpsi(\vw)|\leq\int_0^1{|D_\Vpsi \CMcal{R}(\lambda\Vpsi,\delta,\param)(\vw)|d\lambda}|\Vpsi(\vw)|\leq\frac{K}{\log^2(1/\delta)}|\Vpsi(\vw)|,$$
and then, since $\|\Vpsi\|_{\textup{mch},\times}\leq1$, it is clear that:
$$\|\CMcal{B}(\Vpsi)\|_{\textup{mch},\times}\leq\frac{K}{\log^2(1/\delta)}\|\Vpsi\|_{\textup{mch},\times}\leq\frac{K}{\log^2(1/\delta)}.$$
\end{proof}

\subsection{The independent term}\label{subsecindepterm}
Finally, in this subsection we will study the independent term $\CMcal{I}(c_1,c_2)+\CMcal{L}\circ \CMcal{R}(0,\delta,\param)$. First we note that if in Lemma \ref{lemaxu} we take $\Vpsi=\Vpsi_0^\uns\in\bsmatch[]{}\times\bsmatch[]{}$, noting that $\xu(\Vpsi_0^\uns,\vs,\delta,\param)=\CMcal{R}(0,\delta,\param)(\vs)$ (see the definition \eqref{defRmatch} of $\CMcal{R}$) and that for $\rho$ big enough $\|\Vpsi_0\|_{\textup{mch},\times}\leq1$, we obtain immediately the following corollary:
\begin{cor}\label{corolariR}
 $\CMcal{R}(0,\delta,\param)\in\bsmatch[]{}\times\bsmatch[]{}$ and there exists a constant $K$ such that:
$$\|\CMcal{R}(0,\delta,\param)\|_{\textup{mch},\times}\leq K\delta.$$
\end{cor}

\begin{lem}\label{lemaI+LRmatch}
 $\CMcal{I}(c_1,c_2)+\CMcal{L}\circ \CMcal{R}(0,\delta,\param)\in\bsmatch[]{}\times\bsmatch[]{}$ and:
$$\|\CMcal{L}\circ \CMcal{R}(0,\delta,\param)\|_{\textup{mch},\times}\leq K\delta.$$
Moreover:
$$\|\CMcal{I}(c_1,c_2)+\CMcal{L}\circ \CMcal{R}(0,\delta,\param)\|_{\textup{mch},\times}\leq K\delta^{1-\gamma}.$$
\end{lem}
\begin{proof}
The first part is a direct consequence of Lemma \ref{lemLmatch} and Corollary \ref{corolariR}. To prove the second part, recall that:
$$\CMcal{I}(c_1,c_2)(\vs)=\CMcal{M}(\vs)\left(\begin{array}{c}c_1 \\ c_2\end{array}\right)=\left(\begin{array}{c}m_1(\vs)c_1 \\ m_2(\vs)c_2\end{array}\right),$$
where $\CMcal{M}$ was defined in \eqref{matriuMinner}. Focusing on the first component, using \eqref{c1c2fixedpointmatch}, we have:
$$m_1(\vs)c_1=\frac{m_1(\vs)}{m_1(\vs_1)}\vpsi_1^\uns(\vs_1).$$
First we claim that:
\begin{equation}\label{fitaVpsi1s1}
 |\vpsi_1^\uns(\vs_1)|\leq K\delta^{3(1-\gamma)}
\end{equation}
Indeed, we have $|\vpsi_1^\uns(\vs_1)|\leq|\vpsi^\uns(\vs_1)|+|\vpsi_0^\uns(\vs_1)|$, so we just have to check that both terms satisfy the bound. On the one hand, for $\delta$ small enough $i\pi/2+\vs_1\delta\in \doutT{\uns}$ (see \eqref{defDTDinf} for the definition of $\doutT{\uns}$ and \eqref{etsj} for $s_1$). Then by Theorem \ref{theoremouter}:
$$|\vpsi^\uns(\vs_1)|=|\delta\vxi^\uns(\vs_1\delta+i\pi/2)|\leq K\delta^2|\het(\vs_1\delta+i\pi/2)-1|^3\leq K\delta^2|s_1\delta|^{-3}.$$
Then using that $|\vs_1\delta|\geq K_1\delta^\gamma$ (see \eqref{fitasj}) we obtain immediately that:
$$|\vpsi_1^\uns(\vs_1)|\leq\delta^{3(1-\gamma)}K.$$
On the other hand since, by Proposition \ref{propvarietatsinner}, $\vpsi_0^\uns\in\bsin[3]{\uns}$, from definition \eqref{defnormin} of the norm $\|.\|_{\textup{in},3,\times}^\uns$ we know that:
$$|\vpsi_0^\uns(\vs_1)|\leq\|\vpsi_0^\uns\|_{\textup{in},\nu,\times}^\uns|\vs_1|^{-3}\leq \delta^{3(1-\gamma)}K,$$
and then the claim is clear.

Then, by \eqref{fitaVpsi1s1} and Lemma \ref{lemafitaintegrand} we obtain:
$$|m_1(\vs)c_1|\leq\delta^{3(1-\gamma)}Ke^{-\alpha\im(\vs_1-\vs)}\leq\delta^{3(1-\gamma)}K.$$
For the second component we obtain an analogous bound, and therefore it is clear that:
$$\|\CMcal{I}(c_1,c_2)\|_{\textup{mch},\times}\leq\sup_{\vs\in\dmatchs[]{}}\delta^{3(1-\gamma)}|\vs|^2K\leq K\delta^{1-\gamma},$$
and the lemma is proved, since $\delta<\delta^{1-\gamma}$.
\end{proof}

\subsection{End of the proof of Proposition \ref{propmatching}}\label{subsecendproof}
From definition \eqref{defBmatchGmatch} of $\CMcal{G}$ and using Lemmas \ref{lemLmatch} and \ref{lemanormaBmatch}, we obtain $\|\CMcal{G}(\Vpsi)\|_{\textup{mch},\times}\leq K\log^{-2}(1/\delta)$. Hence, we have:
$$\|\CMcal{G}\|:=\max_{\|\Vpsi\|_{\textup{mch},\times}\leq1}\{\|\CMcal{G}(\Vpsi)\|_{\textup{mch},\times}\}\leq\frac{K}{\log^2(1/\delta)},$$
and then it is clear that for $\delta$ sufficiently small $\|\CMcal{G}\|<1$. This fact implies that $\textup{Id}-\CMcal{G}$ is invertible in $\bsmatch[]{}\times\bsmatch[]{}$. Then equation \eqref{fixedpointmatchcompacte} and Lemma \ref{lemaI+LRmatch} yield:
\begin{eqnarray*}
 \|\Vpsi_1^\uns\|_{\textup{mch},\times}&=&\|(\textup{Id}-\CMcal{G})^{-1}(\CMcal{I}(c_1,c_2)+\CMcal{L}\circ \CMcal{R}(0,\delta,\param))\|_{\textup{mch},\times}\\
&\leq&\|(\textup{Id}-\CMcal{G})^{-1}\|\|\CMcal{I}(c_1,c_2)+\CMcal{L}\circ \CMcal{R}(0,\delta,\param)\|_{\textup{mch},\times}\leq K\delta^{1-\gamma},
\end{eqnarray*}
proving thus Proposition \ref{propmatching}.

\section{Proof of Theorem \ref{theoremsplit}}\label{prooftheoremsplit}
Let $\Delta\vphi$ be the difference between the parameterizations $\vphi^{\uns,\sta}$ defined in \eqref{defDeltavphi}. Our goal now is to provide a dominant term for this difference, as Theorem \ref{theoremsplit} enunciates. Note that $\Delta\vphi$, being a solution of \eqref{eqsplit}, satisfies:
\begin{equation}\label{fixedpointsplit}
 \Delta\vphi(\vu)=\mathcal{M}(\vu)\left[\left(\begin{array}{c} c_1 \medskip\\ c_2\end{array}\right)+\left(\begin{array}{c} \int_{\vu_+}^\uns m_1^{-1}(\vw)\pi^1(\mathcal{B}(\vw)\Delta\vphi(\vw))d\vw \medskip\\ \int_{\vu_-}^\uns m_2^{-1}(\vw)\pi^2(\mathcal{B}(\vw)\Delta\vphi(\vw))d\vw\end{array}\right)\right],
\end{equation}
for some suitable constants $c_1$ and $c_2$.

As we did in the previous sections, we first need to introduce suitable complex domains and Banach spaces in which we will work. First of all, we define:
$$u_+=i\left(\frac{\pi}{2}-\kappa\delta\log(1/\delta)\right)=it_+,\qquad u_-=i\left(-\frac{\pi}{2}+\kappa\delta\log(1/\delta)\right)=it_-.$$

Now, let $E=\{it\in\mathbb{C}\,:\, t\in(t_-,t_+)\}$. We consider the following Banach spaces:
$$\bssplit=\{\phi:E\rightarrow\mathbb{C}\,:\,\phi \textrm{ analytic, } \sup_{it\in E}|e^{\alpha (\pi/2-|t|)/\delta}\cos^{-d}(t)\phi(it)|<\infty\},$$
with the norm:
\begin{equation}\label{defnormsplit}\|\phi\|_{\textup{spl}}=\sup_{it\in E}|e^{\alpha (\pi/2-|t|)/\delta}\cos^{-d}(t)\phi(it)|.\end{equation}
As usual, in the product space $\bssplit\times\bssplit$ we will take the norm:
$$\|(\phi_1,\phi_2)\|_{\textup{spl},\times}=\|\phi_1\|_{\textup{spl}}+\|\phi_2\|_{\textup{spl}}.$$

The main result of this section, which is equivalent to Theorem \ref{theoremsplit}, is the following:
\begin{prop}\label{propasymsplit}
Let:
$$\Delta\vphi_0(\vu)=\mathcal{M}(\vu)\left(\begin{array}{c}c_1^0\\c_2^0\end{array}\right),$$
where $\mathcal{M}$ was defined in \eqref{defMsplit}, and:
\begin{equation}\label{defc10c20}
\left(\begin{array}{c} c_1^0\medskip\\ c_2^0\end{array}\right)=\left(\begin{array}{c} m_1^{-1}(\vu_+)\frac{(-i\kk)^\coef}{\delta} e^{-\alpha
\kk+i(c+\alpha h_0)\log(-i\kk)}C_{\textup{in}}\medskip\\ m_2^{-1}(\vu_-)\frac{(i\kk)^\coef}{\delta} e^{-\alpha \kk-i(c+\alpha h_0)\log(i\kk)}\overline{C_{\textup{in}}}\end{array}\right),
\end{equation}
where $\kk=\dist\log(1/\delta)$
and $C_{\textup{in}}$ is the constant defined in Theorem \ref{theoreminner}.
Then, if $C_{\textup{in}}\neq0$, we have that $\Delta\vphi=\Delta\vphi_0+\Delta\vphi_1$, where $\Delta\vphi_1$ is such that:
$$\|\Delta\vphi_1\|_{\textup{spl},\times}\leq\frac{K\|\Delta\vphi_0\|_{\textup{spl},\times}}{\log(1/\delta)},$$
for some constant $K$ independent of $\delta$.
\end{prop}

Now we proceed to prove Lemma \ref{lemamatriufonsplit}, which was stated in Subsection \ref{subsectionsplit}. To do that we will use the following technical lemma, which we do not prove (see \cite{DS97}).
\begin{lem}\label{lematecnicintegral}
 Let $\nu>1$. Then there exists a constant $K$ such that, if $\vu\in E$, then:
$$\left|\int_0^\uns\frac{1}{|\cosh\vw|^\nu} d\vw\right|\leq \frac{K}{\delta^{\nu-1}\log^{\nu-1}(1/\delta)}.$$
\end{lem}

\begin{proof}[Proof of Lemma \ref{lemamatriufonsplit}]
Since $\mathcal{M}(\vu)$ is a fundamental matrix of $\dot z=\mathcal{A}(\vu)z$, with $\mathcal{A}(\vu)=\text{diag}(a_{1}(\vu),a_{2}(\vu))$ defined in \eqref{defAsplit}, we have that:
\begin{equation}\label{etM}\mathcal{M}(\vu)=e^{\int_0^\vu \mathcal{A}(\vw)d\vw}=\left(\begin{array}{cc} e^{\int_0^\vu a_{1}(\vw)d\vw} & 0 \\ 0 & e^{\int_0^\vu a_{2}(\vw)d\vw}\end{array}\right).\end{equation}
Let us compute just $m_1(\vu)$. We have:
\begin{eqnarray*}
 \int_0^\vu a_{1}(\vw)d\vw&=&-\frac{\alpha  i}{\delta}\int_0^\vu\frac{1}{1-\frac{\delta h_0\het^3(\vw)}{-1+\het^2(\vw)}}d\vw+\param\int_0^\vu\frac{1}{1-\frac{\delta h_0\het^3(\vw)}{-1+\het^2(\vw)}}d\vw\\&+&(-\coef-ic)\int_0^\vu\frac{1}{1-\frac{\delta h_0\het^3(\vw)}{-1+\het^2(\vw)}}\het(\vw)d\vw=:I_1(\vu)+I_2(\vu)+I_3(\vu).
\end{eqnarray*}
Now we shall give an asymptotic expression for each of these integrals separately. Note that:
$$\left|\frac{\delta h_0\het^3(\vu)}{-1+\het^2(\vu)}\right|\leq\frac{K}{\log(1/\delta)}<1,$$
if $\delta$ small enough, and hence we can write:
$$\frac{1}{1-\frac{\delta h_0\het^3(\vu)}{-1+\het^2(\vu)}}=\sum_{k=0}^{\infty}\left(\frac{\delta h_0\het^3(\vu)}{-1+\het^2(\vu)}\right)^k.$$
Hence, we can express $I_1$ as:
\begin{equation*}
 I_1(\vu)=-\frac{\alpha  i}{\delta}\int_0^\vu\left(1+\frac{\delta h_0\het^3(\vw)}{-1+\het^2(\vw)}\right)d\vw
 -\frac{\alpha  i}{\delta}\int_0^\vu \left(\frac{1}{1-\frac{\delta h_0\het^3(\vw)}{-1+\het^2(\vw)}}-1-\frac{\delta h_0\het^3(\vw)}{-1+\het^2(\vw)}\right)d\vw.
\end{equation*}
Now, note that:
\begin{eqnarray*}
 \left|\frac{\alpha  i}{\delta}\int_0^\vu \left(\frac{1}{1-\frac{\delta h_0\het^3(\vw)}{-1+\het^2(\vw)}}-1-\frac{\delta h_0\het^3(\vw)}{-1+\het^2(\vw)}\right)d\vw\right|&\leq& \frac{K}{\delta}\int_0^\vu \delta^2h_0^2\left|\frac{\het^3(\vw)}{-1+\het^2(\vw)}\right|^2d\vw\leq  K\delta\int_0^\vu\frac{1}{|\cosh\vw|^2}d\vw\\
&\leq& \frac{K}{\log(1/\delta)},
\end{eqnarray*}
where in the last inequality we have used Lemma \ref{lematecnicintegral}. Hence we have:
\begin{eqnarray}\label{i1}
 I_1(\vu)&=&-\frac{\alpha  i}{\delta}\int_0^\vu\left(1+\frac{\delta h_0\het^3(\vw)}{-1+\het^2(\vw)}\right)d\vw+O\left(\frac{1}{\log(1/\delta)}\right)\nonumber\\
&=&-\frac{\alpha  iu}{\delta}+\alpha  h_0i\left(-\frac{1}{2}\sinh^2u+\log\cosh u\right)+O\left(\frac{1}{\log(1/\delta)}\right).
\end{eqnarray}

In the case of $I_2$, we can rewrite it in the following form:
$$I_2(\vu)= \param\int_0^\vu d\vw+\param\int_0^\vu\left(\frac{1}{1-\frac{\delta h_0\het^3(\vw)}{-1+\het^2(\vw)}}-1\right)d\vw.$$
Now, we have:
\begin{eqnarray*}
 \left|\int_0^\vu\left(\frac{1}{1-\frac{\delta h_0\het^3(\vw)}{-1+\het^2(\vw)}}-1\right)d\vw\right|&\leq& K\int_0^\vu\left|\frac{\delta h_0\het^3(\vw)}{-1+\het^2(\vw)}\right|d\vw\leq  K\delta\int_0^\vu\frac{1}{|\cosh\vw|}d\vw\leq\frac{K}{\log(1/\delta)},
\end{eqnarray*}
where we have used that for $\vu\in E$ one has $|\cosh^{-1}\vu|\leq\delta^{-1}\log(1/\delta)$ and $|\vu|\leq\pi/2$. Then, it is clear that:
\begin{equation}\label{i2}
 I_2(\vu)=\param\int_0^\vu d\vw+O\left(\frac{1}{\log(1/\delta)}\right)=\param\vu+O\left(\frac{1}{\log(1/\delta)}\right).
\end{equation}

Finally, $I_3$ can be decomposed as:
$$I_3(\vu)=(-\coef-ic)\int_0^\vu\het(\vw) d\vw+(-\coef-ic)\int_0^\vu\left(\frac{1}{1-\frac{\delta h_0\het^3(\vw)}{-1+\het^2(\vw)}}-1\right)\het(\vw) d\vw.$$
Again, we have:
\begin{eqnarray*}
\left|(-\coef-ic)\int_0^\vu\left(\frac{1}{1-\frac{\delta h_0\het^3(\vw)}{-1+\het^2(\vw)}}-1\right)\het(\vw) d\vw\right|&\leq& K\int_0^\vu\left|\frac{\delta h_0\het^3(\vw)}{-1+\het^2(\vw)}\right|\het(\vw)d\vw\leq K\delta\int_0^\vu\frac{1}{|\cosh^2\vw|}d\vw\\
&\leq&\frac{K}{\log(1/\delta)},
\end{eqnarray*}
where in the last inequality we have used Lemma \ref{lematecnicintegral}.
Then:
\begin{equation}\label{i3}
 I_3(\vu)=(-\coef-ic)\int_0^\vu\het(\vw) d\vw+O\left(\frac{1}{\log(1/\delta)}\right)=(\coef+ic)\log\cosh\vu+O\left(\frac{1}{\log(1/\delta)}\right).
\end{equation}

In conclusion, from \eqref{i1}, \eqref{i2} and \eqref{i3} and the fact that $m_1(\vu)=e^{I_1(\vu)+I_2(\vu)+I_3(\vu)}$, the asymptotic formula \eqref{asyformm1} is proved.
\end{proof}

\begin{lem}\label{lemam1invers}
We have:
\begin{equation}\label{asyformm1menys1} \begin{array}{rcl}m_1^{-1}(\vu_+)&=&\frac{1}{\dist^\coef\delta^{\coef+\dist\alpha}\log^\coef(1/\delta)}e^{-\frac{\alpha\pi}{2\delta}}e^{-i[\frac{\param\pi}{2}+\frac{\alpha h_0}{2}+(c+\alpha h_0)\log\delta]}e^{-i(c+\alpha h_0)\log\kk}\left(1+O\left(\frac{1}{\log(1/\delta)}\right)\right),\medskip\\
m_2^{-1}(\vu_+)&=&\frac{1}{\dist^\coef\delta^{\coef+\dist\alpha}\log^\coef(1/\delta)}e^{-\frac{\alpha\pi}{2\delta}}e^{i[\frac{\param\pi}{2}+\frac{\alpha h_0}{2}+(c+\alpha h_0)\log\delta]}e^{i(c+\alpha h_0)\log\kk}\left(1+O\left(\frac{1}{\log(1/\delta)}\right)\right),
\end{array}\end{equation}
where $\kk=\dist\log(1/\delta)$.
\end{lem}
\begin{proof}
Again, we will prove the asymptotic expression just for $m_1(\vu_+)$, being the other case analogous. First of all, from Lemma \ref{lemamatriufonsplit} we obtain:
\begin{equation}\label{m1invers}
 m_1^{-1}(\vu_+)=\cosh^{-d}(\vu_+)e^{\frac{\alpha i\vu_+}{\delta}}e^{-\param\vu_+}e^{-\alpha  h_0i\left[-\frac{1}{2}\sinh^2\vu_++\log\cosh\vu_+\right]}e^{-ic\log\cosh\vu_+}\left(1+O\left(\frac{1}{\log(1/\delta)}\right)\right).\end{equation}
Recall that $\vu_+=i\pi/2-i\dist\delta\log(1/\delta)$. Then we have:
\begin{eqnarray}
 \cosh^{-d}(\vu_+)&=&\frac{1}{\dist^\coef\delta^\coef\log^\coef(1/\delta)}\left(1+O(\delta\log(1/\delta)\right)\label{asymcosh-d},\\
-\frac{1}{2}\sinh^2(\vu_+)&=&\frac{1}{2}+O\left(\delta^2\log^2(1/\delta)\right)\label{asymsinh2},\\
\log\cosh(\vu_+)&=&\log\delta+\log(\dist\log(1/\delta))+O\left(\delta^2\log^2(1/\delta)\right).\label{asymlogcosh}
\end{eqnarray}
Moreover, it is clear that:
\begin{equation}\label{fitaeialphavudelta}
 e^{i\alpha \vu_+/\delta}=e^{-\frac{\alpha \pi}{2\delta}}e^{\alpha \dist\log(1/\delta)}=e^{-\frac{\alpha \pi}{2\delta}}\frac{1}{\delta^{\dist\alpha}}.
\end{equation}
Finally, we have:
\begin{equation}\label{eimuvu1}
 e^{-\param\vu_+}=e^{-\frac{i\param\pi}{2}}(1+O(\delta\log(1/\delta))).
\end{equation}
Substituting \eqref{asymcosh-d}, \eqref{asymsinh2}, \eqref{asymlogcosh}, \eqref{fitaeialphavudelta} and \eqref{eimuvu1} in \eqref{m1invers} the claim is proved.
\end{proof}

In the following we will proceed to prove Proposition \ref{propmatching}, which will be possible with the lemmas below. In order to simplify the process, we will introduce the notation. For $k_1,k_2\in\mathbb{C}$, we define:
\begin{equation}\label{defIsplit}
 \mathcal{I}(k_1,k_2)(\vu)=\mathcal{M}(\vu)\left(\begin{array}{c}k_1\\k_2\end{array}\right),
\end{equation}
where the matrix $\mathcal{M}(\vu)$ was defined in \eqref{etM}. Note that with this notation we have that $\Delta\vphi_0=\mathcal{I}(c_1^0,c_2^0)$.

For functions $\phi\in\bssplit\times\bssplit$ we define the following operator:
\begin{equation}\label{defGsplit}
 \mathcal{G}(\phi)(\vu)=\left(\begin{array}{c}\mathcal{G}_1(\phi)(\vu)\\ \mathcal{G}_2(\phi)(\vu)\end{array}\right)=\left(\begin{array}{c} m_1(\vu)\int_{u_+}^\uns m_1^{-1}(\vw)\pi^1(\mathcal{B}(\vw)\phi(\vw))d\vw \medskip\\ m_2(\vu)\int_{u_-}^\uns m_2^{-1}(\vw)\pi^2(\mathcal{B}(\vw)\phi(\vw))d\vw\end{array}\right),
\end{equation}
where the matrix $\mathcal{B}(\vw)$ is defined in \eqref{defBsplit}, $m_1(\vw)$ and $m_2(\vw)$ are defined in \eqref{asyformm1}.

\begin{lem}\label{lemaprincipalsplit}
 $\Delta\vphi_1=\Delta\vphi-\Delta\vphi_0$ satisfies:
\begin{eqnarray}\label{eqvphi1split}
 \Delta\vphi_1(\vu)&=&
\mathcal{I}(c_1-c_1^0,c_2-c_2^0)(\vu)+\mathcal{G}(\Delta\vphi_0)(\vu)+\mathcal{G}(\Delta\vphi_1)(\vu).
\end{eqnarray}
Moreover,
\begin{equation}\label{fitacissplit}
 |c_1-c_1^0|,\,|c_2-c_2^0|\leq\frac{Ke^{-\frac{\alpha \pi}{2\delta}}}{\delta^{1+\coef}\log(1/\delta)}.
\end{equation}

\end{lem}
\begin{proof}
To prove the first statement we just need to realize that the fixed point equation for $\Delta\vphi$ \eqref{fixedpointsplit} can be written as:
\begin{equation}\label{igualtat}
 \Delta\vphi=\mathcal{I}(c_1,c_2)+\mathcal{G}(\Delta\vphi).
\end{equation}
Then, since $\Delta\vphi=\Delta\vphi_0+\Delta\vphi_1$ and $\Delta\vphi_0=\mathcal{I}(c_1^0,c_2^0)$, this last equality yields:
$$\Delta\vphi_1=\mathcal{I}(c_1,c_2)-\mathcal{I}(c_1^0,c_2^0)+\mathcal{G}(\Delta\vphi_0+\Delta\vphi_1),$$
and using that the operators $\mathcal{I}$ and $\mathcal{G}$ are linear we obtain equality \eqref{eqvphi1split}.

Now we proceed to prove bound \eqref{fitacissplit}. We will just bound $c_1-c_1^0$, since the other component is analogous. We will write $\Delta\vphi=(\Delta\vxi,\Delta\vxib)$ and $\Delta\vphi_j=(\Delta\vxi_j,\Delta\vxib_j)$, for $j=0,1$.

First note that, since $\mathcal{G}_1(\phi)(\vu_+)=0$ for all $\phi\in\bssplit\times\bssplit$, equalities \eqref{defIsplit} and \eqref{igualtat} yield:
$$c_1=m_1^{-1}(\vu_+)\Delta\vxi(\vu_+).$$
Moreover, since by definition $\Delta\vphi_0(\vu_+)=\mathcal{I}(c_1^0,c_2^0)(\vu_+)$, we also have:
$$c_1^0=m_1^{-1}(\vu_+)\Delta\xi_0(\vu_+).$$
Then it is clear that:
\begin{equation}\label{primerarestac1c10}
 |c_1-c_1^0|=|m_1^{-1}(\vu_+)||\Delta\vxi(\vu_+)-\Delta\xi_0(\vu_+)|.
\end{equation}
Now, taking into account that $u_+\in\dmatch{\uns}\cap\dmatch{\sta}$, from Corollary \ref{corollarymatching} we know that:
\begin{equation}\label{igualtatcorollarymatch}
 \Delta\vxi(\vu_+)=\frac{1}{\delta}\left[\Delta\vpsi_0\left(\frac{\vu_+-i\pi/2}{\delta}\right)+\Delta\vpsi_1\left(\frac{\vu_+-i\pi/2}{\delta}\right)\right],
\end{equation}
where we have written $\Delta\vpsi_j=\vpsi_j^\uns-\vpsi_j^\sta$, for $j=0,1$. Here $\vpsi_0^{\uns,\sta}$ is the first component of the corresponding solution of the inner system \eqref{sysinner}, $\Vpsi_0^{\uns,\sta}$, and $\vpsi_1^{\uns,\sta}$ satisfy:
\begin{equation}\label{fitavpsi1}
 \frac{1}{\delta}\left|\vpsi_1^{\uns,\sta}\left(\frac{\vu_+-i\pi/2}{\delta}\right)\right|\leq\frac{K\delta^{-\gamma}}{\log^{2}(1/\delta)}
\end{equation}
for some constant $K$. From \eqref{primerarestac1c10} and \eqref{igualtatcorollarymatch} we have:
\begin{equation}\label{segonarestac1c10}
 |c_1-c_1^0|\leq|m_1^{-1}(\vu_+)|\left[\left|\frac{1}{\delta}\Delta\vpsi_0\left(\frac{\vu_+-i\pi/2}{\delta}\right)-\Delta\vxi_0(\vu_+)\right|+\left|\frac{1}{\delta}\Delta\vpsi_1\left(\frac{\vu_+-i\pi/2}{\delta}\right)\right|\right].
\end{equation}

Now, on one hand, from Lemma \ref{lemam1invers} it is clear that:
\begin{equation}\label{fitam1menys1}
 |m_1^{-1}(\vu_+)|\leq Ke^{-\frac{\alpha \pi}{2\delta}}\frac{1}{\delta^{\coef+\dist\alpha}\log^\coef(1/\delta)},
\end{equation}
where we have used that:
\begin{equation}\label{exponorma1}\left|e^{-i[\frac{\param\pi}{2}+\frac{\alpha h_0}{2}+(\alpha h_0+c)\log\delta]}e^{-i(\alpha h_0+c))\log\kk}\right|=1.\end{equation}
On the other hand, from definition \eqref{defc10c20} of $c_1^0$ and $c_2^0$ and the fact that $\Delta\vxi_0(\vu_+)=m_1(\vu_+)c_1^0$ it is clear that:
$$\Delta\vxi_0(\vu_+)=\frac{(-i\kk)^\coef}{\delta} e^{-\alpha \kk+i(c+\alpha  h_0)\log(-i\kk)}C_{\textup{in}},$$
where $\kk=\dist\log(1/\delta)$. Then, from formula \eqref{asyVpsi0} of $\Delta\Vpsi_0$ in Theorem \ref{theoreminner} and this last equality we have that:
$$\left|\frac{1}{\delta}\Delta\vpsi_0\left(\frac{\vu_+-i\pi/2}{\delta}\right)-\Delta\vxi_0(\vu_+)\right|=\left|\frac{(-i\kk)^\coef}{\delta}e^{-\alpha \kk+i(c+\alpha  h_0)\log(-i\kk)}\chi_1(-i\kk)\right|,$$
where $\chi_1$ is the first component of the function $\chi$ in Theorem \ref{theoreminner}. Now, since by this proposition we know that $|\chi_1(\vs)|\leq K|\vs|^{-1}$, we have:
\begin{equation}\label{fitavpsi0vxi0}
\left|\frac{1}{\delta}\Delta\vpsi_0\left(\frac{\vu_+-i\pi/2}{\delta}\right)-\Delta\vxi_0(\vu_+)\right|\leq K\frac{\kk^{\coef-1}}{\delta}\delta^{\dist\alpha }|e^{i(c+\alpha h_0)\log(-i\kk)}|\leq K\frac{\kk^{\coef-1}}{\delta}\delta^{\dist\alpha}.
\end{equation}
Then, bounds \eqref{fitavpsi1}, \eqref{fitam1menys1} and \eqref{fitavpsi0vxi0} yield:
\begin{eqnarray}\label{fitac1-c10}
 &&|m_1^{-1}(\vu_+)|\left(\left|\frac{1}{\delta}\Delta\vpsi_0\left(\frac{\vu_+-i\pi/2}{\delta}\right)-\Delta\vxi_0(\vu_+)\right|+\left|\frac{1}{\delta}\Delta\vpsi_1\left(\frac{\vu_+-i\pi/2}{\delta}\right)\right|\right)\nonumber\\
&&\leq Ke^{-\frac{\alpha \pi}{2\delta}}\left(\frac{1}{\delta^{1+\coef}\log(1/\delta)}+\frac{1}{\delta^{\coef+\dist\alpha +\gamma}\log(1/\delta)}\right).
\end{eqnarray}
%
Then, taking $\dist>0$ such that $0<\dist\alpha <1-\gamma$, from \eqref{segonarestac1c10} and \eqref{fitac1-c10} we obtain immediately:
$$|c_1-c_1^0|\leq Ke^{-\frac{\alpha \pi}{2\delta}}\frac{1}{\delta^{1+\coef}\log(1/\delta)}.$$
\end{proof}

\begin{lem}\label{lemafitaIsplit}
 Let $k_1$, $k_2\in\mathbb{C}$. Then, $\mathcal{I}(k_1,k_2)\in\bssplit\times\bssplit$ and:
$$\|\mathcal{I}(k_1,k_2)\|_{\textup{spl},\times}=(|k_1|+|k_2|)e^{\frac{\alpha  \pi}{2\delta}}\left(1+O\left(\frac{1}{\log(1/\delta)}\right)\right).$$
\end{lem}
\begin{proof}
 We will bound just the norm of the first component of $\mathcal{I}(k_1,k_2)$, that is $\pi^1\mathcal{I}(k_1,k_2)=m_1(u)k_1$. For $it\in E$, from Lemma \ref{lemamatriufonsplit} we have:
\begin{equation*}
 |m_1(it)k_1||\cos^{-\coef}te^{\frac{\alpha(\pi/2-|t|)}{\delta}}|= |k_1||e^{\frac{\alpha(\pi/2+t-|t|)}{\delta}}||e^{i\param t}||e^{\alpha  h_0i[\sin^2t/2+\log\cos t]}||e^{ic\log\cos t}|\left(1+O\left(\frac{1}{\log(1/\delta)}\right)\right).
\end{equation*}

Note that, since $t$ is real and $|t|<\pi/2$, we have that $\log\cos t$ is real. Moreover, since $\param, \alpha, h_0\in\mathbb{R}$, we have that $|e^{i\param t}|=|e^{\alpha h_0i[\sin^2t/2+\log\cos t]}|=1$, and then:
$$|m_1(it)k_1||\cos^{-\coef}te^{\alpha  (\pi/2-|t|)/\delta}|=|k_1|e^{\alpha  (\pi/2+t-|t|)/\delta}\left(1+O\left(\frac{1}{\log(1/\delta)}\right)\right).$$
Then it is clear that:
\begin{eqnarray*}
 \sup_{it\in E}|m_1(it)k_1||\cos^{-\coef}te^{\alpha  (\pi/2-|t|)/\delta}|&=&|k_1|e^{\frac{\alpha \pi}{2\delta}}\sup_{it\in E}e^{\alpha (t-|t|)/\delta}\left(1+O\left(\frac{1}{\log(1/\delta)}\right)\right)\\
&=&|k_1|e^{\frac{\alpha \pi}{2\delta}}\left(1+O\left(\frac{1}{\log(1/\delta)}\right)\right),
\end{eqnarray*}
and hence:
$$\|\pi^1\mathcal{I}(k_1,k_2)\|_{\textup{spl}}=|k_1|e^{\frac{\alpha  \pi}{2\delta}}\left(1+O\left(\frac{1}{\log(1/\delta)}\right)\right),$$
obtaining the desired bound.
\end{proof}

\begin{lem}\label{lemaasynormadeltavphi0}
 $$\|\Delta\vphi_0\|_{\textup{spl},\times}=\frac{1}{\delta^{\coef+1}}e^{\frac{\pi}{2}(c+\alpha h_0)}\left(|C_{\textup{in}}|+|\overline{C_{\textup{in}}}|\right)\left(1+O\left(\frac{1}{\log(1/\delta)}\right)\right).$$
\end{lem}
\begin{proof}
Since $\Delta\vphi_0=\mathcal{I}(c_1^0,c_2^0)$, we just have to bound $c_1^0$ and $c_2^0$ and then use Lemma \ref{lemafitaIsplit}. We will just bound $c_1^0$, being the other case analogous. Recall that:
$$c_1^0=m_1^{-1}(\vu_+)\frac{(-i\kk)^\coef}{\delta}e^{-\alpha\kk+i(c+\alpha h_0)\log(-i\kk)}C_{\textup{in}},$$
with $\kk=\dist\log(1/\delta)$.
On one hand, from formula \eqref{asyformm1menys1} and \eqref{exponorma1} it is clear that:
\begin{equation}\label{cosa1}
 |m_1^{-1}(\vu_+)|=\frac{1}{\dist^\coef\delta^{\coef+\dist\alpha}\log^\coef(1/\delta)}e^{-\frac{\alpha\pi}{2\delta}}\left(1+O\left(\frac{1}{\log(1/\delta)}\right)\right).
\end{equation}
On the other hand, noting that $\log(-i\kk)=\log\kk-i\pi/2$ and $e^{-\alpha\kk}=e^{-\alpha\dist\log(1/\delta)}=\delta^{\dist\alpha}$, we have:
\begin{equation}\label{cosa2}
 \left|\frac{(-i\kk)^\coef}{\delta} e^{-\alpha \kk+i(c+\alpha  h_0)\log(-i\kk)}\right|=\frac{\dist^d\log^\coef(1/\delta)}{\delta}\delta^{\dist\alpha}e^{\frac{\pi}{2}(c+\alpha  h_0)}\left|e^{i(c+\alpha  h_0)\log\kk}\right|=\frac{\dist^d\log^\coef(1/\delta)}{\delta}\delta^{\dist\alpha}e^{\frac{\pi}{2}(c+\alpha  h_0)}.
\end{equation}
From \eqref{cosa1} and \eqref{cosa2} it is clear that:
$$|c_1^0|=\frac{1}{\delta^{1+\coef}}e^{\frac{\pi}{2}(c+\alpha h_0)}e^{-\frac{\alpha\pi}{2\delta}}|C_{\textup{in}}|,$$
and then the initial claim is proved by Lemma \ref{lemafitaIsplit}.
\end{proof}

\begin{lem}\label{lemabij}
 There exists a constant $K$ such that, for all $it\in E$ and $l,j=1,2$, the matrix $\mathcal{B}=(b_{lj})$ satisfies:
$$|b_{lj}(it)\cos^2(t)|\leq K\delta.$$
\end{lem}
\begin{proof}
 Recall that:
$$\mathcal{B}(\vu)=\int_0^1D\mathcal{R}(\vphi_\lambda)(\vu)d\lambda,$$
where $\mathcal{R}$ was defined in \eqref{defRsplit}. Note that $\vphi_\lambda=(1-\lambda)\vphi^\sta-\lambda\vphi^\uns$ and hence, by Theorem \ref{theoremouter} it is clear that:
\begin{equation}\label{fitavphilambda}
 |\vphi_\lambda(\vu)|\leq K\delta^2|\cosh^{-3}\vu|.
\end{equation}
 We will prove that for $j=1,2$:
\begin{equation}
\label{resultat}|\pi^jD\mathcal{R}(\vphi_\lambda)(\vu)|\leq\frac{ K\delta}{|\cosh^2\vu|},
\end{equation}
and then from the definition of $\mathcal{B}$ and the fact that $\cosh(it)=\cos t$, the statement will be clear. In fact, we will just do the proof for the first entry of the matrix $D\mathcal{R}$, since all the others are analogous. If we compute this entry, we get:
\begin{eqnarray*}
 D_\vxi\mathcal{R}_1(\vphi_\lambda)(\vu)&=&D_\xi\left[\frac{\delta^{-2} F_1(\delta\vphi_\lambda,\delta\het(\vu),\delta,\delta\param)}{1+\frac{b\vxi_\lambda\vxib_\lambda+\delta^{-2}H(\delta\vphi_\lambda,\delta\het(\vu),\delta,\delta\param)}{-1+\het(\vu)^2}}\right]\\
&&+\left(\frac{1}{1+\frac{b\vxi_\lambda\vxib_\lambda+\delta^{-2}H(\delta \vphi_\lambda, \delta \het(\vu), \delta, \delta\param)}{-1+\het^2(\vu)}}-\frac{1}{1-\frac{\delta h_0\het^3(\vu)}{-1+\het^2(\vu)}}\right)a_1(\vu)\\
&&+\frac{-1}{\left(1+\frac{b\vxi_\lambda\vxib_\lambda+\delta^{-2}H(\delta \vphi_\lambda, \delta \het(\vu), \delta, \delta\param)}{-1+\het^2(\vu)}\right)^2}\frac{b\vxib_\lambda+\delta^{-1}D_\xi H(\delta\vphi_\lambda,\delta\het(\vu),\delta,\delta\param)}{-1+\het^2(\vu)}a_1(\vu)\vxi_\lambda(\vu)\medskip\\
&:=&D_\vxi\mathcal{R}_1^1(\vphi_\lambda)(\vu)+D_\vxi\mathcal{R}_1^2(\vphi_\lambda)(\vu)+D_\vxi\mathcal{R}_1^3(\vphi_\lambda)(\vu).
\end{eqnarray*}

First we will prove that:
\begin{equation}\label{fitatildeF}
 \left|D_\vxi\mathcal{R}_1^1(\vphi_\lambda)(\vu)\right|\leq\frac{ K\delta}{|\cosh^2\vu|}.
\end{equation}
In order to do that, we introduce the auxiliary function:
$$\tilde F_1(\phi)=\frac{\delta^{-2} F_1(\phi)}{1+\displaystyle\frac{\delta^{-2}\left(b\phi_1\phi_2+H(\phi)\right)}{-1+\delta^{-2}\phi_3^2}},$$
where $\phi=(\phi_1,\phi_2,\phi_3,\phi_4,\phi_5)$, which is analytic in the (non-empty) open set:
$$\CMcal{U}:=B^3(r_0)\times B(\delta_0)\times B(\param_0)\cap\left\{\phi\in\mathbb{C}^3\times\mathbb{R}^2\,:\, \left|\frac{\displaystyle \delta^{-2}\left(b\phi_1\phi_2+H(\phi)\right)}{\displaystyle-1+\delta^{-2}\phi_3^2}\right|<\frac{1}{2}\right\}.$$
It is easy to see that, for $\delta$ small enough and $\vu\in\dout[\dist][]{}$, we have that $(\delta\vphi_\lambda(\vu),\delta\het(\vu),\delta,\delta\param)\in B^3(\tilde r_0/2)\times B(\delta_0/2)\times B(\param_0/2)$, where $\tilde r_0<r_0$ is such that $B^3(r_0)\times B(\delta_0)\times B(\param_0)\subset\CMcal{U}$. Moreover, it is easy to check that:
\begin{equation}\label{igualtatderivadesftildef}
 D_\vxi\mathcal{R}_1^1(\vphi_\lambda)(\vu)=\delta D_{\phi_1}\tilde F_1(\delta\vphi_\lambda,\delta\het(\vu),\delta,\delta\param).
\end{equation}
Now, by definition of $\CMcal{U}$ and using Remark \ref{ordreF12Horiginals} we obtain that:
$$\left|\tilde F_1(\phi)\right|\leq 2 \delta^{-2}F_1(\phi)\leq \delta^{-2}K|\phi|^3,$$
and then, using Lemma \ref{lemafitaDFiDH} with $\phi^*=0$, it is clear that if $\phi\in B^3(\tilde r_0/2)\times B(\delta_0/2)\times B(\param_0/2)$, then:
$$\left|D_{\phi_1}\tilde F_1(\phi)\right|\leq \delta^{-2} K|\phi|^2.$$
Hence, from \eqref{igualtatderivadesftildef} and using \eqref{fitavphilambda}, we obtain immediately \eqref{fitatildeF}:
\begin{equation*}\left|D_\vxi\mathcal{R}_1^1(\vphi_\lambda)(\vu)\right|\leq K\delta^{-1}|\delta\vphi_\lambda,\delta\het(\vu),\delta,\delta\param)|^2\leq \frac{K\delta}{\cosh^2\vu}.\end{equation*}

Our next claim is that:
\begin{equation}\label{fitaterme3}
\left|D_\vxi\mathcal{R}_1^2(\vphi_\lambda)(\vu)\right|\leq\frac{ K\delta^2}{|\cosh^3\vu|}.
\end{equation}
First of all note that for $\delta$ small enough:
\begin{equation}
 \left|\frac{b\vxi_\lambda\vxib_\lambda+\delta^{-2}H(\delta \vphi_\lambda, \delta \het(\vu), \delta, \delta\param)}{-1+\het^2(\vu)}\right|,\left|\frac{\delta h_0\het^3(\vu)}{-1+\het^2(\vu)}\right|\leq \frac{K}{\log(1/\delta)}<\frac{1}{2},\label{fitadivisio}
\end{equation}
and then by Lemma \ref{lemadifdiv}, we have:
\begin{eqnarray}\label{fitaterme3intermig}
 &&\left| \left(\frac{1}{1+\frac{b\vxi_\lambda\vxib_\lambda+\delta^{-2}H(\delta \vphi_\lambda, \delta \het(\vu), \delta, \delta\param)}{-1+\het^2(\vu)}}-\frac{1}{1-\frac{\delta h_0\het^3(\vu)}{-1+\het^2(\vu)}}\right)\right|\nonumber\\
&\leq&\frac{4}{|-1+\het^2(\vu)|}\left|b\vxi_\lambda\vxib_\lambda+\delta^{-2}(H(\delta \vphi_\lambda, \delta \het(\vu), \delta, \delta\param)+\delta^3h_0\het(\vu))\right|\nonumber\\
&\leq&\frac{K}{|-1+\het^2(\vu)|}\left[b|\vxi_\lambda||\vxib_\lambda|+ K\delta(|\vphi_\lambda^3|+|\vphi_\lambda^2\het(\vu)|+|\vphi_\lambda\het^2(\vu)|)\right]
\end{eqnarray}
where in the last inequality we have used the definition of $h_0$. It is easy to check that, since bound \eqref{fitavphilambda} holds, for $\vu\in E$, we have that:
$$|\delta\vphi_\lambda^3(\vu)|,|\delta\vphi_\lambda^2(\vu)\het(\vu)|,|\delta\vphi_\lambda(\vu)\het^2(\vu)|\leq \frac{\delta^3K}{|\cosh^5\vu|}.$$
Moreover, we also have that:
$$|\vxi_\lambda\vxib_\lambda|\leq\frac{\delta^4K}{|\cosh^6\vu|}\leq\frac{\delta^3K}{|\cosh^5\vu|},$$
and then \eqref{fitaterme3intermig} yields:
\begin{equation*}
 \left| D_\vxi\mathcal{R}_1^2(\vphi_\lambda)(\vu)\right|\leq\frac{K|a_1(\vu)|}{|-1+\het^2(\vu)|}\frac{\delta^3}{|\cosh^5\vu|}\leq\frac{K\delta^3|a_1(\vu)|}{|\cosh^3\vu|}.
\end{equation*}
Finally we just need to note that $|a_1(\vu)|\leq K/\delta$ to obtain bound \eqref{fitaterme3}.

Our last claim is:
\begin{equation}\label{fitaterme4}
\left|D_\vxi\mathcal{R}_1^3(\vphi_\lambda)(\vu)\right|\leq\frac{ K\delta^2}{|\cosh^3\vu|}.
\end{equation}
This is quite straightforward to prove, using inequalities \eqref{fitavphilambda} and \eqref{fitadivisio}, Lemma \ref{lemafitaDFiDH} and that $|a_1(\vu)|\leq K/\delta$.

In conclusion, from bounds \eqref{fitatildeF}, \eqref{fitaterme3} and \eqref{fitaterme4} we have:
$$|D_\vxi\mathcal{R}_1(\vu)|\leq K\left(\frac{\delta }{|\cosh^2\vu|}+\frac{\delta^2}{|\cosh^3\vu|}\right)\leq \frac{ K\delta}{|\cosh^2\vu|},$$
and thus \eqref{resultat} is proved.
\end{proof}

\begin{lem}\label{lemainvertirGsplit}
 The operator $\mathcal{G}:\bssplit\times\bssplit\rightarrow\bssplit\times\bssplit$ is well defined, and for $\phi\in\bssplit\times\bssplit$:
 $$\|\mathcal{G}(\phi)\|_{\textup{spl},\times}\leq\frac{K\|\phi\|_{\textup{spl},\times}}{\log(1/\delta)}.$$
\end{lem}
\begin{proof}
 Again, we will bound just the first component:
$$|\mathcal{G}_1(\phi)(it)|=\left|m_1(it)\int_{t_+}^t m_1^{-1}(i\vw)\pi^1(\mathcal{B}(i\vw)(\phi(i\vw)))d\vw\right|.$$
Recalling the asymptotic formula \eqref{asyformm1} for $m_1(it)$ it is clear that:
$$|m_1(it)|\leq K\cos^dte^{\alpha  t/\delta},\qquad |m_1^{-1}(i\vw)|\leq K\cos^{-d}\vw e^{-\alpha  \vw/\delta}.$$
Using these bounds and Lemma \ref{lemabij} we can bound $\mathcal{G}_1(\phi)(it)$:
$$|\mathcal{G}_1(\phi)(it)|\leq K\cos^dt e^{\alpha  t/\delta}\int_t^{t_+}\cos^{-d}\vw e^{-\alpha  \vw/\delta}\frac{ K\delta}{\cos^2\vw}|\phi(i\vw)|d\vw.$$
Then, since $\phi\in\bssplit\times\bssplit$, recalling definition \eqref{defnormsplit} of the norm $\|.\|_{\textup{spl},\times}$ we have:
$$|\phi(i\vw)|\leq\|\phi\|_{\textup{spl},\times}\cos^d\vw e^{-\alpha (\pi/2-|\vw|)/\delta}$$
and therefore:
$$|\mathcal{G}_1(\phi)(it)|\leq K\delta\cos^dt e^{\alpha  t/\delta}e^{-\frac{\alpha \pi}{2\delta}}\|\phi\|_{\textup{spl},\times}\int_t^{t_+}e^{-\alpha (\vw-|\vw|)/\delta}\frac{1}{\cos^2\vw}d\vw.$$
It is not difficult to check that for $t\in[t_-,t_+]$, there exists a constant $C$ independent of $\delta$ such that:
$$e^{\alpha  t/\delta}\int_t^{t_+}e^{-\alpha (\vw-|\vw|)/\delta}\frac{1}{\cos^2\vw}d\vw\leq Ce^{\alpha |t|/\delta}\frac{1}{\dist\delta\log(1/\delta)},$$
and then we obtain the desired bound:
$$\|\mathcal{G}_1(\phi)\|_{\textup{spl},\times}\leq \frac{K\|\phi\|_{\textup{spl},\times}}{\log(1/\delta)}.$$
\end{proof}

\begin{proof}[End of the proof of Proposition \ref{propasymsplit}]
From Lemma \ref{lemaprincipalsplit} we can write:
$$(Id-\mathcal{G})\Delta\vphi_1=\mathcal{I}(c_1-c_1^0,c_2-c_2^0)+\mathcal{G}(\Delta\vphi_0).$$
We note that for $\delta>0$ $\Delta\vphi_1\in\bssplit\times\bssplit$, although \textit{a priori} its norm is exponentially large with respect to $\delta$. Indeed, we have $\Delta\vphi_1=\Delta\vphi-\Delta\vphi_0$, and it is clear by Lemma \ref{lemaasynormadeltavphi0} that $\Delta\vphi_0\in\bssplit\times\bssplit$. Moreover, we have:
\begin{eqnarray}\label{normagran}
 |\Delta\vphi(it)\cos^{-\coef}te^{\alpha(\pi/2-|t|)/\delta}|&\leq&(\|\vphi^\uns\|_{\textup{out},\times}^\uns+\|\vphi^\sta\|_{\textup{out},\times}^\sta)|\het(it)-1|^3|\cos^{-\coef}t|e^{\alpha(\pi/2-|t|)/\delta}\nonumber\\
&\leq& K\delta^2|\cos^{3-\coef}t|e^{\alpha(\pi/2-|t|)/\delta}\leq\delta^{-(\coef+1)}e^{\frac{\alpha\pi}{2\delta}}K<\infty,
\end{eqnarray}
and thus it is clear that $\Delta\vphi\in\bssplit\times\bssplit$, and hence $\Delta\vphi_1\in\bssplit\times\bssplit$. Since from Lemma \ref{lemainvertirGsplit} we know that $\|\mathcal{G}\|<1$ for $\delta$ small enough, the operator $Id-\mathcal{G}$ is invertible in $\bssplit\times\bssplit$. Therefore we can write:
$$\Delta\vphi_1=(Id-\mathcal{G})^{-1}\left[\mathcal{I}(c_1-c_1^0,c_2-c_2^0)+\mathcal{G}(\Delta\vphi_0)\right],$$
and consequently we have:
\begin{eqnarray}\label{fitanormadeltavphi1}
 \|\Delta\vphi_1\|_{\textup{spl},\times}&\leq&\|Id-\mathcal{G}\|_{\textup{spl},\times}^{-1}\left[\|\mathcal{I}(c_1-c_1^0,c_2-c_2^0)\|_{\textup{spl},\times}+\|\mathcal{G}(\Delta\vphi_0)\|_{\textup{spl},\times}\right]\nonumber\\
&\leq&K\left(\|\mathcal{I}(c_1-c_1^0,c_2-c_2^0)\|_{\textup{spl},\times}+\|\mathcal{G}(\Delta\vphi_0)\|_{\textup{spl},\times}\right).
\end{eqnarray}
Now, from \eqref{fitanormadeltavphi1} we will be able to improve bound \eqref{normagran}, realizing that in fact it is not exponentially large with respect to $\delta$. On one hand, using first Lemma \ref{lemafitaIsplit} and after \ref{lemaprincipalsplit}, we have:
\begin{equation}
 \|\mathcal{I}(c_1-c_1^0,c_2-c_2^0)\|_{\textup{spl},\times}\leq K(|c_1-c_1^0|+|c_2-c_2^0|)e^{\frac{\alpha\pi}{2\delta}}\leq\frac{K}{\delta^{\coef+1}\log(1/\delta)}.\nonumber
\end{equation}
Then, from Lemma \ref{lemaasynormadeltavphi0} it is clear that, if $\|\Delta\vphi_0\|_{\textup{spl},\times}\neq0$ (which is equivalent to $C_{\textup{in}}\neq0$), we have:
\begin{equation}\label{fitanormaIsplit}
  \|\mathcal{I}(c_1-c_1^0,c_2-c_2^0)\|_{\textup{spl},\times}\leq\frac{K\|\Delta\vphi_0\|_{\textup{spl},\times}}{\log(1/\delta)}.
\end{equation}
On the other hand, from Lemma \ref{lemainvertirGsplit} we have:
\begin{equation}\label{fitaGDeltavphi0}
 \|\mathcal{G}(\Delta\vphi_0)\|_{\textup{spl},\times}\leq\frac{K\|\Delta\vphi_0\|_{\textup{spl},\times}}{\log(1/\delta)}.
\end{equation}

Substituting \eqref{fitanormaIsplit} and \eqref{fitaGDeltavphi0} in \eqref{fitanormadeltavphi1} we obtain the desired bound:
$$\|\Delta\vphi_1\|_{\textup{spl},\times}\leq\frac{K\|\Delta\vphi_0\|_{\textup{spl},\times}}{\log(1/\delta)}.$$
\end{proof}

\begin{proof}[End of the Proof of Theorem \ref{theoremsplit}]
From Proposition \ref{propasymsplit}, we know that $\Delta\vphi=\Delta\vphi_0+\Delta\vphi_1$, with:
$$|\Delta\vphi_1(it)|\leq K\frac{\|\Delta\vphi_0\|_{\textup{spl},\times}}{\log(1/\delta)}e^{-\alpha (\pi/2-|t|)/\delta}|\cos^\coef t|,$$
and hence by Lemma \ref{lemaasynormadeltavphi0} we obtain:
$$|\Delta\vphi_1(it)|\leq \frac{K}{\delta^{\coef+1}\log(1/\delta)}e^{\frac{\pi}{2}(c+\alpha h_0)}\left(\left|C_{\textup{in}}\right|+|\overline{C_{\textup{in}}}|\right)e^{-\alpha (\pi/2-|t|)/\delta}|\cos^\coef t|.$$
For $t=0$ this formula gives the bound:
\begin{equation}\label{final1}
 |\Delta\vphi_1(0)|\leq \frac{K}{\delta^{\coef+1}\log(1/\delta)}e^{\frac{\pi}{2}(c+\alpha h_0)}\left(\left|C_{\textup{in}}\right|+|\overline{C_{\textup{in}}}|\right)e^{-\frac{\alpha \pi}{2\delta}}.
\end{equation}
Moreover, by definition of $\Delta\vphi_0$ it is clear that $\Delta\vphi_0(0)=(c_1^0,c_2^0)$. Then by definition \eqref{defc10c20} of $c_1^0$ and $c_2^0$, Lemma \ref{lemam1invers} and formula \eqref{cosa2} we obtain:
\begin{equation}\label{final2}
c_1^0=\frac{1}{\delta^{\coef+1}}e^{-\frac{\alpha \pi}{2\delta}}C_{\textup{in}}e^{\frac{\pi}{2}(c+\alpha h_0)-i\left(\frac{\param\pi}{2}+\frac{\alpha h_0}{2}+(c+\alpha h_0)\log\delta\right)}\left(1+O\left(\frac{1}{\log(1/\delta)}\right)\right),
\end{equation}
and $c_2^0=\overline{c_1^0}$. Finally, we just need to realize that, since $\alpha=\alpha_0+\alpha_1\delta\param+O(\delta^2)$, we have:
$$e^{-\frac{\alpha \pi}{2\delta}}=e^{-\frac{\alpha_0\pi}{2\delta}-\frac{\alpha_1\param\pi}{2}}(1+O(\delta)),$$
and:
$$e^{\frac{\pi}{2}(c+\alpha h_0)-i\left(\frac{\param\pi}{2}+\frac{\alpha h_0}{2}+(c+\alpha h_0)\log\delta\right)}=e^{\frac{\pi}{2}(c+\alpha_0h_0)-i\left(\frac{\param\pi}{2}+\frac{\alpha_0h_0}{2}+(c+\alpha_0h_0)\log\delta\right)}(1+O(\delta)),$$
so that \eqref{final2} becomes:
\begin{equation}\label{final3}
c_1^0=\frac{1}{\delta^{\coef+1}}e^{-\frac{\alpha_0\pi}{2\delta}}C_{\textup{in}}e^{\frac{\pi}{2}(c+\alpha_0h_0-\param\alpha_1)-i\left(\frac{\param\pi}{2}+\frac{\alpha_0h_0}{2}+(c+\alpha_0h_0)\log\delta\right)}\left(1+O\left(\frac{1}{\log(1/\delta)}\right)\right).
\end{equation}
Then, from \eqref{final1} and \eqref{final3} and the fact that $\Delta\vphi(0)=(c_1^0,c_2^0)+\Delta\vphi_1(0)$ we obtain:
$$\Delta\vphi(0)=\frac{1}{\delta^{\coef+1}}e^{-\frac{\alpha_0\pi}{2\delta}}e^{\frac{\pi}{2}(c+\alpha_0h_0-\param\alpha_1)}\left(\left(\begin{array}{c}C_{\textup{in}}e^{-i\left(\frac{\param\pi}{2}+\frac{\alpha_0h_0}{2}+(c+\alpha_0h_0)\log\delta\right)}\\\overline{C_{\textup{in}}}e^{i\left(\frac{\param\pi}{2}+\frac{\alpha_0h_0}{2}+(c+\alpha_0h_0)\log\delta\right)}\end{array}\right)+O\left(\frac{1}{\log(1/\delta)}\right)\right).$$
and the theorem is proved.
\end{proof}
\section*{Aknowledgements}
T.M. Seara and O. Castej\'{o}n have been partially supported by MICINN-FEDER grant MTM2009-06973 and CUR-DIUE grant 2009SGR859. In addition,
I. Baldom\'{a} has been supported by the Spanish Grant MEC-FEDER MTM2006-05849/Consolider,
the Spanish Grant MTM2010-16425 and the Catalan SGR grant 2009SGR859.
The research of O. Castej\'{o}n has been supported by the Catalan PhD grant FI-AGAUR.

\bibliographystyle{model1-num-names}
\bibliography{hopfzero}
\end{document}